\documentclass{amsart}

\usepackage{amscd}
\usepackage{amsmath}
\usepackage{amssymb}
\usepackage{amsthm}
\usepackage{epsf}
\usepackage{latexsym}
\usepackage{verbatim}
\usepackage[all, cmtip]{xy}
\usepackage{tikz}
\usetikzlibrary{positioning}
\usetikzlibrary{matrix}

\tikzstyle{bsq}=[rectangle, draw, thick, minimum width=1cm, minimum height=1cm]
\tikzstyle{bver}=[rectangle, draw, thick, minimum width=1cm, minimum height=2cm]
\tikzstyle{bhor}=[rectangle, draw, thick, minimum width=2cm, minimum height=1cm]

\newtheorem{theorem}{Theorem}[section]
\newtheorem{definition}[theorem]{Definition}
\newtheorem{lemma}[theorem]{Lemma}

\newtheorem{corollary}[theorem]{Corollary}
\newtheorem{proposition}[theorem]{Proposition}
\newtheorem{question}[theorem]{Question}
\newtheorem{varexample}[theorem]{Example}
\theoremstyle{definition}
\newtheorem{remark}[theorem]{Remark}
\newtheorem*{BNthm}{Brill-Noether Theorem}
\newtheorem*{GPthm}{Gieseker-Petri Theorem}
\newtheorem*{TropicalRR}{Tropical Riemann-Roch Theorem}
\newtheorem*{MetrizedRR}{Riemann-Roch for Metrized Complexes}
\newtheorem*{TropicalAJ}{Tropical Abel-Jacobi Theorem}
\newtheorem*{TropicalClifford}{Tropical Clifford's Theorem}
\newtheorem*{SpecializationThm}{Specialization Theorem}
\newtheorem*{MRC}{Maximal Rank Conjecture}

\newtheorem*{WtedSpecializationThm}{Weighted Specialization Theorem}
\newtheorem*{MetrizedSpecialization}{Specialization Theorem for Metrized Complexes}

\newcommand{\Spec}{\mathrm{Spec}\,}

\newcommand{\CC}{\mathbb{C}}
\newcommand{\PP}{\mathbb{P}}
\newcommand{\QQ}{\mathbb{Q}}
\newcommand{\RR}{\mathbb{R}}
\newcommand{\ZZ}{\mathbb{Z}}

\newcommand{\cC}{\mathcal{C}}
\newcommand{\cD}{\mathcal{D}}
\newcommand{\cG}{\mathcal{G}}
\newcommand{\cL}{\mathcal{L}}

\newcommand{\cW}{\mathcal{W}}

\newcommand{\cF}{\mathcal{F}}

\newcommand{\fC}{\mathfrak{C}}

\newcommand{\f}{\mathfrak f}
\newcommand{\g}{\mathfrak g}

\newcommand{\ord}{\operatorname{ord}}
\newcommand{\Trop}{\operatorname{Trop}}
\newcommand{\trop}{\operatorname{trop}}
\newcommand{\ddiv}{\operatorname{div}}
\newcommand{\Div}{\operatorname{Div}}
\newcommand{\PL}{\operatorname{PL}}
\newcommand{\Jac}{\operatorname{Jac}}
\newcommand{\val}{\operatorname{val}}
\newcommand{\Pic}{\operatorname{Pic}}
\newcommand{\outdeg}{\mathrm{outdeg}}
\newcommand{\indeg}{\mathrm{indeg}}
\newcommand{\Cliff}{\mathrm{Cliff}}
\newcommand{\br}{\mathrm{br}}

\newcommand{\an}{\mathrm{an}}

\newenvironment{example}{\begin{varexample}
\begin{normalfont}}{\end{normalfont}
\end{varexample}}
\begin{document}
\title[Degenerations of Linear Series from the Tropical Point of View]{Degeneration of Linear Series From the Tropical Point of View and Applications}
\author{Matthew Baker}
\author{David Jensen}
\date{\today}
\bibliographystyle{alpha}

\thanks{The authors would like to thank Omid Amini and Sam Payne for enlightening discussions, and Eric Katz and Joe Rabinoff for helpful feedback on our summary of \cite{KRZB15}.  They also thank Spencer Backman, Dustin Cartwright, Melody Chan, Yoav Len, and Sam Payne for their comments on an early draft of this paper.}

\maketitle

\begin{abstract}
We discuss linear series on tropical curves and their relation to classical algebraic geometry, describe the main techniques of the subject, and survey some of the recent major developments in the field, with an emphasis on applications to problems in Brill-Noether theory and arithmetic geometry.
\end{abstract}

\tableofcontents

\section{Introduction}

Algebraic curves play a central role in the field of algebraic geometry.  Over the past century, curves have been the focus of a significant amount of research, and despite being some of the most well-understood algebraic varieties, there are still many important open questions.  The goal of classical Brill-Noether theory is to study the geometry of a curve $C$ by examining all of its maps to projective space, or equivalently the existence and behavior of all line bundles on $C$.  Thus, we have classical results such as Max Noether's Theorem \cite[p 117]{ACGH} and the Enriques-Babbage Theorem \cite{Babbage39} that relate the presence of linear series on a curve to its geometric properties.  A major change in perspective occurred during the twentieth century, as the field shifted from studying \emph{fixed} to \emph{general} curves -- that is, general points in the moduli space of curves $M_g$.  Many of the major results in the field, such as the Brill-Noether \cite{GriffithsHarris80} and Gieseker-Petri \cite{Gieseker82} Theorems, remained open for nearly a century as they awaited this new point of view.

\medskip

A major milestone in the geometry of general curves was the development of limit linear series by Eisenbud and Harris \cite{EisenbudHarris86}.  This theory allows one to study linear series on general curves by studying one-parameter degenerations where the central fiber is a special kind of singular curve, known as a curve of compact type.  One property of curves of compact type is that if they have positive genus then they must have components of positive genus.  Shortly after the development of limit linear series, many researchers became interested in a different type of nodal curve, which have only rational components, and where the interesting geometric data is encoded by the combinatorics of how the components meet each other.  Early examples using so-called graph curves to establish properties of general curves include Bayer and Eisenbud's work on Green's Conjecture for the general curve \cite{BayerEisenbud91}, and the Ciliberto-Harris-Miranda result on the surjectivity of the Wahl map \cite{CHM88}.

\medskip

Much like the theory of limit linear series does for curves of compact type, the recent development of tropical Brill-Noether theory provides a systematic approach to this kind of degeneration argument \cite{BakerNorine07, Baker08, AminiBaker12}.  A major goal of this survey is to introduce the basic definitions and techniques of this theory, as well as describing some recent applications to the geometry of general curves and the behavior of Weierstrass points in degenerating families.
Degeneration arguments also play a major role in arithmetic geometry, and we also survey how linear series on tropical curves can be used to study rational points on curves.

\medskip

Here are just a few of the interesting theorems which have been proved in recent years with the aid of the theory of linear series on tropical curves.

\medskip

1. {\em The Maximal Rank Conjecture for quadrics.} In \cite{MRC}, Jensen and Payne prove that for fixed $g,r,$ and $d$, if $C$ is a general curve of genus $g$ and $V \subset \cL(D)$ is a general linear series on $C$ of rank $r$ and degree $d$, then the multiplication map $\mu_2: {\rm Sym}^2 V \rightarrow \cL(2D)$
is either injective or surjective.

\smallskip

2. {\em Uniform boundedness for rational points of curves of small Mordell-Weil rank.} In \cite{KRZB15}, Katz, Rabinoff, and Zureick--Brown prove that if $C/\QQ$ is a curve of genus $g$ with Mordell-Weil rank at most $g-3$, then $\# C(\QQ) \leq 76g^2 -82g + 22.$  This is the first such bound depending only on the genus of $C$.

\smallskip

3. {\em Non-Archimedean equidistribution of Weierstrass points.} In \cite{Amini14}, Amini proves that if $C$ is an algebraic curve over $\CC ((t))$ and $L$ is an ample line bundle on $C$, then the Weierstrass points of $L^{\otimes n}$ become equidistributed with respect to the Zhang measure on the
Berkovich analytic space $C^{\an}$ as $n$ goes to infinity.  This gives precise asymptotic information on the limiting behavior of Weierstrass points in degenerating one-parameter families.

\smallskip

4. {\em Mn{\"e}v universality for the lifting problem for divisors on graphs.} In \cite{Cartwright15}, Cartwright shows that if $X$ is a scheme of finite type over $\Spec \ZZ$, there exists a graph $G$ and a rank 2 divisor $D_0$ on $G$ such that, for any infinite field $k$, there are a curve $C$ over $k((t))$ and a rank 2 divisor $D$ on $C$ tropicalizing to $G$ and $D_0$, respectively, if and only if $X$ has a $k$-point.

\medskip

We will discuss the proofs of these and other results after going through the foundations of the basic theory.
To accommodate readers with various interests, this survey is divided into three parts.  The first part covers the basics of tropical Brill-Noether theory, with an emphasis on combinatorial aspects and the relation to classical algebraic geometry.  The second part covers more advanced topics, including the  nonarchimedean Berkovich space perspective, tropical moduli spaces, and the theory of metrized complexes.  Each of these topics is an important part of the theory, but is not strictly necessary for many of the applications discussed in Part \ref{Part:Applications}, and the casual reader may wish to skip Part \ref{Part:Advanced} on the first pass.  The final part covers applications of tropical Brill-Noether theory to problems in algebraic and arithmetic geometry.  For the most part, the sections in Part \ref{Part:Applications} are largely independent of each other and can be read in any order.  Aside from a few technicalities, the reader can expect to follow the applications in Sections \ref{Section:Applications}, \ref{Section:Lifting} and \ref{Section:Arithmetic},
as well as most of Section \ref{Section:Weierstrass}, without reading Part \ref{Part:Advanced}.

\part{Introductory Topics}
\label{Part:Intro}

\section{Jacobians of Finite Graphs}
\label{Section:Finite}

\subsection{Degeneration of line bundles in one-parameter families}
A recurring theme in the theory of linear series on curves is that it is very important to understand the behavior of line bundles on generically smooth one-parameter families of curves.  One of the key facts about such families is the semistable reduction theorem, which asserts that after a finite base change, one can guarantee that the singularities of the family are ``as nice as possible'', i.e., the total space is regular and the fibers are reduced and have only nodal singularities (see \cite[Chapter 3C]{HarrisMorrison98} or \cite[\S{10.4}]{LiuBook}).  The questions we want to answer about such families are all local on the base, so it is convenient to consider the following setup.  Let $R$ be a discrete valuation ring with field of fractions $K$ and algebraically closed residue field $\kappa$, let $C$ be a smooth proper and geometrically integral curve over $K$, and let $\mathcal{C}$ be a {\em regular strongly semistable model} for $C$, that is, a proper flat $R$-scheme with general fiber $C$ satisfying:
\begin{enumerate}
\item The total space $\mathcal{C}$ is regular.
\item The central fiber $C_0$ of $\mathcal{C}$ is \emph{strongly semistable}, i.e., the irreducible components of $C_0$ are all smooth and $C_0$ has only nodes as singularities.\footnote{We say $C_0$ is \emph{semistable} if it is reduced and has only nodes as singularities.  The term \emph{strongly semistable} is not completely standard.}
\end{enumerate}
By the semistable reduction theorem, a regular strongly semistable model for $C$ always exists after passing to a finite extension of $K$.

\medskip

Let $L$ be a line bundle on the general fiber $C$.  Because the total space $\mathcal{C}$ is regular, there exists an extension $\mathcal{L}$ of $L$ to the family $\mathcal{C}$.  One can easily see, however, that this extension is not unique -- one can obtain other extensions by twisting by components of the central fiber.  More concretely, if $\mathcal{L}$ is an extension of $L$ and $Y \subset C_0$ is an irreducible component of the central fiber, then $\mathcal{L} (Y) = \mathcal{L} \otimes \mathcal{O}_{\mathcal{C}} (Y)$ is also an extension of $L$.

\medskip

To understand the effect of the twisting operation, we consider the \emph{dual graph} of the central fiber $C_0$.  One constructs this graph by first assigning a vertex $v_Z$ to each irreducible component $Z$ of $C_0$, and then drawing an edge between two vertices for every node at which the corresponding components intersect.

\begin{example}
\label{Ex:CompleteGraph}
If $C_0$ is a union of $m$ general lines in $\PP^2$, its dual graph is the complete graph on $m$ vertices.  Indeed, every pair of lines meets in one point, so between every pair of vertices in the dual graph, there must be an edge.
\end{example}

\begin{example}
\label{Ex:CompleteBipartite}
If $C_0 \subset \PP^1 \times \PP^1$ is a union of $a$ lines in one ruling and $b$ lines in the other ruling, then the dual graph is the complete bipartite graph $K_{a,b}$.  This is because a pair of lines in the same ruling do not intersect, whereas a pair of lines in opposite rulings intersect in one point.
\end{example}

\begin{example}
\label{Ex:PetersenGraph}
Let $C_0$ be the union of the $-1$ curves on a del Pezzo surface of degree 5.  In this case, the dual graph of $C_0$ is the well-known {\em Petersen graph}.
\end{example}

\subsection{Divisors and linear equivalence on graphs}
In this paper, by a {\em graph} we will always mean a finite connected graph which is allowed to have multiple edges between pairs of vertices but is not allowed to have any loop edges.  Given a graph $G$, we write $\Div (G)$ for the free abelian group on the vertices of $G$.  An element $D$ of $\Div (G)$ is called a \emph{divisor} on $G$, and is written as a formal sum
$$ D = \sum_{v \in V(G)} D(v) v , $$
where the coefficients $D(v)$ are integers.  The \emph{degree} of a divisor $D$ is defined to be the sum
$$ \deg (D) = \sum_{v \in V(G)} D(v) . $$

Returning now to our family of curves, let us fix for a moment an extension $\mathcal{L}$ of our line bundle $L$.  We define a corresponding divisor ${\rm mdeg}(\mathcal{L})$ on $G$, called the {\em multi-degree} of $\mathcal{L}$, by the formula
\[
{\rm mdeg}(\mathcal{L}) = \sum_Z \left( \deg \mathcal{L} \vert_Z \right) v_Z,
\]
where the sum is over all irreducible components $Z$ of $C_0$.
The quantity $\deg \mathcal{L} \vert_Z$ can also be interpreted as the intersection multiplicity of $\mathcal{L}$ with $Z$ considered as a (vertical) divisor on the surface $\mathcal{C}$.  Note that the degree of ${\rm mdeg}(\mathcal{L})$ is equal to ${\rm deg}(L)$.

\medskip

We now ask ourselves how ${\rm mdeg}(\mathcal{L})$ changes if we replace $\mathcal{L}$ by a different extension.
Since any two such extensions differ by a sequence of twists by components of the central fiber, it suffices to study the effect of twisting by one such component $Y$.  Given a vertex $v$ of a graph, let $\val (v)$ denote its valence.
As the central fiber $C_0$ is a principal divisor on $\cC$, we have
\[
{\rm mdeg}(\mathcal{L}(Y)) = {\rm mdeg}(\mathcal{L}) + \sum_Z \left( Y \cdot Z \right) v_Z
\]
with
\begin{displaymath}
Y \cdot Z = \left\{ \begin{array}{ll}
- \val(v_Y) & \textrm{if $Z = Y$}\\
\vert Z \cap Y \vert & \textrm{if $Z \neq Y$}
\end{array} \right\} .
\end{displaymath}


\medskip

The corresponding operation on the dual graph is known as a \emph{chip-firing move}. This is because we may think of a divisor on the graph as a configuration of chips (and anti-chips) on the vertices.  In this language, the effect of a chip-firing move is that a vertex $v$ ``fires'' one chip along each of the edges emanating from $v$.  This decreases the number of chips at $v$ by the valence of $v$, and increases the number of chips at each of the neighbors $w$ of $v$ by the number of edges between $v$ and $w$.

\medskip

Motivated by these observations, we say that two divisors $D$ and $D'$ on a graph $G$ are \emph{equivalent}, and we write $D \sim D'$, if one can be obtained from the other by a sequence of chip-firing moves.
We define the {\em Picard group} $\Pic(G)$ of $G$ to be the group of divisors on $G$ modulo equivalence.
Note that the degree of a divisor is invariant under chip-firing moves, so there is a well defined homomorphism
$$ \deg : \Pic (G) \to \mathbb{Z} . $$
We will refer to the kernel $\Pic^0(G)$ of this map as the \emph{Jacobian} $\Jac (G)$ of the graph $G$.
This finite abelian group goes by many different names in the mathematical literature -- in combinatorics, it is commonly referred to as the \emph{sandpile group} or the \emph{critical group} of $G$ (see for example \cite{Biggs99,Dhar-et-al,LevinePropp,Perkinson-et-al}).

\begin{remark}
\label{Remark:CombLit}
There is a tremendous amount of combinatorial literature concerning Jacobians of graphs.  As our focus is on applications in algebraic geometry, however, we will not go into many details here -- we refer the reader to \cite{BakerNorine07,HLMPPW08} and the references therein.
We cannot resist mentioning one remarkable fact, however: the cardinality of $\Jac(G)$ is the number of {\em spanning trees} in $G$.  (This is actually a disguised form of Kirchhoff's celebrated {\em Matrix-Tree Theorem}.)
\end{remark}

Our discussion shows that for any two extensions of the line bundle $L$, the corresponding multidegrees are equivalent divisors on the dual graph $G$.  There is therefore a well-defined degree-preserving map
$$ \Trop : \Pic (C) \to \Pic (G) , $$
which we refer to as the \emph{specialization} or \emph{tropicalization} map.

\medskip

For later reference, we note that there is a natural homomorphism $\rho : \Div(C) \to \Div(G)$ defined by setting $\rho(D) = {\rm mdeg}({\mathcal D})$, where ${\mathcal D}$ is the Zariski closure of $D$ in ${\mathcal C}$.
The map $\rho$ takes principal (resp. effective) divisors to principal (resp. effective) divisors, and the map $\Jac(C) \to \Jac(G)$ induced by $\rho$ coincides with $\Trop$ (cf. \cite[\S{2.1}]{Baker08}).

\medskip

It is useful to reformulate the definition of equivalence of divisors on $G$ as follows.
For any function $f: V(G) \to \mathbb{Z}$, we define
$$ \ord_v (f) : = \sum_{w \text{ adjacent to } v} f(v) - f(w) $$
and
$$ \ddiv (f) := \sum_{v \in V(G)} \ord_v (f) v . $$
Divisors of the form $\ddiv (f)$ are known as \emph{principal} divisors, and two divisors are equivalent if and only if their difference is principal.
The reason is that the divisor $\ddiv (f)$ can be obtained, starting with the 0 divisor, by firing each vertex $v$ exactly $f(v)$ times.

\subsection{Limit linear series and N{\'e}ron models}

The Eisenbud-Harris theory of limit linear series focuses on the case where the dual graph of $C_0$ is a tree.  In this case, the curve $C_0$ is said to be of \emph{compact type}.\footnote{The reason for the name {\em compact type} is that the Jacobian of a nodal curve $C_0$ is an extension of the Jacobian of the normalization of $C_0$ by a torus of dimension equal to the first Betti number of the dual graph.  It follows that the Jacobian of $C_0$ is compact if and only if its dual graph is a tree.}  Although they would not have stated it this way, a key insight of the Eisenbud-Harris theory is that the Jacobian of a tree is trivial.  Given a line bundle $\mathcal{L}$ of degree $d$ on $\mathcal{C}$, if the dual graph of $C_0$ is a tree then one can repeatedly twist to obtain a line bundle with any degree distribution summing to $d$ on the components of the central fiber.  In particular, given a component $Y \subset C_0$, there exists a twist $\mathcal{L}_Y$ such that
\begin{displaymath}
\deg \mathcal{L}_Y \vert_Z = \left\{ \begin{array}{ll}
d & \textrm{if $Z = Y$}\\
0 & \textrm{if $Z \neq Y$}
\end{array} \right\} .
\end{displaymath}
This observation is the jumping-off point for the basic theory of limit linear series.

\medskip

At the other end of the spectrum is the {\em maximally degenerate} case where all of the components of $C_0$ have genus zero.  In this case, the first Betti number of the dual graph $G$ (which we refer to from now on as the \emph{genus} of $G$) is equal to the geometric genus of the curve and
essentially all of the interesting information about degenerations of line bundles is combinatorial.  At its core, tropical Brill-Noether theory studies the behavior of line bundles on the curve $C$ using the combinatorics of their specializations to the graph $G$.

\medskip

This discussion can also be understood in the context of N{\'e}ron models.  An important theorem of Raynaud \cite{Raynaud70} asserts (in the language of this paper) that the group $\Phi$ of connected components of the special fiber $\bar{ \mathcal J}$ of the N{\'e}ron model of $J = \Jac(C)$ is canonically isomorphic to $\Jac(G)$, where $G$ is the dual graph of the special fiber of  $\mathcal{C}$.   This result can be summarized by the commutativity of the following diagram:
  \begin{equation} \label{eq:discrete.retract.ses}
    \xymatrix @R=.25in{
      0 \ar[r] &
      {{\rm Prin}(C)} \ar[r] \ar[d]^{\rm \rho} &
      {{\rm Div}^0(C)} \ar[r] \ar[d]^{\rm \rho} &
      {J(K)} \ar[r] \ar[d] & 0 \\
      0 \ar[r] &
      {{\rm Prin}(G)} \ar[r] &
      {{\rm Div}^0(G)} \ar[r] &
      {\Jac(G)=\Phi} \ar[r] & 0
    }\end{equation}
where the right vertical arrow is the canonical quotient map
$$ J(K)\to J(K)/{\mathcal J}^0(R)=\Phi . $$

\medskip

The tropical approach to Brill-Noether theory and the approach via the theory of limit linear series are in some sense orthogonal.  The former utilizes the component group $\Phi$, or its analytic counterpart, the tropical Jacobian, whereas classical limit linear series are defined only when $\Phi$ is trivial.  On the other hand,
limit linear series involve computations in the compact part of $\bar{\mathcal J}$, which is trivial in the maximally degenerate case.

\medskip

Recent developments have led to a sort of hybrid of tropical and limit linear series that can be used to study degenerations of line bundles to arbitrary nodal curves.  We will discuss these ideas in Section~\ref{Section:MetrizedComplexes}.

\section{Jacobians of Metric Graphs}
\label{Section:Metric}

\subsection{Behavior of dual graphs under base change}
The field of fractions of a DVR is never algebraically closed.  For many applications, we will be interested in $\Pic (C_{\overline{K}})$ rather than $\Pic (C_K)$, and we must therefore study the behavior of the specialization map under base change.

\medskip

Let $K'$ be a finite extension of $K$, let $R'$ be the valuation ring of $K'$, and let $C_{K'} = C \times_K K'$.  An important issue is that the new total space $\mathcal{C}_{K'} = \mathcal{C} \times_K K'$ may not be regular; it can pick up singularities at the nodes of the central fiber.
More specifically, if a point $z$ on ${\mathcal C}$ corresponding to a node of the central fiber has a local analytic equation of the form $xy = \pi$, where $\pi$ is a uniformizer for $R$, then
a local analytic equation for $z$ over $R'$ will be $xy=(\pi')^e$, where $\pi'$ is a uniformizer for $R'$ and $e$ is the ramification index of the extension $K' / K$.
A standard computation shows that we can resolve such a singularity by a chain of $e-1$ blowups.  Repeating this procedure for each singular point of the special
fiber, we obtain a regular strongly semistable model $\mathcal{C}'$ for $C_{K'}$.   The dual graph $G'$ of the central fiber of $\mathcal{C}' $ is obtained by subdividing each edge of the original dual graph $G$ $e-1$ times.  In other words, if we assign a length of 1 to each edge of $G$, and a length of $\frac{1}{e}$ to each edge of $G'$, then $G$ and $G'$ are isomorphic as {\em metric graphs}.

\medskip

A {\em metric graph} $\Gamma$ is, roughly speaking, a finite graph $G$ in which each edge $e$ has been identified with a real interval $I_e$ of some specified length $\ell_e > 0$.  The points of $\Gamma$ are the vertices of $G$ together with all points in the relative interiors of the intervals $I_e$.  More precisely, a metric graph is an equivalence class of finite edge-weighted graphs, where two weighted graphs $G$ and $G'$ give rise to the same metric graph if they have a common length-preserving refinement.  A finite weighted graph $G$ representing the equivalence class of $\Gamma$ is called a \emph{model} for $\Gamma$.

\begin{example}
\label{Ex:FamilyOfConics}
Let $K = \mathbb{C}((t))$, and consider the family $\cC : xy = tz^2$ of smooth conics degenerating to a singular conic in $\PP^2$.
The dual graph $G$ of the central fiber consists of two vertices $v,v'$ connected by a single edge.
Let $D = P + Q$ be the divisor on the general fiber cut out by the line $y=x$.  In homogeneous coordinates, we have $P = (\sqrt{t}:\sqrt{t}:1)$ and $Q = (-\sqrt{t}:-\sqrt{t}:1)$.  Although the divisor $D$ itself is $K$-rational, $P$ and $Q$ are not, and both points specialize to the node of the special fiber.
It is not hard to check that $\rho(D) = v + v' \in \Div(G)$.  If $K' = \mathbb{C}((\sqrt{t}))$, then the total space of the family $\mathcal{C} \times_K K'$ has a singularity at the node of the central fiber.  This singularity can be resolved by blowing up the node, and the dual graph $G'$ of the new central fiber is a chain of 3 vertices connected by two edges, with the new vertex $v''$ corresponding to the exceptional divisor of the blowup (see Figure~\ref{Fig:BaseChange}).  The base change $D'$ of $D$ to $K'$ specializes to a sum of two smooth points on the exceptional divisor, and in particular $\rho_{{\mathcal C}'}(D') = 2v''$ is not the image of $\rho_{\mathcal C}(D)$
with respect to the natural inclusion map $\Div(G) \hookrightarrow \Div(G')$.
Note, however, that $\rho_{{\mathcal C}'}(D')$ and $\rho_{\mathcal C}(D)$ are linearly equivalent on $G'$ (this turns out to be a general phenomenon).
\begin{figure}
\begin{tikzpicture}

\draw [ball color=black] (0,0) circle (0.55mm);
\draw [ball color=black] (0,2) circle (0.55mm);
\draw [ball color=black] (2,0) circle (0.55mm);
\draw [ball color=black] (2,1) circle (0.55mm);
\draw [ball color=black] (2,2) circle (0.55mm);
\draw (0,0)--(0,2);
\draw (2,0)--(2,2);
\draw (-0.25,0) node {$v$};
\draw (-0.25,2) node {$v'$};
\draw (2.25,0) node {$v$};
\draw (2.25,1) node {$v''$};
\draw (2.25,2) node {$v'$};

\end{tikzpicture}
\caption{The dual graph of the central fiber in Example \ref{Ex:FamilyOfConics} initially (on the left), and after base change followed by resolution of the singularity (on the right).  If
we give the segment on the left a length of $1$ and each of the segments on the right a length of $1/2$, then both weighted graphs are models for the same underlying metric graph $\Gamma$, which is a closed segment of length 1.}
\label{Fig:BaseChange}
\end{figure}
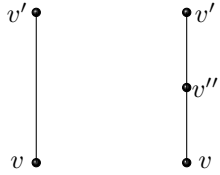
\end{example}

\begin{remark}
A similar example, with a semistable family of curves of genus $3$, is given in \cite[\S{4.4}]{Baker08}.
\end{remark}

There are two possible ways to address the lack of functoriality with respect to base change illustrated in Example~\ref{Ex:FamilyOfConics}.  One is to only consider the induced maps on Picard groups, rather than divisors.  The other is to replace $\rho: \Div(C) \to \Div(G)$ with a map $\Div(C_{\bar K}) \to \Div(\Gamma)$, where $\Gamma$ is the metric graph underlying $G$, defined by first base-changing to an extension $K'$ over which all points in the support of $D$ are rational and then applying $\rho$.  The most natural way to handle these simultaneous base-changes and prove theorems about the resulting map is to work on the Berkovich analytification of $C$; see
Section~\ref{Section:Berkovich} for details.

\subsection{Divisors and linear equivalence on metric graphs}
\label{sec:MetricGraphDivisors}

Let $\Gamma$ be a metric graph.  A \emph{divisor} $D$ on $\Gamma$ is a formal linear combination
$$ D = \sum_{v \in \Gamma} D(v) v $$
with $D(v) \in \ZZ$ for all $v \in \Gamma$ and $D(v) = 0$ for all but finitely many $v \in \Gamma$.
Let ${\rm PL} (\Gamma)$ denote the set of continuous, piecewise linear functions $f: \Gamma \to \mathbb{R}$ with integer slopes.  The \emph{order} $\ord_v (f)$ of a function $f$ at a point $v \in \Gamma$ is the sum of its incoming\footnote{As a caution, some authors use the opposite sign convention.} slopes along the edges containing $v$.  As in the case of finite graphs, we write
$$ \ddiv (f) := \sum_{v \in \Gamma} \ord_v (f) v . $$
A divisor is said to be \emph{principal} if it is of the form $\ddiv (f)$ for some $f \in {\rm PL} (\Gamma)$, and two divisors $D, D'$ are \emph{equivalent} if $D-D'$ is principal.
We let ${\rm Prin}(\Gamma)$ denote the subgroup of $\Div^0(\Gamma)$ (the group of degree-zero divisors on $\Gamma$) consisting of principal divisors.
By analogy with the case of finite graphs, the group
$$ \Jac (\Gamma) := \Div^0(\Gamma) / {\rm Prin}(\Gamma) $$
is called the (tropical) {\em Jacobian} of $\Gamma$.

\begin{example}
\label{Ex:MetricTree}
The Jacobian of a metric tree $\Gamma$ is trivial.  To see this, note that given any two points $P,Q \in \Gamma$, there is a unique path from $P$ to $Q$.  One can construct a continuous, piecewise linear function $f$ that has slope 1 along this path, and has slope 0 everywhere else.  We then see that $\ddiv (f) = Q-P$, so any two points on the tree are equivalent.

\begin{figure}
\begin{tikzpicture}

\draw [ball color=black] (0,0) circle (0.55mm);
\draw [ball color=black] (1,0) circle (0.55mm);
\draw [ball color=black] (1,1) circle (0.55mm);
\draw [ball color=black] (2,0) circle (0.55mm);
\draw [ball color=black] (2,1) circle (0.55mm);
\draw [ball color=black] (2,-1) circle (0.55mm);
\draw [ball color=black] (3,0) circle (0.55mm);
\draw (0,0)--(1,0);
\draw[->] (1,1)--(1,.25);
\draw (1,.25)--(1,0);
\draw[->] (1,0)--(1.75,0);
\draw (1.75,0)--(2,0);
\draw (1,1)--(2,1);
\draw[->] (2,0)--(2,-.75);
\draw (2,-.75)--(2,-1);
\draw (2,0)--(3,0);
\draw (1,1.25) node {$P$};
\draw (2,-1.25) node {$Q$};
\draw (1.15,.5) node {$1$};
\draw (1.5,.15) node {$1$};
\draw (2.15,-.5) node {$1$};
\draw (.5,.15) node {$0$};
\draw (2.5,.15) node {$0$};
\draw (1.5,1.15) node {$0$};

\end{tikzpicture}
\caption{Slopes of a function $f$ on a tree with $\ddiv (f) = Q-P$.}
\label{Fig:Metric Tree}
\end{figure}
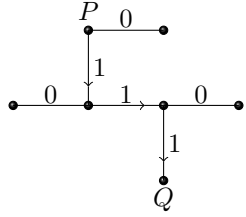
\end{example}

\begin{example}
\label{Ex:Circle}
A circle $\Gamma$ is a torsor for its own Jacobian.  To see this, fix a point $O \in \Gamma$.  Given two points $P,Q \in \Gamma$, there exists a continuous, piecewise linear function $f$ that has slope 1 on the interval from $O$ and $P$, slope $-1$ on the interval from some 4th point $R$ to $Q$, and slope 0 everywhere else.  We then have $\ddiv (f) = O+Q-(P+R)$, so $P-Q \sim R-O$.  It follows that every divisor of degree zero is equivalent to a divisor of the form $R-O$ for some point $R$.  It is not difficult to see that the point $R$ is in fact unique.

\begin{figure}
\begin{tikzpicture}
\draw (0,0) circle (1);
\draw [->] (-1,0) arc [radius=1, start angle=180, end angle = 250];
\draw [->] (0,1) arc [radius=1, start angle=90, end angle = 20];
\draw [ball color=black] (-1,0) circle (0.55mm);
\draw [ball color=black] (1,0) circle (0.55mm);
\draw [ball color=black] (0,-1) circle (0.55mm);
\draw [ball color=black] (0,1) circle (0.55mm);
\draw (-1.25,0) node {$P$};
\draw (1.25,0) node {$Q$};
\draw (0,-1.25) node {$O$};
\draw (0,1.25) node {$R$};
\draw (-.8,-.8) node {1};
\draw (-.8,.8) node {0};
\draw (.8,.8) node {1};
\draw (.8,-.8) node {0};

\end{tikzpicture}
\caption{A function $f$ on the circle with $\ddiv (f) = O+Q-(P+R)$.}
\label{Fig:Circle}
\end{figure}
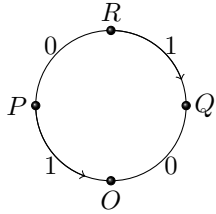
\end{example}

Note that if two divisors are equivalent in the finite graph $G$, then they are also equivalent in the corresponding metric graph $\Gamma$, called the {\em regular realization} of $G$, in which every edge of $G$ is assigned a length of $1$.  It follows that there is a natural inclusion $\iota : \Pic (G) \hookrightarrow \Pic (\Gamma)$.  As we saw above,
the multidegree of a line bundle $L$ on $C_{\overline{K}}$ can be identified with a divisor on some subdivision of $G$, which can in turn be identified with a divisor on $\Gamma$.  
One can show that this yields a well-defined map
$$ \Trop : \Pic (C_{\overline{K}}) \to \Pic (\Gamma ) $$
whose restriction to $\Pic(C)$ coincides with the previously defined map
$$\Trop : \Pic (C) \to \Pic (G)$$
via the inclusion $\iota$.

One can see from this construction that principal divisors on $C$ specialize to principal divisors on $\Gamma$.
More precisely, there is a natural way to define a map
$$ \trop : \overline{K}(C)^* \to \PL ( \Gamma ) $$
on rational functions such that
\[
\Trop(\ddiv(f)) = \ddiv(\trop(f))
\]
for every $f \in \overline{K}(C)^*$.  This is known as the {\em Slope Formula}, cf. Theorem~\ref{thm:SlopeFormula} below.
We refer the reader to \S \ref{Subsection:AnalyticTrop} for a formal definition of the map $\trop$.


As already discussed in the remarks following Example~\ref{Ex:FamilyOfConics},
there is also a natural way to define a map ${\rm Div}(C_{\overline{K}}) \to {\rm Div}(\Gamma)$ which
induces the map $\Trop : \Pic (C_{\overline{K}}) \to \Pic (\Gamma)$.  As with the map $\trop$ on rational functions, the map $\Trop$ on divisors is most conveniently described using Berkovich's theory of non-Archimedean analytic spaces, and we defer a detailed discussion to Section~\ref{Section:Berkovich}.

\subsection{The tropical Abel-Jacobi Map}
\label{Subsection:AbelJacobi}

A \emph{1-form} on a graph $G$ is an element of the real vector space generated by the formal symbols $de$, as $e$ ranges over the oriented edges of $G$, subject to the relations that if $e,e'$ represent the same edge with opposite orientations then $de' = -de$.  After fixing an orientation on $G$, a 1-form $\omega = \sum \omega_e de$ is called \emph{harmonic} if, for all vertices $v$, the sum $\sum \omega_e$ over the outgoing edges at $v$ is equal to $0$.
Denote by $\Omega (G)$ the space of harmonic 1-forms on $G$.  It is well-known that $\Omega (G)$ is a real vector space of dimension equal to the \emph{genus} $g$ of $G$, which can also be defined combinatorially as the number of edges of $G$ minus the number of vertices of $G$ plus one, or topologically as the dimension of $H_1(G,\RR)$.

\medskip

If $G$ and $G'$ are models for the same metric graph $\Gamma$, then $\Omega (G')$ is canonically isomorphic to $\Omega (G)$.  We may therefore define the space $\Omega (\Gamma)$  of harmonic 1-forms on $\Gamma$ as $\Omega (G)$ for any weighted graph model $G$.  (We also define the \emph{genus} of a metric graph $\Gamma$ to be the genus of any model for $\Gamma$.)  Given an isometric path $\gamma : [a,b] \to \Gamma$, any harmonic 1-form $\omega$ on $\Gamma$ pulls back to a classical 1-form on the interval, and we can thus define the integral $\int_{\gamma} \omega$.  Note that the definition of a harmonic $1$-form does not depend on the metric, but the integral $\int_{\gamma} \omega$ does.

\medskip

Fix a base point $v_0 \in \Gamma$.  For any point $v \in \Gamma$, the integral
$$ \int_{v_0}^v \omega $$
is well-defined up to a choice of path from $v_0$ to $v$.
This gives a map
$$ AJ_{v_0} : \Gamma \to \Omega (\Gamma)^* / H_1 (\Gamma , \mathbb{Z})$$
known as the \emph{tropical Abel-Jacobi map}.
Extending linearly to $\Div(\Gamma)$ and then restricting to $\Div^0(\Gamma)$, we obtain a map
$$ AJ : \Div^0(\Gamma) \to \Omega (\Gamma)^* / H_1 (\Gamma , \mathbb{Z})$$
which does not depend on the choice of a base point.  As in the classical case of Riemann surfaces, we have (cf.  \cite{MikhalkinZharkov08}):

\begin{TropicalAJ}
\label{Thm:AbelJacobi}
The map $AJ$ is surjective and its kernel is precisely ${\rm Prin}(\Gamma)$.  Thus there is a canonical isomorphism
$$ AJ : \Div^0(\Gamma) / {\rm Prin}(\Gamma) \cong \Omega (\Gamma)^* / H_1 (\Gamma , \mathbb{Z}) $$
between the Jacobian of $\Gamma$ and a $g$-dimensional real torus.
\end{TropicalAJ}

\begin{example}
\label{Ex:Genus2}
We consider the two metric graphs of genus 2 pictured in Figure \ref{Fig:Genus2}.  In the first case, we can choose a basis $\omega_1 , \omega_2$ of harmonic 1-forms by assigning the integer 1 to one of the (oriented) loops, and the integer 0 to the other.  We let $\eta_1 , \eta_2$ be elements of the dual basis.  We then see from the tropical Abel-Jacobi Theorem that
$$ \Jac (\Gamma) \cong \mathbb{R}^2 / (\mathbb{Z} \eta_1 + \mathbb{Z} \eta_2 ) . $$
A similar argument in the second case yields
$$ \Jac (\Gamma') \cong \mathbb{R}^2 / \left( \mathbb{Z} [ \eta_1 + \frac{1}{2} \eta_2 ] + \mathbb{Z} [ \frac{1}{2} \eta_1 + \eta_2 ]\right) . $$

\begin{figure}
\begin{tikzpicture}

\draw [ball color=black] (0,0) circle (0.55mm);
\draw (-1,0) circle (1);
\draw (1,0) circle (1);
\draw (-.25,0) node {$v_0$};

\draw [ball color=black] (3,0) circle (0.55mm);
\draw [ball color=black] (5,0) circle (0.55mm);
\draw (3,0)--(5,0);
\draw (4,0) circle (1);
\draw (2.75,0) node {$v_0$};

\end{tikzpicture}
\caption{Two metric graphs $\Gamma$, $\Gamma '$ of genus 2.}
\label{Fig:Genus2}
\end{figure}
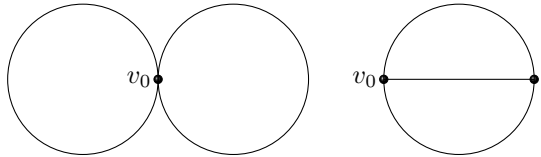
\end{example}

Although the Jacobians of the metric graphs $\Gamma$ and $\Gamma'$ from Example~\ref{Ex:Genus2} are isomorphic as abstract real tori, they are non-isomorphic as
{\em principally polarized} real tori in the sense of \cite{MikhalkinZharkov08}.  In fact, there is an analogue of the Torelli theorem in this context saying that up to certain
``Whitney flips'', the Jacobian as a principally polarized real torus determines the metric graph $\Gamma$; see \cite{CaporasoViviani10}.
The map $AJ_{v_0} : \Gamma \to \Omega (\Gamma)^* / H_1 (\Gamma , \mathbb{Z})$ is {\em harmonic} (or {\em balanced}) in a certain natural sense; see e.g. \cite[Theorem 4.1]{BakerFaber11}.  This fact is used in the paper \cite{KRZB15}, about which we will say more in \S\ref{Section:KRZB}.
Basic properties of the tropical Abel-Jacobi map are also used in important ways in \cite{Coppens14} and \cite{lifting}.




\section{Ranks of Divisors}
\label{Section:Rank}

\subsection{Linear systems}
\label{Subsection:LinearSystems}

By analogy with algebraic curves, a divisor $D = \sum D(v) v$ on a (metric) graph is called \emph{effective} if $D(v) \geq 0$ for all $v$, and we write $D \geq 0$.  The complete linear series of a divisor $D$ is defined to be
$$ \vert D \vert = \{ E \geq 0 \; \vert \; E \sim D \} . $$
Similarly, we write
$$ R(D) = \{ f \in \PL (\Gamma ) \; \vert \; \ddiv (f) + D \geq 0 \} $$
for the set of tropical rational functions with poles along the divisor $D$.

\medskip

As explained in \cite{GathmannKerber08}, the complete linear series $|D|$ has the structure of a compact polyhedral complex.  However, this polyhedral complex often fails to be equidimensional, as the following example shows.

\begin{example}
\label{Ex:NotEquidimensional}
Consider the metric graph pictured in Figure \ref{Fig:NotEquidimensional}, consisting of two loops attached at a point $v$.  Let $D = 2v + w$, where $w$ is a point on the interior of the first loop.  The complete linear system is the union of two tori.  The first, which consists of divisors equivalent to $D$ that are supported on the first loop, has dimension two, while the other, which consists of divisors equivalent to $D$ that have positive degree on the second loop, has dimension one.

\begin{figure}
\begin{tikzpicture}

\draw [ball color=black] (0,0) circle (0.55mm);
\draw (-1,0) circle (1);
\draw (1,0) circle (1);
\draw (-.25,0) node {$v$};
\draw [ball color=black] (-2,0) circle (0.55mm);
\draw (-2.25,0) node {$w$};

\end{tikzpicture}
\caption{The linear system $\vert 2v+w \vert$ is not equidimensional.}
\label{Fig:NotEquidimensional}
\end{figure}
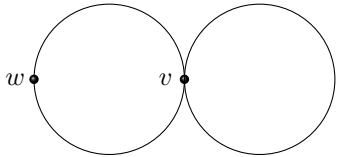
\end{example}

We now wish to define the rank of a divisor on a graph.  As the previous example shows, the appropriate definition should not be the dimension of the linear system $\vert D \vert$, considered as a polyhedral complex.\footnote{Another natural idea would be to try to define $r(D)$ as one less than the ``dimension'' of $R(D)$ considered as a semimodule over the tropical semiring ${\mathbf T}$ consisting of $\RR \cup \{ \infty \}$ together with the operations of min and plus.  However, this approach also faces significant difficulties.  See \cite{HMY12} for a detailed discussion of the tropical semimodule structure on $R(D)$.}
Instead, we note that a line bundle $L$ on an algebraic curve $C$ has rank at least $r$ if and only if, for every collection of $r$ points of $C$, there is a non-zero section of $L$ that vanishes at those points.  This motivates the following definition.

\begin{definition}
\label{Def:Rank}
Let $D$ be a divisor on a (metric) graph.  If $D$ is not equivalent to an effective divisor, we define its rank to be $-1$.  Otherwise, we define $r(D)$ to be the largest non-negative integer $r$ such that $\vert D-E \vert \neq \emptyset$ for all effective divisors $E$ of degree $r$.
\end{definition}

\begin{example}
\label{Ex:RankNotDimension}
Even when a linear series is equidimensional, its dimension may not be equal to the rank of the corresponding divisor.  For example, consider the metric graph pictured in Figure \ref{Fig:RankNotDimension}, consisting of a loop meeting a line segment in a vertex $v$.  The linear system $\vert v \vert$ is 1-dimensional, because $v$ is equivalent to any point on the line segment.  The rank of $v$, however, is 0, because if $w$ lies in the interior of the loop, then $v$ is not equivalent to $w$.

\begin{figure}
\begin{tikzpicture}

\draw [ball color=black] (0,0) circle (0.55mm);
\draw (-1,0) circle (1);
\draw (0,0)--(2,0);
\draw [ball color=black] (2,0) circle (0.55mm);
\draw (-.25,0) node {$v$};

\end{tikzpicture}
\caption{The vertex $v$ moves in a one-dimensional family, but has rank zero.}
\label{Fig:RankNotDimension}
\end{figure}
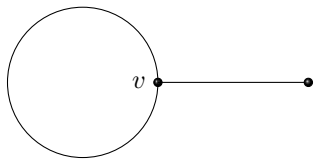
\end{example}

\begin{remark}
\label{Remark:OtherRanks}
There are several other notions of rank in the literature related to our setup of a line bundle on a degenerating family of curves.  We mention, for example, the generalized rank functions of Katz and Zureick-Brown \cite{KZB13} and the algebraic rank of Caporaso \cite{Caporaso13}.  The rank as defined here is sometimes referred to as the \emph{combinatorial rank} to distinguish it from these other invariants.
\end{remark}

As in the case of curves, we write
$$ W^r_d ( \Gamma ) := \{ D \in \Pic^d (\Gamma) \; \vert \; r(D) \geq r \} . $$
We similarly define the \emph{gonality} of a graph to be the smallest degree of a divisor of rank at least one.  The \emph{Clifford index} of the graph is
$$ \Cliff (\Gamma) := \min \{ \deg (D) - 2r(D) \; \vert \; r(D) > \max \{ 0, \deg (D) - g+1 \} \} . $$

\begin{remark}
The definition of gonality above is sometimes called the \emph{divisorial} gonality, to distinguish it from the \emph{stable gonality}, which is the smallest degree of a harmonic morphism from a modification of the given metric graph to a tree.
The divisorial gonality is always less than or equal to the stable gonality, see e.g. \cite{ABBR14b}.
\end{remark}

\subsection{Specialization}
\label{Subsection:Specialization}

One of the key properties of the combinatorial rank is its behavior under specialization.  Note that the specialization map takes effective line bundles to effective divisors.  Combining this with the fact that it takes principal divisors to principal divisors, we see that, for any divisor $D$ on $C$, we have
$$ \Trop \vert D \vert \subseteq \vert \Trop (D) \vert \mbox{ \ and \ } $$
$$ \trop \mathcal{L} (D) \subseteq R ( \Trop (D)) . $$
Combining this with the definition of rank yields the following semicontinuity result.

\begin{SpecializationThm} \cite{Baker08}
\label{Thm:Specialization}
Let $D$ be a divisor on $C$.  Then
\[
r(D) \leq r( \Trop (D)).
\]
\end{SpecializationThm}

Another way of stating this is that $\Trop (W^r_d (C)) \subseteq W^r_d ( \Gamma ).$
The power of the Specialization Theorem lies in the fact that the rank of the divisor $D$ is an algebro-geometric invariant, whereas the rank of $\Trop (D)$ is a combinatorial invariant.  We can therefore use techniques from each field to inform the other.  For example, an immediate consequence of specialization is the following fact.

\begin{theorem}
\label{Thm:WrdNonempty}
Let $\Gamma$ be a metric graph of genus $g$, and let $d,r$ be positive integers such that $g-(r+1)(g-d+r) \geq 0$.  Then $W^r_d (\Gamma) \neq \emptyset$.
\end{theorem}

\begin{proof}
As we will se in Corollary \ref{Cor:Surjective}, there exists a curve $C$ over a discretely valued field $K$, and a semistable $R$-model $\cC$ of $C$ such that the dual graph of the central fiber is isometric to $\Gamma$.
A well-known theorem of Kempf and Kleiman-Laksov \cite{Kempf71, KleimanLaksov74} asserts that $W^r_d (C) \neq \emptyset$. It follows from Theorem \ref{Thm:Specialization} that $W^r_d (\Gamma) \neq \emptyset$ as well.
\end{proof}

\begin{corollary}
\label{Cor:Gonality}
A metric graph of genus $g$ has gonality at most $\lceil \frac{g+2}{2} \rceil$.
\end{corollary}

\begin{remark}
\label{Remark:FiniteGraphGonality}
We are unaware of a purely combinatorial proof of Theorem~\ref{Thm:WrdNonempty}.
There are many reasons that such a proof would be of independent interest.  For example, a combinatorial proof could shed some light on whether the analogous statement is true for finite graphs, as conjectured in \cite[Conjectures 3.10, 3.14]{Baker08}.\footnote{Sam Payne has pointed out that there is a gap in the proof of Conjectures 3.10 and 3.14 of \cite{Baker08} given in \cite{Caporaso12}.  The claim on page 82 that $W^r_{d, \phi}$ has nonempty fiber over $b_0$ does not follow from the discussion that precedes it.  {\em A priori}, the fiber of $\overline{W^r_{d, \phi}}$ over $b_0$ might be contained in the boundary $\overline{P^d_\phi} \smallsetminus \Pic^d_\phi$ of the compactified relative Picard scheme.}
\end{remark}

\subsection{Riemann-Roch}
\label{Subsection:RiemannRoch}

Another key property of the combinatorial rank is that it satisfies a tropical analogue of the Riemann-Roch theorem:

\begin{TropicalRR}[\cite{GathmannKerber08, MikhalkinZharkov08}]
\label{Thm:RiemannRoch}
Let $\Gamma$ be a metric graph of genus $g$, and let $K_{\Gamma} := \sum_{v \in \Gamma} (\val(v) - 2)v$ be the {\em canonical divisor} on $\Gamma$.  Then for every divisor $D$ on $\Gamma$,
$$ r(D) - r(K_{\Gamma} - D) = \deg (D) - g + 1 . $$
\end{TropicalRR}

The first result of this kind was the discrete analogue of Theorem~\ref{Thm:RiemannRoch} for finite (non-metric) graphs proved in \cite{BakerNorine07}.  In the Baker-Norine Riemann-Roch theorem, one defines $r(D)$ for a divisor $D$ on a graph $G$ exactly as in Definition~\ref{Def:Rank}, the subtle difference being that the effective divisor $E$ is restricted to the vertices of $G$.  Gathmann and Kerber \cite{GathmannKerber08} showed that one can deduce Theorem~\ref{Thm:RiemannRoch} from the Baker-Norine theorem using a clever approximation argument, whereas Mikhalkin and Zharkov \cite{MikhalkinZharkov08} generalized the method of proof from \cite{BakerNorine07} to the metric graph setting.  Later on, Hladky-Kral-Norine \cite{HladkyKralNorine10} and Luo \cite{Luo11} proved theorems which imply that one can also deduce Riemann-Roch for graphs from tropical Riemann-Roch.  (We will discuss Luo's theory of rank-determining sets in \S\ref{Subsection:RankDetermining}.)
So in retrospect, one can say that in some sense the Baker-Norine theorem and the theorem stated above are equivalent.

\medskip

Baker and Norine's strategy of proof for Theorem~\ref{Thm:RiemannRoch}, as modified by Mikhalkin and Zharkov, is to first show that if $\mathcal{O}$ is an orientation of the graph (i.e., a choice of a head vertex and tail vertex for each edge of $G$), then
$$ D_{\mathcal{O}} := \sum_{v \in V(G)} (\indeg_{\mathcal O} (v) - 1)v $$
is a divisor of degree $g-1$, and this divisor has rank $-1$ if and only if the orientation $\mathcal{O}$ is acyclic.  This fact helps to establish the Riemann-Roch Theorem in the case of divisors of degree $g-1$, which serves as the base case for the more general argument.
It is interesting to note that the Tropical Riemann-Roch Theorem has thus far has resisted attempts to prove it via classical algebraic geometry.  At present, neither the Tropical Riemann-Roch Theorem nor the classical Riemann-Roch Theorem for algebraic curves is known to imply the other.

\medskip

If $\cC$ is a strongly semistable $R$-model for a curve $C$ over a discretely valued field $K$ with the property that all irreducible components of the special fiber $C_0$ have genus $0$, then the multidegree of the relative dualizing sheaf $\Omega^1_{\cC / R}$ is equal to the canonical divisor of the graph $G$.  This is a simple consequence of the adjunction formula, which shows more generally that ${\rm mdeg}(\Omega^1_{\cC / R}) = K_{(G,\omega)}^\#$ in the terminology of \S\ref{Subsection:vertexweight} below.
If $K$ is not discretely valued, this is still true with the right definition of the sheaf $\Omega^1_{\cC / R}$ (see \cite{KRZB15}).
This ``explains'' in some sense why there is a canonical {\em divisor} on a metric graph while on an algebraic curve there is merely a canonical {\em divisor class}.

\medskip

It is clear from the definition of rank that if $D$ and $E$ are divisors on a metric graph $\Gamma$ having non-negative rank, then $r(D+E) \geq r(D) + r(E)$.  Combining this with tropical Riemann-Roch, one obtains a tropical version of Clifford's inequality:
\begin{TropicalClifford}
\label{Thm:Clifford}
Let $D$ be a {\em special} divisor on a metric graph $\Gamma$, that is, a divisor such that both $D$ and  $K_{\Gamma} - D$ have nonnegative rank.  Then
$$ r(D) \leq \frac{1}{2} \deg (D) . $$
\end{TropicalClifford}

\begin{remark}
The classical version of Clifford's Theorem is typically stated in two parts.  The first part is the inequality above, while the second part states that, when equality holds, either $D \sim 0$, $D \sim K_C$, or the curve $C$ is hyperelliptic and the linear equivalence class of $D$ is a multiple of the unique $\g^1_2$.
The same conditions for equality hold in the tropical case as well, by a recent theorem of Coppens, but the proof is quite subtle as the classical methods do not work in the tropical context.   See \cite{Coppens14} for details.
\end{remark}

Note that, as in the case of curves, the Riemann-Roch Theorem significantly limits the possible ranks that a divisor of fixed degree on a metric graph may have.  For example, a divisor of negative degree necessarily has rank $-1$, so a divisor of degree $d > 2g-2$ must have rank $d-g$.  It is only in the intermediate range $0 \leq d \leq 2g-2$ where there are multiple possibilities for the rank.

\begin{example}
\label{Ex:CircleRiemannRoch}
The canonical divisor on a circle is trivial, and it is the only divisor of degree 0 with non-negative rank.  If $D$ is a divisor of degree $d > 0$, then by Riemann-Roch $D$ has rank $d-1$.
This can also be seen using the fact that the circle is a torsor for its Jacobian, as in Example~\ref{Ex:Circle}:
if $E$ is an effective divisor of degree $d-1$, then there is a unique point $P$ such that $D-E \sim P$, and hence $r(D) \geq d-1$.
\end{example}

\begin{example}
\label{Ex:Hyperelliptic}
The smallest genus for which the rank of an effective divisor is not completely determined by the degree is genus 3.  Pictured in Figure \ref{Fig:Hyperelliptic} are two examples of genus 3 metric graphs, the first of which is {\em hyperelliptic}, meaning that
it admits a divisor of degree 2 and rank 1, and the second of which is not.
For the first graph, one can check by hand that the sum of the two vertices on the left has rank at least 1, and it cannot have rank higher than 1 by Clifford's Theorem.  We will show that the second graph is not hyperelliptic in Example \ref{Ex:Nonhyperelliptic}, as the argument will require some techniques for computing ranks of divisors that we will discuss in the next section.

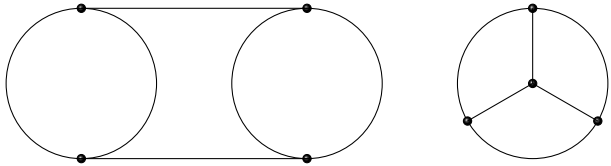
\begin{figure}
\begin{tikzpicture}

\draw [ball color=black] (0,1) circle (0.55mm);
\draw [ball color=black] (0,-1) circle (0.55mm);
\draw (0,0) circle (1);
\draw [ball color=black] (3,1) circle (0.55mm);
\draw [ball color=black] (3,-1) circle (0.55mm);
\draw (3,0) circle (1);
\draw (0,1)--(3,1);
\draw (0,-1)--(3,-1);

\draw (6,0) circle (1);
\draw [ball color=black] (6,0) circle (0.55mm);
\draw [ball color=black] (6,1) circle (0.55mm);
\draw [ball color=black] (5.134,-.5) circle (0.55mm);
\draw [ball color=black] (6.866,-.5) circle (0.55mm);
\draw (6,0)--(6,1);
\draw (6,0)--(5.134,-.5);
\draw (6,0)--(6.866,-.5);

\end{tikzpicture}
\caption{Two metric graphs of genus 3, the first of which is hyperelliptic, and the second of which is not.  (All edges have length 1.)}
\label{Fig:Hyperelliptic}
\end{figure}
\end{example}

\subsection{Divisors on vertex-weighted graphs}
\label{Subsection:vertexweight}
In \cite{AminiCaporaso13}, Amini and Caporaso formulate a refinement of the Specialization Theorem which takes into account the genera of the components of the special fiber.  In this section we describe their result following the presentation in \cite{AminiBaker12}.

\medskip

A {\em vertex-weighted metric graph} is a pair $(\Gamma,\omega)$ consisting of a metric graph $\Gamma$ and a weight function $\omega : \Gamma \to \ZZ_{\geq 0}$ such that $\omega(x)=0$ for all but finitely many $x \in \Gamma$.
Following~\cite{AminiCaporaso13}, we define a new metric graph $\Gamma^\#$ by attaching $\omega(x)$ loops of arbitrary positive length at each point $x \in \Gamma$.  There is a natural inclusion of $\Gamma$ into $\Gamma^\#$.
The {\em canonical divisor} of $(\Gamma,\omega)$ is defined to be
\[
K^\# = K_\Gamma + \sum_{x \in \Gamma} 2\omega(x),
\]
which can naturally be identified with the canonical divisor of $K_{\Gamma^\#}$ restricted to $\Gamma$.
Its degree is $2g^\# -2$, where $g^\# = g(\Gamma) + \sum_{x \in \Gamma} \omega(x)$ is the genus of $\Gamma^\#$.

\medskip

Following~\cite{AminiCaporaso13}, the {\em weighted rank} $r^\#$ of a divisor $D$ on $\Gamma$ is defined to be $r^\#(D):= r_{\Gamma^\#}(D)$.
By \cite[Corollary 4.12]{AminiBaker12}, we have the more intrinsic description
\[
r^\#(D) = \min_{0\leq E \leq \mathcal W} \bigl(\deg(E)+ r_\Gamma(D-2E)\bigr),
\]
where $\mathcal W = \sum_{x \in \Gamma} \omega(x)(x)$.

\medskip

The Riemann-Roch theorem for $\Gamma^\#$ implies the following ``vertex-weighted'' Riemann-Roch theorem for $\Gamma$:
\[
r^\#(D) - r^\#(K^\# - D) = \deg(D) + 1 - g^\#.
\]

\medskip

If $C$ is a curve over $K$, together with a semistable model $\cC$ over $R$, we define the associated vertex-weighted metric graph $(\Gamma,\omega)$ by taking $\Gamma$ to be the skeleton of $\cC$
and defining the weight function $\omega$  by $\omega(v) = g_v$.
With this definition, the genus of the weighted metric graph $(\Gamma,\omega)$ is equal to the genus of $C$ and we have $\trop(K_C) \sim K^\#$ on $\Gamma$ \cite[\S{4.7.1}]{AminiBaker12}.
The following weighted version of the Specialization Theorem, inspired by the results of ~\cite{AminiCaporaso13}, is proved in \cite[Theorem 4.13]{AminiBaker12}:

\begin{WtedSpecializationThm}
\label{thm:weightedmetricspecialization}
For every divisor $D \in \Div(C)$, we have $r_C(D) \leq r^\#(\trop(D))$.
\end{WtedSpecializationThm}

\section{Combinatorial Techniques}
\label{Section:Techniques}

The tropical approach to degeneration of line bundles in algebraic geometry derives its power from the combinatorial tools which one has available, many of which have no classical analogues.  We describe some of these tools in this section.

\subsection{Reduced divisors and Dhar's burning algorithm}
\label{Subsection:ReducedDivisors}

\begin{definition}
\label{def:v-reduced}
Let $G$ be a finite graph.  Given a divisor $D \in {\rm Div}(G)$ and a vertex $v$ of $G$, we say that $D$ is \textit{$v$-reduced} if
\begin{itemize}
\item[(RD1)] $D(w) \geq 0$ for every $w \neq v$, and
\item[(RD2)] for every non-empty set $A \subseteq V(G) \smallsetminus \{ v \}$ there is a vertex $w \in A$ such that $\mathrm{outdeg}_A (w) > D(w)$.
\end{itemize}
\end{definition}

Here $\mathrm{outdeg}_A (w)$ denotes the {\rm outdegree} of $w$ with respect to $A$, i.e., the number of edges connecting $w \in A$ to a vertex not in $A$.
The following important result (cf. \cite[Proposition 3.1]{BakerNorine07}) shows that $v$-reduced divisors form a distinguished set of representatives for linear equivalence classes of divisors on $G$:

\begin{lemma}
\label{Lem:ReducedDivisors}
Every divisor on $G$ is equivalent to a unique $v$-reduced divisor.
\end{lemma}

If $D$ has nonnegative rank, the $v$-reduced divisor equivalent to $D$ is the divisor in $\vert D \vert$ that is lexicographically ``closest'' to $v$.  It is a discrete analogue of the unique divisor in a classical linear series $|D|$ with the highest possible order of vanishing at a given point $p \in C$.

\medskip

There is a simple algorithm for determining whether a given divisor satisfying (RD1) above is $v$-reduced, known as \textit{Dhar's burning algorithm}.  For $w \neq v$, imagine that there are $D(w)$ buckets of water at $w$.  Now, light a fire at $v$.  The fire starts spreading through the graph, burning through an edge as soon as one of its endpoints is burnt, and burning a vertex $w$ if the number of burnt edges adjacent to $w$ is greater than $D(w)$ (that is, there is not enough water to fight the fire).  The divisor $D$ is $v$-reduced if and only if the fire consumes the whole graph.  For a detailed account of this algorithm, we refer to \cite{Dhar90} and the more recent \cite[Section 5.1]{BS13}.

\medskip

There is a completely analogous set of definitions and results for metric graphs.  Let $\Gamma$ be a metric graph, let $D$ be a divisor on $\Gamma$, and choose a point $v \in \Gamma$ (which need not be a vertex).
We say that $D$ is {\em $v$-reduced} if the two conditions from Definition~\ref{def:v-reduced} hold, with the second condition replaced by
\medskip
\begin{itemize}
\item[(RD2$^\prime$)] for every closed, connected, non-empty set $A \subseteq \Gamma \smallsetminus \{ v \}$ there is a point $w \in A$ such that $\mathrm{outdeg}_A (w) > D(w)$.
\end{itemize}
\medskip

It is not hard to see that condition (RD2$^\prime$) is equivalent to requiring that for every non-constant tropical rational function $f \in R(\Gamma)$ with a global maximum at $v$, the divisor $D' := D+{\rm div}(f)$ does not satisfy (RD1), i.e., there exists $w\neq v$ in $\Gamma$ such that $D'(w) < 0$.

\medskip

The analogue of Lemma~\ref{Lem:ReducedDivisors} remains true in the metric graph context: every divisor on $\Gamma$ is equivalent to a unique $v$-reduced divisor.  Moreover, Dhar's burning algorithm as formulated above holds almost {\em verbatim} for metric graphs:
the fire starts spreading through $\Gamma$, getting blocked at a point $w \in \Gamma$ iff the number of burnt tangent directions at $w$ is less than or equal to $D(w)$; the divisor $D$ is $v$-reduced if and only if the fire consumes all of $\Gamma$.

\medskip

We note the following important fact, which is useful for computing ranks of divisors.

\begin{lemma}
\label{Lem:ReducedRank}
Let $D$ be a divisor on a finite or metric graph, and suppose that $D$ is $v$-reduced for some $v$.  If $D$ has non-negative rank, then $D(v) \geq r(D)$.
\end{lemma}

\begin{example}
\label{Ex:Nonhyperelliptic}
Let $\Gamma$ be the complete graph on 4 vertices endowed with arbitrary edge lengths.  We can use the theory of reduced divisors to show that $\Gamma$ is not hyperelliptic, justifying one of the claims in Example \ref{Ex:Hyperelliptic}.
Suppose that there exists a divisor $D$ on $\Gamma$ of degree 2 and rank 1 and choose a vertex $v$.  Since $D$ has rank 1, $D \sim D' := v+v'$ for some $v' \in \Gamma$.  Now, let $w \neq v,v'$ be a vertex.  Note that there are at least two paths from $w$ to $v$ that do not pass through $v'$, and if $v = v'$, there are three.  It follows by Dhar's burning algorithm that $D'$ is $w$-reduced.  But $D'(w) = 0$, contradicting the fact that $D$ (and hence $D'$) has rank 1.

\end{example}

If a given divisor is not $v$-reduced, Dhar's burning algorithm provides a method for finding the unique equivalent $v$-reduced divisor.  In the case of finite graphs, after performing Dhar's burning algorithm, if we fire the vertices that are left unburnt we obtain a divisor that is lexicographically closer to the $v$-reduced divisor, and after iterating the procedure a finite number of times, it terminates with the $v$-reduced divisor (cf. \cite{BS13}).  For metric graphs, there is a similar procedure but with additional subtleties --- we refer the interested reader to \cite[Algorithm 2.5]{Luo11} and \cite{Backman14a}.

\begin{example}
\label{Ex:Dhar}
Consider the finite graph depicted in Figure \ref{Fig:Dhar}, consisting of two triangles meeting at a vertex $v_3$.  We let $D = v_1 + v_2$, and compute the $v_5$-reduced divisor equivalent to $D$.  After performing Dhar's burning algorithm once, we see that vertices $v_1$ and $v_2$ are left unburnt.  Firing these, we see that $D$ is equivalent to $2v_3$.  Performing Dhar's burning algorithm a second time, all three of the vertices $v_1, v_2, v_3$ are not burnt.  Firing these, we obtain the divisor $v_4 + v_5$.  A third run of Dhar's burning algorithm shows that $v_4 + v_5$ is $v_5$-reduced.

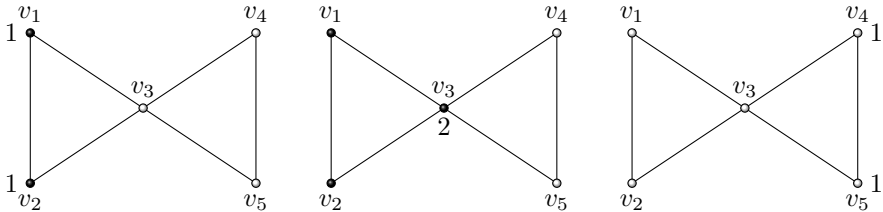
\begin{figure}
\begin{tikzpicture}

\draw (0,1)--(0,-1);
\draw (0,1)--(1.5,0);
\draw (0,-1)--(1.5,0);
\draw (1.5,0)--(3,1);
\draw (1.5,0)--(3,-1);
\draw (3,1)--(3,-1);
\draw [ball color=black] (0,1) circle (0.55mm);
\draw [ball color=black] (0,-1) circle (0.55mm);
\draw [ball color=white] (1.5,0) circle (0.55mm);
\draw [ball color=white] (3,1) circle (0.55mm);
\draw [ball color=white] (3,-1) circle (0.55mm);
\draw (0,1.25) node {$v_1$};
\draw (0,-1.25) node {$v_2$};
\draw (1.5,.25) node {$v_3$};
\draw (3,1.25) node {$v_4$};
\draw (3,-1.25) node {$v_5$};
\draw (-.25,1) node {1};
\draw (-.25,-1) node {1};

\draw (4,1)--(4,-1);
\draw (4,1)--(5.5,0);
\draw (4,-1)--(5.5,0);
\draw (5.5,0)--(7,1);
\draw (5.5,0)--(7,-1);
\draw (7,1)--(7,-1);
\draw [ball color=black] (4,1) circle (0.55mm);
\draw [ball color=black] (4,-1) circle (0.55mm);
\draw [ball color=black] (5.5,0) circle (0.55mm);
\draw [ball color=white] (7,1) circle (0.55mm);
\draw [ball color=white] (7,-1) circle (0.55mm);
\draw (4,1.25) node {$v_1$};
\draw (4,-1.25) node {$v_2$};
\draw (5.5,.25) node {$v_3$};
\draw (7,1.25) node {$v_4$};
\draw (7,-1.25) node {$v_5$};
\draw (5.5,-.25) node {2};

\draw (8,1)--(8,-1);
\draw (8,1)--(9.5,0);
\draw (8,-1)--(9.5,0);
\draw (9.5,0)--(11,1);
\draw (9.5,0)--(11,-1);
\draw (11,1)--(11,-1);
\draw [ball color=white] (8,1) circle (0.55mm);
\draw [ball color=white] (8,-1) circle (0.55mm);
\draw [ball color=white] (9.5,0) circle (0.55mm);
\draw [ball color=white] (11,1) circle (0.55mm);
\draw [ball color=white] (11,-1) circle (0.55mm);
\draw (8,1.25) node {$v_1$};
\draw (8,-1.25) node {$v_2$};
\draw (9.5,.25) node {$v_3$};
\draw (11,1.25) node {$v_4$};
\draw (11,-1.25) node {$v_5$};
\draw (11.25,1) node {1};
\draw (11.25,-1) node {1};

\end{tikzpicture}
\caption{Using Dhar's burning algorithm to compute the $v_5$-reduced divisor equivalent to $v_1 + v_2$.  The burnt vertices after each iteration are colored white.}
\label{Fig:Dhar}
\end{figure}
\end{example}

\subsection{Rank-determining sets}
\label{Subsection:RankDetermining}

The definition of the rank of a divisor on a finite graph $G$ implies easily that there is an algorithm for computing it.\footnote{Although there is no known {\em efficient} algorithm; indeed, it is proved in \cite{KissTothmeresz14} that this problem is NP-hard.}
Indeed, since there are only a finite number of effective divisors $E$ of a given degree on $G$, we are reduced to the problem of determining whether a given divisor is equivalent to an effective divisor or not.
This problem can be solved in polynomial time by using the iterated version of Dhar's algorithm described above to compute the $v$-reduced divisor $D'$ equivalent to $D$ for some vertex $v$.  If $D'(v) < 0$ then $|D| = \emptyset$, and otherwise $|D| \neq \emptyset$.

\medskip

If one attempts to generalize this algorithm to the case of metric graphs, there is an immediate problem, since there are now an infinite number of effective divisors $E$ to test.
The idea behind rank-determining sets is that it suffices, in the definition of $r(D)$, to restrict to a finite set of effective divisors $E$.

\begin{definition}
\label{Def:RankDetermining}
Let $\Gamma$ be a metric graph.  A subset $A \subseteq \Gamma$ is a \emph{rank-determining set} if, for any divisor $D$ on $\Gamma$, $D$ has rank at least $r$ if and only if $\vert D-E \vert \neq \emptyset$ for every effective divisor $E$ of degree $r$ supported on $A$.
\end{definition}

In \cite{Luo11}, Luo provides a criterion for a subset of a metric graph to be rank-determining.  A different proof of Luo's criterion has been given recently by Backman \cite{Backman14b}.
Luo defines a \emph{special open set} to be a connected open set $U \subseteq \Gamma$ such that every connected component $X \subseteq \Gamma \smallsetminus U$ contains a boundary point $v$ such that $\outdeg_X (v) \geq 2$.

\begin{theorem}[\cite{Luo11}]
\label{Thm:RankDetermining}
A subset $A \subseteq \Gamma$ is rank-determining if and only if all nonempty special open subsets of $\Gamma$ intersect $A$.
\end{theorem}

\begin{corollary}
\label{Cor:VerticesAreRankDetermining}
Let $G$ be a model for a metric graph $\Gamma$.  If $G$ has no loops, then the vertices of $G$ are a rank-determining set.
\end{corollary}

There are lots of other interesting rank-determining sets besides vertices of models.

\begin{example}
\label{Ex:ComplementOfSpanningTree}
Let $G$ be a model for a metric graph $\Gamma$.
Choose a spanning tree of $G$ and let $e_1 , \ldots , e_g$ be the edges in the complement of the spanning tree.  For each such edge $e_i$, choose a point $v_i$ in its interior, and let $w$ be any other point of $\Gamma$.  Then the set $A = \{ v_1 , \ldots , v_g \} \cup \{ w \}$ is rank-determining.
This construction is used, for example, in \cite{Coppens14}.
\end{example}

\begin{example}
\label{Ex:BlueVerticesAreRankDetermining}
Let $G$ be a bipartite finite graph, and let $\Gamma$ be a metric graph having $G$ as a model.  If we fix a 2-coloring of the vertices of $G$, then the vertices of one color are a rank-determining set.  This is the key observation in \cite{Heawood}, in which the second author shows that the Heawood graph admits a divisor of degree 7 and rank 2, regardless of the choice of edge lengths.  The interest in this example arises because it shows that there is a non-empty open subset of the (highest-dimensional component of the) moduli space $M_8^{\rm trop}$ containing no Brill-Noether general metric graph. (See \S\ref{Subsection:Moduli} for a description of $M_g^{\rm trop}$.)
\end{example}

\subsection{Tropical independence}
\label{Subsection:TropicalIndependence}

Many interesting questions about algebraic curves concern the ranks of linear maps between the vector spaces $\cL(D)$.  For example, both the Gieseker-Petri Theorem and the Maximal Rank Conjecture are statements about the rank of the multiplication maps
$$ \mu : \mathcal{L} (D) \otimes \mathcal{L} (D') \to \mathcal{L} (D+D') $$
for certain pairs of divisors $D,D'$ on a general curve.

\medskip

One simple strategy for showing that a map, such as $\mu$, has rank at least $k$ is to carefully choose $k$ elements of the image, and then check that they are linearly independent.  To this end, we formulate a notion of \emph{tropical independence}, which gives a sufficient condition for linear independence of rational functions on a curve $C$ in terms of the associated piecewise linear functions on the metric graph $\Gamma$.

\begin{definition}[\cite{tropicalGP}]
\label{Def:TropicalIndependence}
A set of piecewise linear functions $\{ f_1, \ldots, f_k \}$ on a metric graph $\Gamma$ is \emph{tropically dependent} if there are real numbers $b_1, \ldots, b_k$ such that the minimum
\[
\min \{f_1(v) + b_1, \ldots, f_k(v) + b_k \}
\]
occurs at least twice at every point $v$ in $\Gamma$.
\end{definition}

If there are no such real numbers $b_1, \ldots, b_k$ then we say $\{ f_1, \ldots, f_k \}$ is \emph{tropically independent}.  We note that linearly dependent functions on $C$ specialize to tropically dependent functions on $\Gamma$.  Although the definition of tropical independence is merely a translation of linear dependence into tropical language, one can often check tropical independence using combinatorial methods.  The following lemma illustrates this idea.

\begin{lemma}[\cite{tropicalGP}]
\label{Lem:MinChips}
Let $D$ be a divisor on a metric graph $\Gamma$, with $f_1, \ldots, f_k$ piecewise linear functions in $R(D)$, and let
\[
\theta = \min \{ f_1, \ldots, f_k \}.
\]
Let $\Gamma_j \subset \Gamma$ be the closed set where $\theta = f_j$, and let $v \in \Gamma_j$.  Then the support of $\ddiv( \theta ) +D$ contains $v$ if and only if $v$ belongs to either
\begin{enumerate}
\item the support of $\ddiv( f_j ) + D$, or
\item the boundary of $\Gamma_j$.
\end{enumerate}
\end{lemma}




\subsection{Break divisors}
\label{sec:BreakDivisorSection}

Another useful combinatorial tool for studying divisor classes on graphs and metric graphs is provided by the theory of {\em break divisors}, which was initiated by Mikhalkin--Zharkov in \cite{MikhalkinZharkov08} and studied further by An--Baker--Kuperberg--Shokrieh in {\cite{ABKS13}}.
Given a metric graph $\Gamma$ of genus $g$, fix a model $G$ for $\Gamma$.
For each spanning tree $T$ of $G$,
let $\Sigma_T$ be the image of the canonical map
$$ \prod_{e \not\in T} {e} \to \Div^g_+(\Gamma) $$
sending $(p_1,\ldots,p_g)$ to $p_1 + \cdots + p_g$.  (Here $\Div^g_+(\Gamma)$ denotes the set of effective divisors of degree $g$ on $\Gamma$ and $e$ denotes a {\em closed} edge of $G$, so the points $p_i$ are allowed to be vertices of $G$.)
We call $B(\Gamma) := \bigcup_T \Sigma_T$ the set of \emph{break divisors} on $\Gamma$.
The set of break divisors does not depend on the choice of the model $G$.
The following result shows that the natural map $B(\Gamma) \subset \Div^g_+(\Gamma) \to \Pic^g(\Gamma)$ is bijective:

\begin{theorem}[\cite{MikhalkinZharkov08, ABKS13}]
\label{thm:BreakDivisormainthm1}
Every divisor of degree $g$ on $\Gamma$ is linearly equivalent to a unique break divisor.
\end{theorem}

Since $B(\Gamma)$ is a compact subset of $\Div^g_+(\Gamma)$ and $\Pic^g(\Gamma)$ is also compact, it follows from general topology that there is a canonical continuous section $\sigma$ to the natural map $\pi: \Div^g(\Gamma) \to \Pic^g(\Gamma)$ whose image is precisely the set of break divisors.
In particular, every degree $g$ divisor class on a metric graph $\Gamma$ has a {\em canonical} effective representative.  The analogue of this statement in algebraic geometry is {\em false}: when $g=2$, for example, the natural map ${\rm Sym}^2(C) = \Div^2_+(C) \to \Pic^2(C)$ is a birational isomorphism which blows down
the $\PP^1$ corresponding to the fiber over the unique ${\mathfrak g}^1_2$, and this map has no section.  This highlights an interesting difference between the algebraic and tropical settings.

\medskip

The proof of Theorem~\ref{thm:BreakDivisormainthm1} in \cite{MikhalkinZharkov08} utilizes the theory of tropical theta functions and the tropical analogue of Riemann's theta constant.  A purely combinatorial proof based on the theory of {\em $q$-connected orientations} is given in \cite{ABKS13}, and the combinatorial proof yields an interesting
analogue of Theorem~\ref{thm:BreakDivisormainthm1} for finite graphs.  If $G$ is a finite graph and $\Gamma$ is its {\em regular realization}, in which all edges are assigned a length of 1, define the set of {\em integral break divisors} on $G$ to be $B(G) = B(\Gamma) \cap \Div(G)$.  In other words, $B(G)$ consists of all break divisors for the underlying metric graph $\Gamma$ which are supported on the vertices of $G$.

\begin{theorem}[\cite{ABKS13}]
\label{thm:BreakDivisormainthm2}
Every divisor of degree $g$ on $G$ is linearly equivalent to a unique integral break divisor.
\end{theorem}

Since $\Pic^g(G)$ and $\Pic^0(G)=\Jac(G)$ have the same cardinality (the former is naturally a torsor for the latter), it follows from Remark~\ref{Remark:CombLit} that
that the number of integral break divisors on $G$ is equal to the number of spanning trees of $G$, though there is in general no canonical bijection between the two.  A family of interesting combinatorial bijections is discussed in \cite{BakerWang14}.

\medskip

If we define $C_T = \pi(\Sigma_T)$, then $\Pic^g(\Gamma) = \bigcup_T C_T$ by Theorem~\ref{thm:BreakDivisormainthm1}.
It turns out that the relative interior of each cell $C_T$ is (the interior of) a parallelotope, and if $T \neq T'$ then the relative interiors of $C_T$ and $C_{T'}$ are disjoint.
Thus $\Pic^g(\Gamma)$ has a polyhedral decomposition depending only on the choice of a model for $\Gamma$.
The maximal cells in the decomposition correspond naturally to spanning trees, and the minimal cells (i.e. vertices) correspond naturally to integral break divisors, as illustrated in Figure~\ref{fig:PolyhedralDecomposition}.

\begin{figure}[h!]
\begin{center}
    \includegraphics[width=0.6\textwidth]{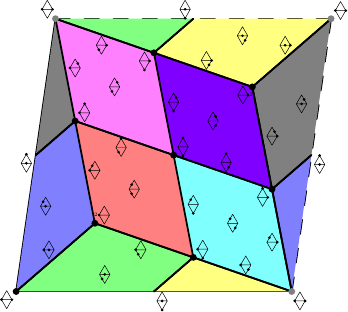}
\end{center}
  \caption{A polyhedral decomposition of $\Pic^2(\Gamma)$ for the metric realization of the graph $G$ obtained by deleting an edge from the complete graph $K_4$.}
  \label{fig:PolyhedralDecomposition}
\end{figure}

\begin{remark}
Mumford's non-Archimedean analytic uniformization theory for degenerating Abelian varieties \cite{Mumford72}, as recently refined by Gubler and translated into the language of tropical geometry \cite{Gubler07,Gubler10}, shows that if $G$ is the dual graph of the special fiber of a regular $R$-model ${\mathcal C}$ for a curve $C$, then the canonical polyhedral decomposition $\{ C_T \}$ of $\Pic^g(\Gamma)$ gives rise to a canonical proper $R$-model for $\Pic^g(C)$.  Sam Payne has asked \cite{PaynePersonal} whether (up to the identification of $\Pic^g$ with $\Pic^0$) this model coincides with the compactification of the N{\'e}ron model of $\Jac(C)$ introduced by Caporaso in \cite{Caporaso08}.
\end{remark}

Break divisors corresponding to the relative interior of some cell $C_T$ are called {\em simple break divisors}.  They can be characterized as the set of degree $g$ effective divisors $D$ on $\Gamma$ such that $\Gamma \backslash {\rm supp}(D)$ is connected and simply connected.
Dhar's algorithm shows that such divisors are {\em universally reduced}, i.e., they are $q$-reduced for all $q \in \Gamma$.
A consequence of this observation and the Riemann-Roch theorem for metric graphs is the following result, which is useful in tropical Brill-Noether theory (cf. \S\ref{Subsection:GPThm}).

\begin{proposition}
\label{prop:SimpleBreakDivisors}
Let $\Gamma$ be a metric graph and let $D$ be a simple break divisor (or more generally any universally reduced divisor) on $\Gamma$.  Then $D$ has rank $0$ and $K_\Gamma - D$ has rank $-1$.  Therefore:
\begin{enumerate}
\item The set of divisor classes in $\Pic^g(\Gamma)$ having rank at least $1$ is contained in the codimension one skeleton of the polyhedral decomposition $\bigcup_T C_T$.
\item If $T$ is a spanning tree for some model $G$ of $\Gamma$ and $D,E$ are effective divisors with $D+E$ linearly equivalent to $K_\Gamma$, then there must be an open edge $e^\circ$ in the complement of $T$ such that $D$ has no chips on $e^\circ$.
\end{enumerate}
\end{proposition}

\begin{remark}
The set $B(\Gamma)$ of all break divisors on $\Gamma$ can be characterized as the topological closure in $\Div^g_+(\Gamma)$ of the set of universally reduced divisors.
\end{remark}

\begin{remark}
The real torus $\Pic^g(\Gamma)$ has a natural Riemannian metric.  One can compute the volume of $\Pic^g(\Gamma)$ in terms of a matrix determinant associated to $G$, and the volume of the cell $C_T$ is the product of the lengths of the edges not in $T$.  Comparing the volume of the torus to the sums of the volumes of the cells $C_T$ yields a {\em dual version} of a weighted form of Kirchhoff's Matrix-Tree Theorem.  See \cite{ABKS13} for details.
\end{remark}

\part{Advanced Topics}
\label{Part:Advanced}

We now turn to the more advanced topics of nonarchimedean analysis, tropical moduli spaces, and metrized complexes.  Each of these topics plays an important role in tropical Brill-Noether theory, and we would be remiss not to mention them here.  We note, however, that most of the applications we discuss in Part \ref{Part:Applications} do not require these techniques, and the casual reader may wish to skip this part on the first pass.

\section{Berkovich Analytic Theory}
\label{Section:Berkovich}

Rather than considering a curve over a discretely valued field and then examining its behavior under base change, we could instead start with a curve over an algebraically closed field and directly associate a metric graph to it.
We do this by making use of Berkovich's theory of analytic spaces.  In addition to being a convenient bookkeeping device for changes in dual graphs and specialization maps under field extensions, Berkovich's theory also allows for clean formulations of some essential results in the theory of tropical linear series, such as the Slope Formula
(Theorem~\ref{thm:SlopeFormula} below).
The theory also furnishes a wealth of powerful tools for understanding the relationship between algebraic curves and their tropicalizations.

\subsection{A quick introduction to Berkovich spaces}
\label{Subsection:BerkovichIntro}
We let $K$ be a field which is complete with respect to a non-Archimedean valuation
$$ \val : K^* \to \mathbb{R} . $$
We let $R \subset K$ be the valuation ring, $\kappa$ the residue field, and $| \cdot | = {\rm exp}(-{\rm val})$ the corresponding norm on $K$.

\medskip

A {\em multiplicative seminorm} on a nonzero ring $A$ is a function $| \cdot | : A \to \RR$ such that for all $x,y \in A$ we have $|0|=0, |1|=1, |xy|=|x| \cdot |y|$, and $|x+y| \leq |x| + |y|$.  A multiplicative seminorm is {\em non-Archimedean} if $|x+y| \leq \max \{ |x|,|y| \}$ for all $x,y \in A$, and is a {\em norm} if $|x|=0$ implies $x=0$.
If $L$ is a field, a function $| \cdot | : L \to \RR$ is a non-Archimedean norm if and only if $-\log | \cdot | : L \to \RR \cup \{ \infty \}$ is a {\em valuation} in the sense of Krull.

\medskip

If $X=\Spec(A)$ is an affine scheme over $K$, we define its {\em Berkovich analytification} $X^{\an}$ to be the set of non-Archimedean multiplicative seminorms on the $K$-algebra $A$ extending the given absolute value on $K$, endowed with the weakest topology such that the map $X^{\an} \to \RR$ defined by $| \cdot |_x \mapsto | f |_x$ is continuous for all $f \in A$.  One can globalize this construction to give a Berkovich analytification of an arbitrary scheme of finite type over $K$.  As we have defined it, the Berkovich analytification is merely a topological space, but it can be equipped with a structure similar to that of a locally ringed space and one can view $X^{\an}$ as an object in a larger category
of (not necessarily algebraizable) {\em Berkovich analytic spaces}.
The space $X^{\an}$ is locally compact and locally path-connected.
It is Hausdorff if and only if $X$ is separated, compact if and only if $X$ is proper, and path-connected if and only if $X$ is connected.
We refer the reader to \cite{Berkovich90,ConradAWS} for more background information on Berkovich spaces in general, and \cite{Baker08b,BPR13} for more details in the special case of curves.

\medskip

There is an alternate perspective on Berkovich spaces which is often useful and highlights the close analogy with schemes.  If $K$ is a field, points of an affine $K$-scheme $\Spec(A)$ can be identified with equivalence classes of pairs $(L,\phi)$ where $L$ is a field extension of $K$, $\phi : A \to L$ is a $K$-algebra homomorphism, and two pairs $(L_1,\phi_1)$ and $(L_2,\phi_2)$ are equivalent if there are embeddings of $L_1$ and $L_2$ into a common overfield $L'$ and a homomorphism $\phi': A \to L'$ such that the composition $\phi_i : A \to L_i \to L$ is $\phi'$.
Indeed, to a pair $(L,\phi)$ one can associate the prime ideal ${\rm ker}(\phi)$ of $A$, and to a prime ideal ${\mathfrak p}$ of $A$ one can associate the pair $(K({\mathfrak p}),\phi)$ where $K({\mathfrak p})$ is the fraction field of $A/{\mathfrak p}$ and $\phi : A \to K({\mathfrak p})$ is the canonical map.

\medskip

Similarly, if $K$ is a complete valued field, points of $\Spec(A)^{\an}$ can be identified with equivalence classes of pairs $(L,\phi)$, where $L$ is a complete valued field extension of $K$ and $\phi : A \to L$ is a $K$-algebra homomorphism.  The equivalence relation is as before, except that $L'$ should be complete and extend the valuation on the $L_i$.
Indeed, to a pair $(L,\phi)$ one can associate the multiplicative seminorm $a \mapsto |\phi(a)|$ on $A$, and to a multiplicative seminorm $| \cdot |_x$ one can associate the pair $(\mathcal{H}(x),\phi)$ where $\mathcal{H}(x)$ is the completion of the fraction field of $A/{\rm ker}(| \cdot |_x)$ and $\phi : A \to \mathcal{H}(x)$ is the canonical map.

\medskip

{\em We will assume for the rest of this section that $K$ is {\bf algebraically closed} and {\bf non-trivially valued}.}
This ensures, for example, that the set $X(K)$ is dense in $X^{\an}$.

\medskip

If $X/K$ is an irreducible variety, there is a dense subset of $X^{\an}$ consisting of the set ${\rm Val}_X$ of {\em norms}\footnote{Note that there is a one-to-one correspondence between norms on $K(X)$ and valuations on $K(X)$, hence the terminology ${\rm Val}_X$.  It is often convenient to work with (semi-)valuations rather than (semi-)norms.} on the function field $K(X)$ that extend the given norm on $K$.
Within the set ${\rm Val}_X$, there is a distinguished class of norms corresponding to {\em divisorial valuations}.  By definition, a valuation $v$ on $K(X)$ is {\em divisorial} if there is an $R$-model $\mathcal{X}$ for $X$ and an irreducible component $Z$ of the special fiber of $\mathcal{X}$ such that $v(f)$ is equal to the order of vanishing
of $f$ along $Z$.   The set of divisorial points is known to be dense in $X^{\an}$.

\begin{remark}
In this survey we have intentionally avoided the traditional perspective of tropical geometry, in which one considers subvarieties of the torus $(K^*)^n$, and the tropicalization is simply the image of coordinatewise valuation.  We refer the reader to \cite{MaclaganSturmfels} for a detailed account of this viewpoint on tropical geometry.  The Berkovich analytification can be thought of as a sort of intrinsic tropicalization -- one that does not depend on a choice of coordinates.  This is reinforced by the result that the Berkovich analytification is the inverse limit of all tropicalizations \cite{analytification, limits}.
\end{remark}

\subsection{Berkovich curves and their skeleta}
\label{section:Berkcurves}

If $C/K$ is a complete nonsingular curve, the underlying set of the Berkovich analytic space $C^{\an}$ consists of the points of $C(K)$ together with the set ${\rm Val}_C$.
We write $$ \val_y : K(C)^* \to \mathbb{R} $$
for the valuation corresponding to a point $y \in {\rm Val}_C = C^{\an} \smallsetminus C(K)$.
The points in $C(K)$ are called type-1 points, and the remaining points of $C^{\an}$ are classified into three more types.  We will not define points of type 3 or 4 in this survey article; see, e.g., \cite[\S{3.5}]{BPR11} for a definition.  Note, however, that every point of $C^\an$ becomes a type-1 point after base-changing to a suitably
large complete non-Archimedean field extension $L/K$.

If the residue field of $K(C)$ with respect to $\val_y$ has transcendence degree 1 over $\kappa$, then $y$ is called a {\em type-2 point}.  These are exactly the points corresponding to divisorial valuations.  Because it has transcendence degree 1 over $\kappa$, this residue field
corresponds to a unique smooth projective curve over $\kappa$, which we denote $C_y$.  A \emph{tangent direction} at $y$ is an equivalence class of continuous injections $\gamma: [0,1] \to C^{\an}$ sending $0$ to $y$, where $\gamma \sim \gamma'$ if $\gamma([0,1]) \cap \gamma'([0,1]) \supsetneq \{ y \}$.  Closed points of $C_y$ are in one-to-one correspondence with tangent directions at $y$ in $C^{\an}$.

\medskip

There is natural metric on the set ${\rm Val}_C$ which is described in detail in \cite[\S{5.3}]{BPR13}.  This metric induces a topology that is much finer than the subspace topology on ${\rm Val}_C \subset C^{\an}$, and with respect to this metric, ${\rm Val}_C$ is locally an {\em $\RR$-tree}\footnote{See \cite[Appendix B]{BakerRumely10} for an
introduction to the theory of $\RR$-trees.  For our purposes, what is most important about $\RR$-trees is that there is a unique path between any two points.} with branching precisely at the type-2 points.  The type-1 points should be thought of as
infinitely far away from every point of ${\rm Val}_C$.

\medskip

The local $\RR$-tree structure arises in the following way.  If $\cC$ is an $R$-model for $C$ and $Z$ is a reduced and irreducible irreducible component of the special fiber of $\cC$, then $Z$ corresponds to a type-2 point $y_Z$ of $\cC$.
Blowing up a nonsingular closed point of $Z$ (with respect to some choice of a uniformizer $\varpi \in {\mathfrak m}_R$) gives a new point $y_{Z'}$ of $C^{\an}$ corresponding to the exceptional divisor $Z'$ of the blowup.  We can then blow up a nonsingular closed point on the exceptional divisor $Z'$ to obtain a new point of $C^{\an}$, and so forth.
The resulting constellation of points obtained by all such sequences of blowups, and varying over all possible choices of $\varpi$, is an $\RR$-tree $T_Z$ rooted at $y_Z$, as pictured in Figure~\ref{fig:EllipticSkeleton}.  The distance between the points $y_Z$ and $y_{Z'}$ is ${\rm val}(\varpi)$.


\medskip

A \emph{semistable vertex set} is a finite set of type-2 points whose complement is a disjoint union of finitely many open annuli and infinitely many open balls.  There is a one-to-one correspondence between semistable vertex sets and semistable models of $C$.  More specifically, the normalized components of the central fiber of this semistable model are precisely the curves $C_y$ for $y$ in the semistable vertex set, and the preimages of the nodes under specialization are the annuli.  The annulus corresponding to a node where $C_y$ meets $C_{y'}$ contains a unique open segment with endpoints $y$ and $y'$, and its length (with respect to the natural metric on ${\rm Val}_C$) is the logarithmic modulus of the annulus.  The union of these open segments together with the semistable vertex set is a closed connected metric graph $\Gamma$ contained in $C^{\an}$, called the \emph{skeleton} of the semistable model ${\mathcal C}$.  If $C$ has genus at least 2, then there is a unique minimal semistable vertex set in $C^{\an}$ and a corresponding minimal skeleton.

\medskip

Fix a semistable model $\mathcal{C}$ of $C$ and a corresponding skeleton $\Gamma = \Gamma_{\mathcal C}$.
Each connected component of $C^\an \smallsetminus \Gamma$ has a unique boundary point in $\Gamma$, and there is a canonical retraction to the skeleton
\[
\tau: C^\an \rightarrow \Gamma
\]
taking a connected component of $C^\an \smallsetminus \Gamma$ to its boundary point.
There is a natural homeomorphism of topological spaces $C^\an \cong \varprojlim \Gamma_{\mathcal C}$, where the inverse limit is taken over all semistable models ${\mathcal C}$ (cf. \cite[Theorem 5.2]{BPR13}).

\begin{example}
Figure~\ref{fig:EllipticSkeleton} depicts the Berkovich analytification of an elliptic curve $E/K$ with non-integral $j$-invariant $j_E$.  In this case, the skeleton $\Gamma$ associated to a minimal proper semistable model of $R$ is isometric to a circle with circumference $-\val(j_E)$.
There are an infinite number of infinitely-branched $\RR$-trees emanating from the circle at each type-2 point of the skeleton.  The retraction map takes a point $x \in E^{\an}$ to the endpoint in $\Gamma$ of the unique path from $x$ to $\Gamma$.  The points of $E(K)$ lie ``out at infinity'' in the picture: they are ends of the $\RR$-trees.
\begin{figure}[h!]
\begin{center}
    \includegraphics[width=0.4\textwidth]{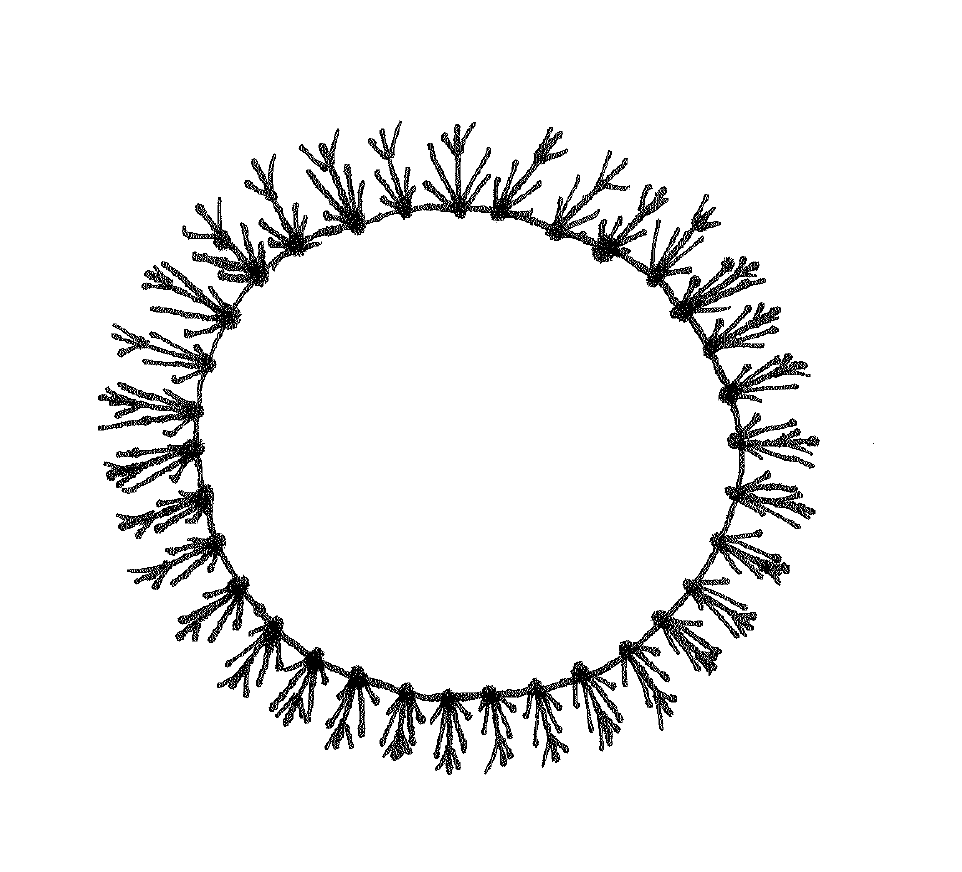}
\end{center}
  \caption{The skeleton of an elliptic curve with non-integral $j$-invariant.}
  \label{fig:EllipticSkeleton}
\end{figure}
\end{example}

\subsection{Tropicalization of divisors and functions on curves}
\label{Subsection:AnalyticTrop}

Restricting the retraction map $\tau$ to $C(K)$ and extending linearly gives the tropicalization map on divisors
\[
\Trop: \Div(C) \rightarrow \Div(\Gamma).
\]

If $K_0$ is a discretely valued subfield of $K$ over which $C$ is defined and has semistable reduction, and if $D$ is a divisor on $C$ whose support consists of $K_0$-rational points,
then the divisor ${\rm mdeg}(D)$ on $G$ (identified with a divisor on $\Gamma$ via the natural inclusion) coincides with the tropicalization $\Trop(D)$ defined via retraction to the skeleton.

\begin{example}
\label{Ex:Conics2}
Returning to Example \ref{Ex:FamilyOfConics}, in which the metric graph $\Gamma$ is a closed line segment of length 1, by considering the divisor cut out by $y^a = x^bz^{a-b}$ for positive integers $a>b$ we see that a divisor on $C_{\overline{K}}$ can tropicalize to any rational point on $\Gamma$.
\end{example}

Given a rational function $f \in K(C)^*$, we write $\trop(f)$ for the real valued function on the skeleton $\Gamma$ given by $y \mapsto \val_y(f)$.  The function $\trop(f)$ is piecewise linear with integer slopes, and thus we obtain a map
\[
\trop: K(C)^* \rightarrow \PL(\Gamma) .
\]

Moreover, this map respects linear equivalence of divisors, in the sense that if $D \sim D'$ on $C$ then $\trop(D) \sim \trop(D')$ on $\Gamma$.  In particular, the tropicalization map on divisors descends to a natural map on Picard groups
\[
\Trop: \Pic(C) \rightarrow \Pic(\Gamma).
\]

One can refine this observation as follows.  Let $x$ be a type-2 point in $C^{\rm an}$.  Given a nonzero rational function $f$ on $C$, one can define its {\em normalized reduction} $\bar{f}_x$ with respect to $x$ as follows.  Choose $c \in K^*$ such that $|f|_x=|c|$.  Define $\bar{f}_x \in \kappa(C_x)^*$ to be the image of $c^{-1}f$ in the residue field of $K(C)$ with respect to $\val_x$, which by definition is isomorphic to $\kappa(C_x)$.
Although $f_x$ is only well-defined up to multiplication by an element of $\kappa^*$, its divisor ${\rm div}(f_x)$ is completely well-defined.  We define the normalized reduction of the zero function to be zero.

\medskip

Given an $(r+1)$-dimensional $K$-vector space $H \subset K(C)$, its reduction $\bar{H}_x = \{ \bar{f}_x \; | \; f \in H \}$ is an $(r+1)$-dimensional vector space over $\kappa$.  Given $\tilde{f}$ in the function field of $C_x$ and a closed point $\nu$ of $C_x$, we let $s^\nu(\tilde{f}) := \ord_\nu(\tilde{f})$ be the order of vanishing of $\tilde{f}$ at $\nu$.  If $\tilde{f} = \bar{f}_x$ for $f \in K(C)^*$, then $s^\nu(\tilde{f})$ is equal to the slope of $\trop(f)$ in the tangent direction at $x$ corresponding to $\nu$.  This is a consequence of the nonarchimedean Poincare-Lelong formula, due to Thuiller \cite{ThuillierThesis}.
Using this observation, one deduces the following important result (cf. \cite[Theorem 5.15]{BPR13}):
\begin{theorem}[Slope Formula]
\label{thm:SlopeFormula}
For any nonzero rational function $f \in K(C)$,
\[
\Trop(\ddiv(f)) = \ddiv(\trop(f)).
\]
\end{theorem}

\subsection{Skeletons of higher-dimensional Berkovich spaces}
\label{Subsection:higherdimensionalskeletons}

The construction of the skeleton of a semistable model of a curve given in \S\ref{section:Berkcurves} can be generalized in various ways to higher dimensions.  For brevity, we mention just three such generalizations.  In what follows, $X$ will denote a proper variety of dimension $n$ over $K$.

\medskip

1. {\em Semistable models.} Suppose $\mathcal X$ is a strictly semistable model of $X$ over $R$.
Then the geometric realization of the dual complex $\Delta(\overline{\mathcal X})$ of the special fiber embeds naturally in the Berkovich analytic space $X^{\an}$,
and as in the case of curves there is a strong deformation retraction of $X^{\an}$ onto $\Delta(\overline{\mathcal X})$.  These facts are special cases of results due to Berkovich; see \cite{Nicaise11} for a lucid explanation of the constructions in the special case of strictly semistable models,
and \cite{GRW14} for a generalization to ``extended skeleta''.

\medskip

2. {\em Toroidal embeddings.} A \emph{toroidal embedding} is, roughly speaking, something which looks {\'e}tale-locally like a toric variety together with its dense big open torus.  When $K$ is trivially valued, Thuillier \cite{Thuillier07} associates a skeleton $\Sigma(X)$ of $U$ and an extended skeleton $\overline{\Sigma}(X)$ of $X$ to any toroidal embedding $U \subset X$ (see also \cite{acp}).
As in the case of semistable models, the skeleton
$\overline{\Sigma}(X)$ embeds naturally into $X^{\an}$ and there is a strong deformation retract $X^{\an} \to \overline{\Sigma}(X)$.

\medskip

3. {\em Abelian varieties.} If $E/K$ is an elliptic curve with non-integral $j$-invariant, the skeleton associated to a minimal proper semistable model of $R$ can also be constructed using Tate's non-Archimedean uniformization theory.
In this case, the skeleton of $E^{\an}$ is the quotient of the skeleton of ${\mathbf G}_m$ (which is isomorphic to $\RR$ and
consists of the unique path from $0$ to $\infty$ in $(\PP^1)^{\an}$) by the map $x \mapsto x - \val(j_E)$.  Using Mumford's higher-dimensional generalization of Tate's theory \cite{Mumford72}, one can define a skeleton associated to any totally degenerate abelian variety $A$;
it is a real torus of dimension $\dim(A)$.
This can be generalized further using Raynaud's uniformization theory to define a canonical notion of skeleton for an arbitrary abelian variety $A/K$ (see e.g. \cite{Gubler10}).
If $A$ is principally polarized, there is an induced {\em tropical principal polarization} on the skeleton of $A$, see \cite[\S{3.7}]{BakerRabinoff13} for a definition.
It is shown in \cite{BakerRabinoff13} that the skeleton of the Jacobian of a curve $C$ is isomorphic to the Jacobian of the skeleton as principally polarized real tori:

\begin{theorem}[\cite{BakerRabinoff13}]
\label{Thm:JacobianSkeleton}
Let $C$ be a curve over an algebraically closed field, complete with respect to a nontrivial valuation, such that the minimal skeleton of the Berkovich analytic space $C^{\an}$ is isometric to $\Gamma$.  Then there is a canonical isomorphism of principally polarized real tori $\Jac (\Gamma ) \cong \Sigma ( \Jac (C)^{\an} )$ making the following diagram commute.
\begin{displaymath}
\xymatrix{
C^{\an} \ar[r]^{AJ} \ar[d] & \Jac (C)^{\an} \ar[rd] & \\
\Gamma \ar[r]^{AJ} & \Jac ( \Gamma ) \ar[r]^{\sim} & \Sigma ( \Jac (C)^{\an} )  .  }
 \end{displaymath}
\end{theorem}

\begin{remark} \label{rem:trop.of.aj}
Theorem~\ref{Thm:JacobianSkeleton} has the following interpretation in terms of tropical moduli spaces, which we discuss in greater detail in \S \ref{Section:ModuliSpaces}.  There is a map
$$ \trop: M_g \to M_{g}^{\trop} $$
from the moduli space of genus $g \geq 2$ curves to the moduli space of tropical curves of genus $g$ which takes a curve $C$ to its minimal skeleton, considered as a vertex-weighted metric graph.  There is also a map (of sets, for example)
$$ \trop: A_g \to A_{g}^{\trop} $$
from the moduli space of principally polarized abelian varieties of dimension $g$ to the moduli space of ``principally polarized tropical abelian varieties'' of dimension $g$, taking an abelian variety to its skeleton in the sense of Berkovich.
Finally, there are Torelli maps $M_g \to A_g$ (resp. $M_g^{\trop} \to A_g^{\trop}$) which take a curve (resp.\ metric graph) to its Jacobian \cite{BrannettiMeloViviani11}.
Theorem~\ref{Thm:JacobianSkeleton} implies that the following square commutes:
  \[\xymatrix @R=.25in{
    {M_g} \ar[r]^(.45){\trop} \ar[d] & {M_g^{\trop}} \ar[d] \\
    {A_g} \ar[r]^(.45){\trop} & {A_g^{\trop}}
  }\]
This is also proved, with slightly different hypotheses, in \cite[Theorem A]{Viviani13}.
  \end{remark}

\section{Moduli Spaces}
\label{Section:ModuliSpaces}

\subsection{Moduli of tropical curves}
\label{Subsection:Moduli}

The moduli space of tropical curves $M_g^{\trop}$ has been constructed by numerous authors \cite{GathmannKerberMarkwig09, Kozlov09, Caporaso11, acp}.  In this section, we give a brief description of this object, with an emphasis on applications to classical algebraic geometry.

\medskip

Given a finite vertex-weighted graph ${\mathbf G} = (G,\omega)$ in the sense of \S\ref{Subsection:vertexweight}, the set of all vertex-weighted metric graphs $(\Gamma,\omega)$ with underlying finite graph $G$ is naturally identified with
$$ M_{\mathbf G}^{\trop} := \mathbb{R}_{>0}^{\vert E(G) \vert} / \mathrm{Aut}({\mathbf G})  $$
with the Euclidean topology.  If ${\mathbf G}'$ is obtained from ${\mathbf G}$ by contracting an edge, then we may think of a metric graph in $M_{{\mathbf G}'}^{\trop}$ as a limit of graphs in $M_{\mathbf G}^{\trop}$ in which the length of the given edge approaches zero.
Similarly, if ${\mathbf G}'$ is obtained from $v$ by contracting a cycle to a vertex $v$ and augmenting the weight of $v$ by one, we may think of a metric graph in $M_{{\mathbf G}'}^{\trop}$ as a limit of graphs in $M_{\mathbf G}^{\trop}$.
In this way, we may construct the {\em moduli space of tropical curves}
$$ M_g^{\trop} :=\bigsqcup M_{\mathbf G}^{\trop} , $$
where the union is over all {\em stable}\footnote{A vertex-weighted finite graph $(G,\omega)$ is called {\em stable} if every vertex of weight zero has valence at least 3.} vertex-weighted graphs ${\mathbf G}$ of genus $g$, and the topology is induced by gluing $M_{{\mathbf G}'}^{\trop}$ to the boundary of $M_{\mathbf G}^{\trop}$ whenever
${\mathbf G}'$ is a contraction of ${\mathbf G}$ in one of the two senses above.

\medskip

We note that the moduli space $M_g^{\trop}$ is not compact, since edge lengths in a metric graph must be finite and thus there is no limit if we let some edge length tend to infinity.  There exists a compactification $\overline{M}_g^{\trop}$, known as the moduli space of {\em extended tropical curves}, which parameterizes vertex-weighted metric graphs with possibly infinite edges; we refer to \cite{acp} for details.

\medskip

Let $M_g$ be the (coarse) moduli space of genus $g$ curves and $\overline{M}_g$ its Deligne-Mumford compactification, considered as varieties over $\CC$ endowed with the trivial valuation.
Points of $M_g^{\an}$ can be identified with equivalence classes of points of $M_g(L)$, where $L$ is a complete non-archimedean field extension of $\CC$ (with possibly non-trivial valuation).
There is a natural map ${\rm Trop}: M_g^{\an} \to M_g^{\rm trop}$ which on the level of $L$-points takes a smooth proper genus $g$ curve $C/L$ to the minimal skeleton of its Berkovich analytification $C^{\an}$.
This map extends naturally to a map ${\rm Trop}: \overline{M}_g^{\an} \to \overline{M}_g^{\rm trop}$.

\medskip

Let $\Sigma(\overline{M}_g)$ (resp. $\overline{\Sigma}(\overline{M}_g)$) denote the skeleton, in the sense of Thuillier, of $M_g^{\an}$ (resp. $\overline{M}_g^{\an}$) with respect to the natural toroidal structure coming from the boundary strata of $\overline{M}_g \smallsetminus M_g$.
According to the main result of \cite{acp}, there is a very close connection between the moduli space of tropical curves $M_g^{\trop}$ and the Thuillier skeleton $\Sigma(\overline{M}_g)$:

\begin{theorem}[\cite{acp}]
\label{Thm:Moduli}
There is a canonical homeomorphism\footnote{This homeomorphism is in fact an isomorphism of ``generalized cone complexes with integral structure'' in the sense of \cite{acp}.}
$$ \Phi : \Sigma (\overline{M}_g^{\an}) \to M_g^{\trop}$$
which extends uniquely to a map
$$ \overline{\Phi} : \overline{\Sigma} (\overline{M}_g^{\an}) \to \overline{M}_g^{\trop}$$
of compactifications in such a way that
 \begin{displaymath}
\xymatrix{
\overline{M}_g^{\an} \ar[r]^{\overline{P}} \ar[rd]^{\Trop} & \overline{\Sigma} (\overline{M}_g^{\an} ) \ar[d]^{\overline{\Phi}} \\
 & \overline{M}_g^{\trop} }
 \end{displaymath}
commutes, where $\overline{P} : \overline{M}_g^{\an} \to \overline{\Sigma} (\overline{M}_g^{\an} )$ is the canonical deformation retraction.
\end{theorem}

It follows from Theorem~\ref{Thm:Moduli} that the map ${\rm Trop}: \overline{M}_g^{\an} \to\overline{M}_g^{\trop}$ is continuous, proper, and surjective.  From this, one easily deduces:

\begin{corollary}
\label{Cor:Surjective}
Let $K$ be a complete and algebraically closed non-Archimedean field with value group $\RR$, and let $\Gamma$ be a stable metric graph of genus at least 2.  Then there exists a curve $C$ over $K$ such that the minimal skeleton of the Berkovich analytic space $C^{\an}$ is isometric to $\Gamma$.
\end{corollary}

\begin{remark}
\label{Remark:Deformation}
A more direct proof of Corollary~\ref{Cor:Surjective}, which in fact proves a stronger statement replacing $\Gamma$ with an arbitrary metrized complex of curves, and the field $K$ with any complete and algebraically closed non-Archimedean field whose value group contains all edge lengths in some model for $\Gamma$, can be found in Theorem 3.24 of \cite{ABBR14a}.  The proof uses formal and rigid geometry.
A variant of Corollary~\ref{Cor:Surjective} for discretely valued fields, proved using deformation theory, can be found in Appendix B of \cite{Baker08}.
\end{remark}

\begin{remark}
\label{Remark:Deformation2}
Let $R$ be a complete DVR with field of fractions $K$ and infinite residue field $\kappa$.  The argument in Appendix B of \cite{Baker08} shows that
for any finite connected graph $G$, there exists a regular, proper, flat curve $\mathcal{C}$ over $R$ whose generic fiber is smooth and whose special fiber is a maximally degenerate semistable curve with dual graph $G$.
One should note the assumption here that $\kappa$ has infinite residue field.  In the case where the residue field is finite -- for example, when $K = \mathbb{Q}_p$ -- the question of which graphs arise in this way remains an open problem.  The significance of this problem is its relation to the rational points of the moduli space of curves.
For example, the existence of Brill-Noether general curves defined over $\mathbb{Q}$ for large $g$ is a well-known open question.  Lang's Conjecture predicts that, for large $g$, such curves should be contained in a proper closed subset of the moduli space of curves.  One suggested candidate for this closed subset is the stable base locus of the canonical bundle, which is known to contain only Brill-Noether special curves.
\end{remark}

\subsection{Brill-Noether rank}
\label{Subsection:BNRank}

The motivating problem of Brill-Noether theory is to describe the variety $W^r_d (C)$ parameterizing divisors of a given degree and rank on a curve $C$.  A first step in such a description should be to compute numerical invariants of $W^r_d (C)$, such as its dimension.  Our goal is to use the combinatorics of the dual graph $\Gamma$ to describe $W^r_d ( \Gamma )$.  Combining this combinatorial description with the Specialization Theorem, we can then hope to understand the Brill-Noether locus of our original curve.  One might be tempted to think that the tropical analogue of $\dim W^r_d (C)$ should be $\dim W^r_d ( \Gamma )$, but as in the case of linear series, the dimension is not a well-behaved tropical invariant.  We note one example of such poor behavior.

\begin{example}
\label{Ex:NotSemicontinuous}
In \cite[Theorem 1.1]{LPP12}, it is shown that the function that takes a metric graph $\Gamma$ to $\dim W^r_d ( \Gamma )$ is not upper semicontinuous on $M_g^{\trop}$.  To see this, the authors construct the following example.  Let $\Gamma$ be the loop of loops of genus 4 depicted in Figure \ref{Fig:NotSemicontinuous}, with edges of length $\ell_1 < \ell_2 < \ell_3$ as pictured.  Suppose that $\ell_1 + \ell_2 > \ell_3$.  Then $\dim W^1_3 (\Gamma ) = 1$.  If, however, we consider the limiting metric graph $\Gamma_0$ as $\ell_1 , \ell_2$ and $\ell_3$ approach zero, then on this graph the only divisor of degree 3 and rank 1 is the sum of the three vertices, hence $\dim W^1_3 (\Gamma_0 ) = 0$.

\begin{figure}
\begin{tikzpicture}

\draw [ball color=black] (-1,3) circle (0.55mm);
\draw [ball color=black] (1,3) circle (0.55mm);
\draw (0,3) circle (1);
\draw [ball color=black] (3,1) circle (0.55mm);
\draw [ball color=black] (3,-1) circle (0.55mm);
\draw (3,0) circle (1);
\draw [ball color=black] (-3,1) circle (0.55mm);
\draw [ball color=black] (-3,-1) circle (0.55mm);
\draw (-3,0) circle (1);
\draw (1,3)--(3,1);
\draw (-1,3)--(-3,1);
\draw (3,-1)--(-3,-1);
\draw (2.25,2.25) node {$\ell_1$};
\draw (-2.25,2.25) node {$\ell_2$};
\draw (0,-1.25) node {$\ell_3$};
\draw (-3,-1.25) node {$w_3$};
\draw (3,-1.25) node {$v_3$};
\draw (-3,1.25) node {$v_2$};
\draw (3,1.25) node {$w_1$};
\draw (-1.25,3) node {$w_2$};
\draw (1.25,3) node {$v_1$};

\end{tikzpicture}
\caption{The metric graph $\Gamma$ from \cite{LPP12}.}
\label{Fig:NotSemicontinuous}
\end{figure}
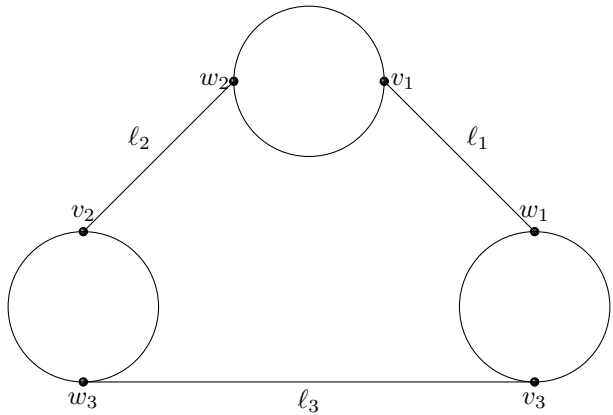
\end{example}

The solution to this problem has a very similar flavor to the definition of rank recorded in Definition \ref{Def:Rank}.  Specifically, given a curve $C$, consider the incidence correspondence
$$ \Phi = \{ (p_1 , \ldots , p_d , D) \in C^d \times W^r_d (C) \; \vert \; p_1 + \cdots + p_d \in \vert D \vert \} . $$
The forgetful map to $W^r_d (C)$ has fibers of dimension $r$, so $\dim \Phi = r + \dim W^r_d (C)$, and hence the image of $\Phi$ in $C^d$ has the same dimension.  This suggests the following surrogate for the dimension of $W^r_d (C)$.

\begin{definition}
\label{Def:BNRank}
Let $\Gamma$ be a metric graph, and suppose that $W^r_d (\Gamma)$ is nonempty.  The \emph{Brill-Noether rank} $w^r_d (\Gamma)$ is the largest integer $k$ such that, for every effective divisor $E$ of degree $r + k$, there exists a divisor $D \in W^r_d (\Gamma)$ such that $\vert D-E \vert \neq \emptyset$.
\end{definition}

\begin{example}
\label{Ex:BNRank}
Note that, in the previous example, although $\dim W^1_3 (\Gamma) = 1$, the Brill-Noether rank $w^1_3 (\Gamma) = 0$.  To see this, it suffices to find a pair of points such that no divisor of degree 3 and rank 1 passes through both points simultaneously.  Indeed, it is shown in \cite[Theorem 1.9]{LPP12} that no divisor of rank 1 and degree 3 contains $v_3 + w_3$.
\end{example}

The Brill-Noether rank is much better behaved than the dimension of the Brill-Noether locus; for example (cf. \cite[Theorem 1.6]{LPP12} and \cite[Theorem 5.4]{Len12}):

\begin{theorem}
\label{Thm:Semicontinuous}
The Brill-Noether rank is upper semicontinuous on the moduli space of tropical curves.
\end{theorem}

The Brill-Noether rank also satisfies the following analogue of the Specialization Theorem (cf. \cite[Theorem 1.7]{LPP12} and \cite[Theorem 5.7]{Len12}):

\begin{theorem}
\label{Thm:BNRankSpecialization}
If $C$ is a curve over an algebraically closed field $K$ with nontrivial valuation, and the skeleton of the Berkovich analytic space $C^{\an}$ is isometric to $\Gamma$, then
$$ \dim W^r_d (C) \leq w^r_d (\Gamma ) . $$
\end{theorem}

We note the following generalization of Theorem \ref{Thm:WrdNonempty}.

\begin{corollary}
\label{Cor:wrd}
Let $\Gamma$ be a metric graph of genus $g$.  Then $w^r_d (\Gamma) \geq \rho := g-(r+1)(g-d+r)$.
\end{corollary}

\begin{proof}
The general theory of determinantal varieties shows that, if $W^r_d (C)$ is nonempty, then its dimension is at least $\rho$.
The result then follows from \cite{Kempf71, KleimanLaksov74} and Theorem \ref{Thm:BNRankSpecialization}.
\end{proof}

\begin{remark}
\label{Remark:LocalDim}
It is unknown whether $W^r_d (\Gamma)$ must have local dimension at least $\rho$.  Note, however, that this must hold in a neighborhood of a divisor $D \in \Trop W^r_d (C)$.  Hence, if $W^r_d (\Gamma)$ has smaller than the expected local dimension in a neighborhood of some divisor $D$, then $D$ does not lift to a divisor of rank $r$
on a curve $C$ having $\Gamma$ as its tropicalization.
\end{remark}

\section{Metrized Complexes of Curves and Limit Linear Series}
\label{Section:MetrizedComplexes}

In this section we describe the work of Amini and Baker \cite{AminiBaker12} on the Riemann--Roch and Specialization Theorems for divisors on metrized complexes of curves, along with applications to the theory of limit linear series.

\subsection{Metrized complexes of curves}

Metrized complexes of curves can be thought of, loosely, as objects which interpolate between classical and tropical algebraic geometry.
More precisely, a {\em metrized complex of algebraic curves} over an algebraically closed field $\kappa$ is a finite metric graph $\Gamma$ together with a fixed model $G$ and a collection of marked complete nonsingular algebraic curves $C_v$ over $\kappa$, one for each vertex $v$ of $G$; the set ${\mathcal A}_v$ of marked points on $C_v$ is in bijection with the edges of $G$ incident to $v$.  A metrized complex over $\CC$ can be visualized as a collection of compact Riemann surfaces connected together via real line segments, as in Figure~\ref{fig:mcsmall}.

\begin{figure}[!h]
\includegraphics[width=6cm]{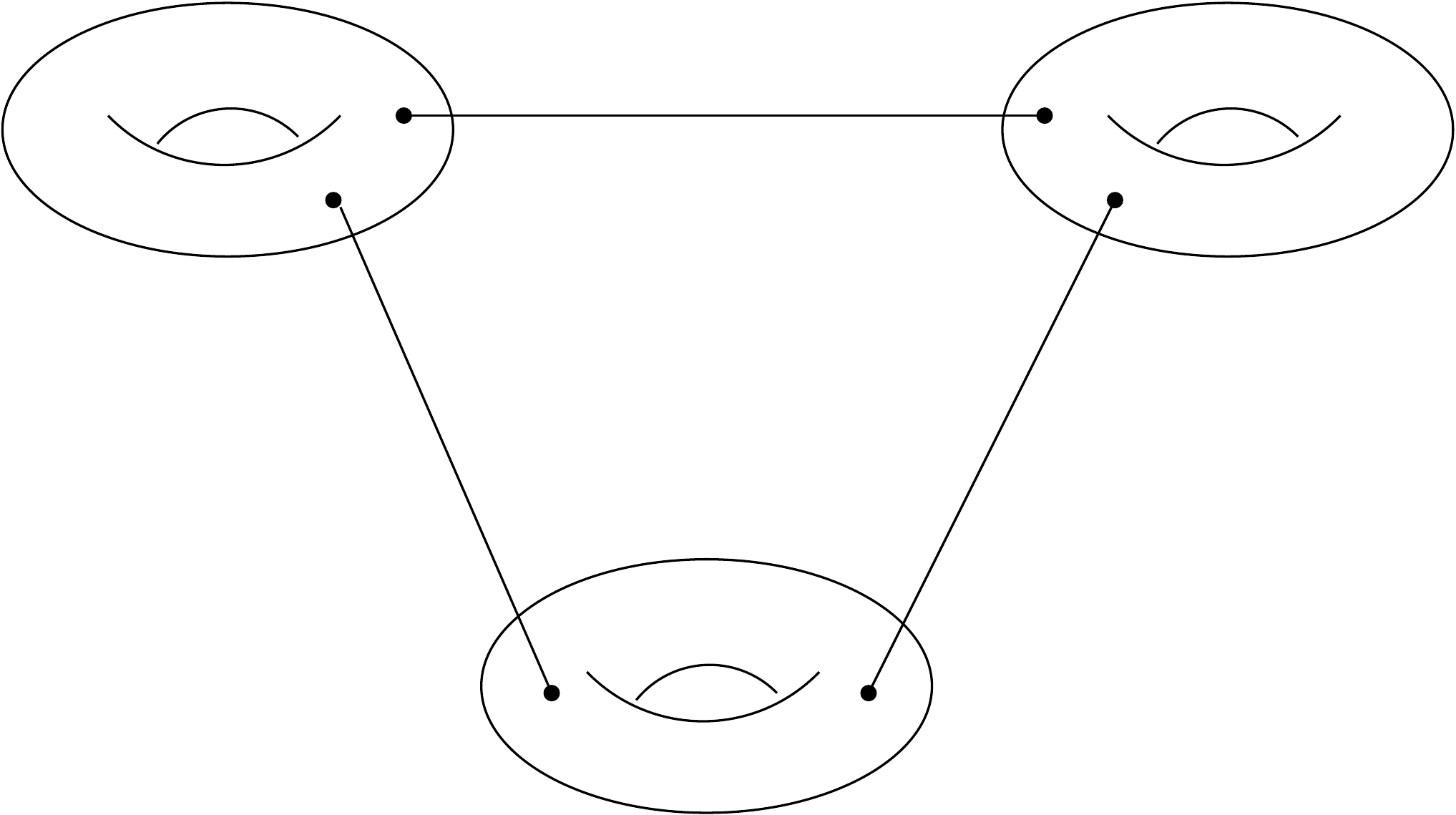}
\caption{An Example of a Metrized Complex}
 \label{fig:mcsmall}
 \end{figure}

\medskip

The {\em geometric realization} $|\fC|$ of a metrized complex of curves $\fC$ is defined as the topological space given by the union of the edges of $G$
and the collection of curves $C_v$, with each endpoint $v$ of an edge $e$ identified with the corresponding marked point $x^e_v$ (as suggested by Figure~\ref{fig:mcsmall}).
The {\em genus} of a metrized complex of curves $\fC$, denoted $g(\fC)$, is by definition
$g(\fC)=g(\Gamma)+\sum_{v\in V}g_v$, where $g_v$ is the genus of $C_v$ and $g(\Gamma)$ is the first Betti number of $\Gamma$.


\medskip

A {\em divisor} on a metrized complex of curves $\fC$ is an element $\mathcal D$ of the free abelian group on $|\fC|$.
Thus a divisor on $\fC$ can be written uniquely as $\mathcal D = \sum_{x \in |\fC|} a_x \, x$ where $a_x \in \ZZ$, all but finitely many of the $a_x$ are zero,
and the sum is over all points of $\Gamma \backslash V$ as well as $C_v(\kappa)$ for $v \in V$.
The {\em degree} of $\mathcal D$ is defined to be $\sum a_x$.

\medskip


A {\em nonzero rational function} $\f$ on a metrized complex of curves $\fC$ is the data of a rational function $f_\Gamma \in \PL (\Gamma)$ and nonzero rational functions $f_v$ on $C_v$ for each $v\in V$.
We call $f_\Gamma$ the {\em $\Gamma$-part} of $\f$ and $f_v$ the {\em $C_v$-part} of $\f$.
The {\em divisor} of a nonzero rational function $\f$ on $\fC$ is defined to be
\[
{\rm div}(\f) := \sum_{x \in |\fC|} \ord_x(\f) \, x,
\]
where $\ord_x(\f)$ is defined as follows\footnote{Note that our sign convention here for the divisor of a rational function on $\Gamma$, which coincides with the one used in \cite{AminiBaker12}, is the opposite of the sign convention used in Section \ref{sec:MetricGraphDivisors}, which is also used in a number of other papers in the subject.
This should not cause any confusion, but it is good for the reader to be aware of this variability when perusing the literature.}:
\begin{itemize}
\item If $x \in \Gamma \backslash V$, then $\ord_x(\f) = \ord_x(f_\Gamma)$, where
$\ord_x(f_\Gamma)$ is the sum of the slopes of $f_\Gamma$ in all tangent directions emanating from $x$.
\item If $x \in C_v(\kappa) \backslash \mathcal A_v$, then $\ord_x(\f) = \ord_x(f_v)$.
\item If $x = x^e_v \in \mathcal A_v$, then $\ord_x(\f) = \ord_x(f_v) + \mathrm{slp}_e(f_\Gamma)$, where
$\mathrm{slp}_e(f_\Gamma)$ is the outgoing slope of $f_\Gamma$ at $v$ in the direction of $e$.
\end{itemize}


Divisors of the form ${\rm div}(\f)$ are called {\em principal}, and the principal divisors form a subgroup of
$\Div^0(\fC)$, the group of divisors of degree zero on $\fC$.
Two divisors in $\Div(\fC)$ are called {\em linearly equivalent} if they differ by a principal divisor.
Linear equivalence of divisors on $\fC$ can be understood rather intuitively in terms of ``chip-firing moves'' on $\fC$.  We refer the reader to \S{1.2} of \cite{AminiBaker12} for details.

\medskip



A divisor $\mathcal E = \sum_{x \in |\fC|} {a_x (x)}$ on $\fC$ is called {\em effective} if $a_x \geq 0$ for all $x$.
The {\em rank} $r_\fC$ of a divisor $\mathcal D \in \Div(\fC)$ is defined to
be the largest integer $k$ such that $\mathcal D - \mathcal E$ is linearly equivalent to an effective divisor
for all effective divisors $\mathcal E$ of degree $k$ on $\fC$ (so in particular
$r_{\fC}(\mathcal D) \geq 0$ if and only if $\mathcal D$ is linearly equivalent to
an effective divisor, and otherwise $r_{\fC}(\mathcal D)=-1$).

\medskip

The theory of divisors, linear equivalence, and ranks on metrized complexes of curves generalizes both the classical theory
for algebraic curves and the corresponding theory for metric graphs.
The former corresponds to the case where $G$ consists of a single vertex $v$ and no edges and $C=C_v$ is an arbitrary smooth curve.  The latter corresponds to the case where the curves $C_v$ have genus zero for all $v \in V$.
Since any two points on a curve of genus zero are linearly equivalent, it is easy to see that the divisor theories and rank functions on $\fC$ and $\Gamma$ are essentially the same.

\medskip


The {\em canonical divisor} on $\fC$ is
$$ \mathcal{K} = \sum_{v \in V} (K_v + \sum_{w \in \mathcal{A}_v} w) , $$
 where $K_v$ is a canonical divisor on $C_v$.

\medskip

The following result generalizes both the classical Riemann-Roch theorem for algebraic curves and the Riemann-Roch theorem for metric graphs:

\begin{MetrizedRR}\label{thm:RR-metrizedcomplexes}
Let $\fC$ be a metrized complex of algebraic curves over $\kappa$.  For any divisor $\mathcal D \in \Div(\fC)$, we have
\[r_\fC(\mathcal D) - r_\fC(\mathcal K -\mathcal D) = \deg(\mathcal D) - g(\fC)+1. \]
\end{MetrizedRR}

As with the tropical Riemann-Roch theorem, the proof of this theorem makes use of a suitable notion of {\em reduced divisors}
for metrized complexes of curves.  We note that the proof of the Riemann-Roch theorem for metrized complexes uses the Riemann-Roch theorem for algebraic curves and does not furnish a new proof of that result.

\subsection{Specialization of divisors from curves to metrized complexes}

Let $K$ be a complete and algebraically closed non-Archimedean field with valuation ring $R$ and residue field $\kappa$, and let $C$ be a smooth proper curve over $K$.  As in \S\ref{Section:Berkovich}, there is a metrized complex $\fC$ canonically associated to any strongly semistable model ${\cC}$ of $C$ over $R$.
The specialization map $\Trop$ defined in Sections \ref{Section:Finite} and \ref{Section:Metric} can be enhanced in a canonical way to a map from divisors on $C$ to divisors on $\fC$.
The enhanced specialization map, which by abuse of terminology we continue to denote by $\Trop$, is obtained by linearly extending a map $\tau : C(K) \to |\fC|$.  The map $\tau$ is defined as follows:
\begin{itemize}
\item For $P \in C(K)$ reducing to a smooth point ${\rm red}(P)$ of the special fiber $C_0$ of $\cC$, $\tau(P)$ is just the point ${\rm red}(P)$.
\item For $P \in C(K)$ reducing to a singular point, $\tau(P)$ is the point ${\Trop (P)}$ in the relative interior of the corresponding edge of the skeleton $\Gamma$ of ${\cC}$.
\end{itemize}

\medskip

The motivation for the definitions of $\Trop : C(K) \to |\fC|$ and ${\rm div}(\f)$ come in part from the following extension of the Slope Formula (Theorem~\ref{thm:SlopeFormula}):

\begin{proposition}
Let $f$ be a nonzero rational function
on $C$ and let $\f$ be the corresponding nonzero rational function on $\fC$, where $f_{\Gamma}$ is the restriction
to $\Gamma$ of the piecewise linear function $\log|f|$ on $C^{\an}$ and $f_v \in \kappa(C_v)$ for $v \in V$ is the
normalized reduction of $f$ to $C_v$ (cf. \S\ref{Subsection:BerkovichIntro}).
Then
\[
\Trop ({\rm div}(f)) ={\rm div}(\f).
\]
\end{proposition}

In particular, we have $\Trop({\rm Prin}(C)) \subseteq {\rm Prin}(\fC)$.

\medskip

The Specialization Theorem from \S \ref{Subsection:Specialization} generalizes to metrized complexes as follows:

\begin{MetrizedSpecialization}
\label{thm:MCSpecialization}
For all $D \in \Div(C)$, we have
\[
r_C(D) \leq r_{\fC}({\rm trop}(D)).
\]
\end{MetrizedSpecialization}

Since $r_{\fC}(\trop(D)) \leq r_\Gamma(\trop(D))$, the specialization theorem for metrized complexes is a strengthening of the analogous specialization result for metrized graphs.
In conjunction with a simple combinatorial argument, this theorem also refines the Specialization Theorem for vertex-weighted graphs 

\medskip

A simple consequence of the Riemann-Roch and Specialization Theorems for metrized complexes is that
for any canonical divisor $K_C$ on $C$, the divisor $\trop(K_C)$ belongs to the canonical class on $\fC$.
Indeed, the Specialization Theorem shows that $r_{\fC}(\trop(K_C)) \geq g-1$, while Riemann-Roch shows that a divisor of degree $2g-2$ and rank at least $g-1$ must be equivalent to $\mathcal{K}$.

\medskip

There is also a version of specialization in which one has {\em equality} rather than just an inequality.
One can naturally associate to a rank $r$ divisor $D$ on $C$ a collection  ${\mathcal H} = \{ H_v \}_{v \in V}$ of $(r+1)$-dimensional subspaces of $\kappa(C_v)$, where $H_v$ is the normalized reduction of $\mathcal L(D)$ to $C_v$ (cf.~\S\ref{Section:Berkovich}).
If ${\mathcal F} = \{ F_v \}_{v \in V}$, where $F_v$ is any $\kappa$-subspace of the function field $\kappa(C_v)$,
then for $\cD \in \Div(\fC)$ we define the {\em ${\mathcal F}$-restricted rank} of $\cD$, denoted $r_{\fC,{\mathcal F}}(\cD)$, to be the largest integer $k$
such that for any effective divisor ${\mathcal E}$ of degree $k$ on $\fC$, there is a rational function $\f$ on $\fC$ whose $C_v$-parts
$f_v$ belong to $F_v$ for all $v \in V$, and such that $\cD - {\mathcal E} + {\rm div}(\f) \geq 0$.

\begin{theorem}[Specialization Theorem for Restricted Ranks]
\label{thm:RestrictedSpecialization}
With notation as above, the ${\mathcal H}$-restricted rank of the specialization of $D$ is equal to the rank of $D$, i.e.,
$r_{\fC,{\mathcal H}}(\trop(D)) = r$.
\end{theorem}


\subsection{Connections with the theory of limit linear series}

The theory of linear series on metrized complexes of curves has close connections with the Eisenbud-Harris theory
of limit linear series for strongly semistable curves of compact type, and allows one to generalize the basic definitions in the Eisenbud-Harris theory to more general semistable curves.
The Eisenbud--Harris theory, which they used to settle a number of longstanding open problems in the theory of algebraic curves, only applies to a rather restricted class of reducible curves, namely those of {\em compact type} (i.e., nodal curves whose dual graph is a tree).
It has been an open problem for some time to generalize their theory to more general semistable curves.\footnote{Brian Osserman \cite{Osserman14} has recently proposed a different framework for doing this.}

\medskip

Recall that the \emph{vanishing sequence} of a linear series $L = (L,W)$ at $p \in C$, where $W \subset H^0(C,L)$, is the ordered sequence
$$ a^{L}_0 (p) < \cdots < a^{L}_r (p) $$
of integers $k$ with the property that there exists some $s \in W$ vanishing to order exactly $k$ at $p$.
For strongly semistable curves of compact type, Eisenbud and Harris define a notion of {\em crude limit $\g^r_d$} $L$ on $C_0$, which is the data of a (not necessarily complete) degree $d$ and rank $r$ linear series $L_v$ on $C_v$ for each vertex $v \in V$ with the following property:  if two components $C_u$ and $C_v$ of $C_0$ meet at a node $p$, then for any $ 0\leq i \leq r,$
\begin{equation*}
 a^{L_v}_{i}( p ) + a^{L_u}_{r-i}( p ) \, \geq \, d\, .
 \end{equation*}

\medskip

We can canonically associate to a proper strongly semistable curve $C_0$ a metrized complex $\fC$ of $\kappa$-curves, called the {\em regularization} of $C_0$, by assigning a length of $1$ to each edge of $G$.
This is the metrized complex associated to any regular smoothing ${\mathfrak C}$ of $C_0$ over any discrete valuation ring $R$ with residue field $\kappa$.

\begin{theorem}
\label{thm:LLS}
Let $\fC$ be the metrized complex of curves associated to a strongly semistable curve $C_0 / \kappa$ of compact
type.  Then there is a bijective correspondence between the following:

\begin{enumerate}
\item Crude limit $\g^r_d$'s on $C_0$ in the sense of Eisenbud and Harris.
\item Equivalence classes of pairs $({\mathcal H},\cD)$, where ${\mathcal H} = \{ H_v \}$, $H_v$ is an $(r+1)$-dimensional subspace of $\kappa(C_v)$ for each $v \in V$, and
$\cD$ is a divisor of degree $d$ supported on the vertices of $\fC$ with
$r_{\fC,{\mathcal H}}(\cD) = r$.  Here we say that $({\mathcal H},\cD) \sim ({\mathcal H}', \cD')$ if there is a rational function $\f$
on $\fC$ such that $D' = D + {\rm div}(\f)$ and $H_v = H'_v \cdot f_v$ for all $v \in V$, where $f_v$ denotes the $C_v$-part of $\f$.
\end{enumerate}
\end{theorem}

Theorem~\ref{thm:LLS}, combined with the Riemann-Roch theorem for metrized complexes of curves, provides a new proof of the fact, originally established in \cite{EisenbudHarris86}, that limit linear series satisfy
analogues of the classical theorems of Riemann and Clifford.
The point is that $r_{\fC, {\mathcal H}}(\cD) \leq r_{\fC}(\cD)$ for all $\cD \in \Div(\fC)$, and therefore upper
bounds on $r_{\fC}(\cD)$ which follow from Riemann-Roch imply corresponding upper bounds on the restricted rank
$r_{\fC, {\mathcal H}}(\cD)$.

\medskip

Motivated by Theorem~\ref{thm:LLS}, Amini and Baker propose the following definition.

\begin{definition} \label{def:limitgrd}
Let $C_0$ be a strongly semistable (but not necessarily compact type) curve over $\kappa$ with regularization $\fC$.
A {\em limit $\g^r_d$} on $C_0$ is an equivalence class of pairs $(\{ H_v \},\cD)$ as above,
where $H_v$ is an $(r+1)$-dimensional subspace of $\kappa(C_v)$ for each $v \in V$,
and $\cD$ is a degree $d$ divisor on $\fC$ with $r_{\fC,{\mathcal H}}(\cD) = r$.
\end{definition}

\medskip

As partial additional justification for Definition~\ref{def:limitgrd},
Amini and Baker prove, using specialization, that a $\g^r_d$ on the smooth general fiber $C$ of a semistable family $\cC$ gives rise in a natural way to a crude limit $\g^r_d$ on the central fiber.




\part{Applications}
\label{Part:Applications}

In this part, we discuss several recent applications of tropical Brill-Noether theory to problems in algebraic and arithmetic geometry.  These sections are largely independent of each other, so the reader should be able to peruse them according to his or her interest.

\section{Applications of Tropical Linear Series to Classical Brill-Noether Theory}
\label{Section:Applications}

Recent years have witnessed several applications of tropical Brill-Noether theory to problems in classical algebraic geometry.  In this section, we survey the major recent developments in the field.

\subsection{The Brill-Noether Theorem}
\label{Subsection:BNThm}

The Brill-Noether Theorem predicts the dimension of the space $W^r_d(C)$ parameterizing divisor classes (or, equivalently, complete linear series) of a given degree and rank on a general curve $C$.

\begin{BNthm}[\cite{GriffithsHarris80}]
\label{Thm:BNThm}
Let $C$ be a general curve of genus $g$ over $\CC$.   Then $W^r_d (C)$ has pure dimension $\rho(g,r,d) = g-(r+1)(g-d+r)$, if this is nonnegative, and is empty otherwise.
\end{BNthm}

The original proof of the Brill-Noether Theorem, due to Griffiths and Harris, involves a subtle degeneration argument \cite{GriffithsHarris80}.  The later development of limit linear series by Eisenbud and Harris led to a simpler proof of this theorem \cite{EisenbudHarris83c, EisenbudHarris86}.  The literature contains several other proofs, some
of which work in any characteristic.  One that is often referenced is due to Lazarsfeld, because rather than using degenerations, Lazarsfeld's argument involves vector bundles on K3 surfaces \cite{Lazarsfeld86}.

\medskip

The first significant application of tropical Brill-Noether theory was the new proof of the Brill-Noether Theorem by Cools, Draisma, Payne and Robeva \cite{tropicalBN}, which successfully realized the program laid out in \cite{Baker08}.  In \cite{tropicalBN}, the authors consider the family of graphs pictured in Figure \ref{Fig:ChainOfLoops}, colloquially known as the chain of loops.\footnote{In fact, they consider the graph in which the lengths of the bridge edges between the loops are all zero.  There is, however, a natural rank-preserving isomorphism between the Jacobian of a metric graph with a bridge and the Jacobian of the graph in which that bridge has been contracted, so their argument works equally well in this case.  We consider the graph with bridges because of its use in \cite{tropicalGP} and \cite{MRC}.}  The edge lengths are further assumed to be generic, which in this case means that, if $\ell_i , m_i$ are the lengths of the bottom and top edges of the $i$th loop, then $\ell_i /m_i$ is not equal to the ratio of two positive integers whose sum is less than or equal to $2g-2$.

\begin{figure}
\begin{tikzpicture}

\draw [ball color=black] (-1.7,-0.45) circle (0.55mm);
\draw (-1.95,-0.65) node {\footnotesize $v_1$};
\draw (-1.5,0) circle (0.5);
\draw (-1,0)--(0,0.5);
\draw [ball color=black] (-1,0) circle (0.55mm);
\draw (-0.85,0.3) node {\footnotesize $w_1$};
\draw (0.7,0.5) circle (0.7);
\draw (1.4,0.5)--(2,0.3);
\draw [ball color=black] (1.4,0.5) circle (0.55mm);
\draw [ball color=black] (0,0.5) circle (0.55mm);
\draw (-0.2,0.75) node {\footnotesize $v_2$};
\draw (2.6,0.3) circle (0.6);
\draw (3.2,0.3)--(3.87,0.6);
\draw [ball color=black] (2,0.3) circle (0.55mm);
\draw [ball color=black] (3.2,0.3) circle (0.55mm);
\draw [ball color=black] (3.87,0.6) circle (0.55mm);
\draw (4.5,0.3) circle (0.7);
\draw (5.16,0.5)--(5.9,0);
\draw (6.4,0) circle (0.5);
\draw [ball color=black] (5.16,0.5) circle (0.55mm);
\draw (5.48,0.74) node {\footnotesize $w_{g-1}$};
\draw [ball color=black] (5.9,0) circle (0.55mm);
\draw [ball color=black] (6.9,0) circle (0.55mm);
\draw (5.7,-.2) node {\footnotesize $v_g$};
\draw (7.3,-.2) node {\footnotesize $w_g$};

\draw [<->] (3.3,0.4) arc[radius = 0.715, start angle=10, end angle=170];
\draw [<->] (3.3,0.2) arc[radius = 0.715, start angle=-9, end angle=-173];

\draw (2.5,1.25) node {\footnotesize$\ell_i$};
\draw (2.75,-0.7) node {\footnotesize$m_i$};
\end{tikzpicture}
\caption{The graph $\Gamma$.}
\label{Fig:ChainOfLoops}
\end{figure}
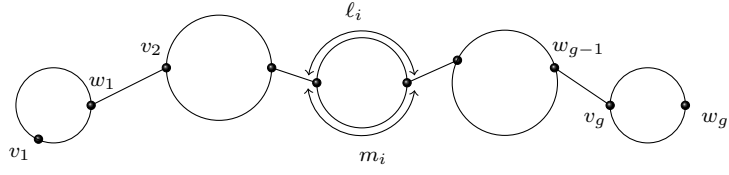

\medskip

Using Theorem~\ref{Thm:WrdNonempty} as the only input from algebraic geometry, the authors of \cite{tropicalBN} employ an intricate combinatorial argument to prove the following:

\begin{theorem}[\cite{tropicalBN}]
\label{Thm:TropicalBNThm}
Let $C$ be a smooth projective curve of genus $g$ over a discretely valued field with a regular, strongly semistable model whose special fiber is a generic chain of loops $\Gamma$.  Then $\dim W^r_d (C) = \rho (g,r,d)$ if this number is nonnegative, and $W^r_d (C) = \emptyset$ otherwise.
\end{theorem}

We note that such a curve $C$ exists by Corollary \ref{Cor:Surjective}.  The Brill-Noether Theorem (over an arbitrary algebraically closed field) then follows from Theorem~\ref{Thm:TropicalBNThm} using
the theory of Brill-Noether rank discussed in \S \ref{Subsection:BNRank}.

\medskip

In fact, \cite{tropicalBN} proves more.  Theorem 4.6 of \cite{tropicalBN} completely describes $W^r_d (\Gamma)$, explicitly classifying all divisors of given degree and rank on a generic chain of loops.  Indeed, it is shown that $W^r_d (\Gamma)$ is a union of $\rho$-dimensional tori.  The set of tori is in bijection with so-called ``lingering lattice paths'', which in turn are in bijection with standard Young tableaux on a rectangle with $r+1$ columns and $g-d+r$ rows containing the numbers $1, \ldots , g$.  From this, one can compute the number of tori to be $$ {{g}\choose{\rho}} (g-\rho)! \prod_{i=0}^r \frac{i!}{(g-d+r+i)!}$$
if $\rho \geq 0$, and $0$ if $\rho < 0$.

\medskip

We briefly discuss the argument here.  Given an effective divisor $D$, we may assume that $D$ is $v_1$-reduced.  The divisor $D$ then has some number $d_1$ of chips at $v_1$, and by Dhar's burning algorithm $D$ has at most 1 chip on each of the half-open loops $\gamma_k$ pictured in Figure \ref{Figure:Cells}, and no chips on the half-open bridges $\br_k$.

\begin{figure}

\begin{tikzpicture}
\matrix[column sep=0.5cm] {
\begin{scope}[baseline]
\draw [ball color=black] (-1.7,-0.45) circle (0.55mm);
\draw (-1.95,-0.65) node {\footnotesize $v_1$};
\draw (-1.5,0) circle (0.5);
\draw [ball color=white] (-1,0) circle (0.55mm);
\draw (-1.5,1.0) node {\footnotesize $\gamma_1$};
\end{scope}
&
\begin{scope}[grow=right,baseline]
\draw [ball color=black] (-1,0) circle (0.55mm);
\draw (-1,0)--(0,0.5);
\draw (-0.8,0.4) node {\footnotesize $w_1$};
\draw [ball color=white] (0,0.5) circle (0.55mm);
\draw (-0.5,1.2) node {\footnotesize $\br_1$};
\end{scope}
&
\begin{scope}[grow=right,baseline]
\draw node at (.5,0.48) {$\cdots$};
\end{scope}
&
\begin{scope}[grow=right,baseline]
\draw (0.7,0.5) circle (0.7);
\draw [ball color=white] (1.4,0.5) circle (0.55mm);
\draw [ball color=black] (0,0.5) circle (0.55mm);
\draw (0.7,1.5) node {\footnotesize $\gamma_i$};
\end{scope}
&
\begin{scope}[grow=right,baseline]
\draw (1.4,0.5)--(2,0.3);
\draw [ball color=white] (2,0.3) circle (0.55mm);
\draw [ball color=black] (1.4,0.5) circle (0.55mm);
\draw (1.7,1.1) node {\footnotesize $\br_i$};
\end{scope}
&
\begin{scope}[grow=right,baseline]
\draw node at (2,0.3) {$\cdots$};
\end{scope}
&
\begin{scope}[grow=right,baseline]
\draw (2.6,0.3) circle (0.6);
\draw [ball color=black] (2,0.3) circle (0.55mm);
\draw [ball color=white] (3.2,0.3) circle (0.55mm);
\draw (2.6,1.2) node {\footnotesize $\gamma_g$};
\end{scope}
&
\begin{scope}[grow=right, baseline]
\draw [ball color = black] (5.16,0.3) circle (0.55mm);
\draw node at (5.16,0.05) {\footnotesize $w_g$};
\end{scope}
\\};
\end{tikzpicture}
\caption{A decomposition of $\Gamma$.}
\label{Figure:Cells}
\end{figure}
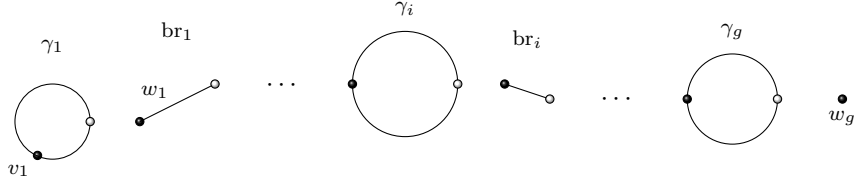

\medskip

The associated \textit{lingering lattice path} is a sequence of vectors $p_i \in \ZZ^r$, starting at $p_1 = (d_1 , d_1 -1, \ldots , d_1 -r+1)$, with the $i$th step given by
\begin{displaymath}
p_{i+1} - p_i = \left\{ \begin{array}{ll}
(-1,-1, \ldots , -1) & \textrm{if $D$ has no chip on $\gamma_i$}\\
e_j & \textrm{if $D$ has a chip on $\gamma_i$, the distance from $v_i$} \\
 & \textrm{to this chip is precisely } (p_i(j) +1)m_{i+1}, \\
  & \textrm{and both $p_i$ and $p_i + e_j$ are in $\mathcal{W}$}\\
0 & \textrm{otherwise}
\end{array} \right\}
\end{displaymath}
Here, the distance from $v_i$ is in the counterclockwise direction.  Since the chip lies on a circle of circumference $m_i + \ell_i$, this distance should be understood to be modulo $(m_i + \ell_i)$.  The symbols $e_0 , \ldots e_{r-1}$ represent the standard basis vectors in $\mathbb{Z}^r$ and $\mathcal{W}$ is the open Weyl chamber
$$ \mathcal{W} = \{ y \in \mathbb{Z}^r \vert y_0 > y_1 > \cdots > y_{r-1} > 0 \}. $$

\medskip

The steps where $p_{i+1} = p_i$ are known as \emph{lingering steps}.  The basic idea of the lingering lattice path is as follows.  By Theorem \ref{Thm:RankDetermining}, the set $\{ v_1 , v_2 , \ldots , v_g , w_g \}$ is rank-determining.  Hence, if $D$ fails to have rank $r$, there is an effective divisor $E$ of degree $r$, supported on these vertices, such that $\vert D-E \vert = \emptyset$.  Starting with the $v_1$-reduced divisor $D$, we move chips to the right and record the $v_i$-degree of the equivalent $v_i$-reduced divisor.  The number $p_i (j)$ is then the minimum, over all effective divisors $E$ of degree $j$ supported at $v_1 , \ldots , v_i$, of the $v_i$-degree of the $v_i$-reduced divisor equivalent to $D-E$.  From this it follows that, when $D$ has rank at least $r$, we must have $p_i (j) \geq r-j$, so the corresponding lingering lattice path must lie in the open Weyl chamber $\mathcal{W}$.

\medskip

The corresponding tableau is constructed by placing the moves in the direction $e_i$ in the $i$th column of the rectangle, and the moves in the direction $(-1 , \ldots , -1)$ in the last column.  If the $k$th step is lingering, then the integer $k$ does not appear in the tableau.  Given this description, we see that each tableau determines the existence and position of the chip on the half-open loop $\gamma_k$ if and only if the integer $k$ appears in the tableau.  Otherwise, the chip on the $k$th loop is allowed to move freely.  The number of chips that are allowed to move freely is therefore $\rho = g-(r+1)(g-d+r)$.  Indeed, we see that not only is the Brill-Noether rank $w^r_d (\Gamma)$ equal to $\rho$, but in fact $\dim W^r_d (\Gamma) = \rho$ as well.  Theorem \ref{Thm:TropicalBNThm} then follows from the specialization result for Brill-Noether rank, Theorem \ref{Thm:BNRankSpecialization}.

\subsection{The Gieseker-Petri Theorem}
\label{Subsection:GPThm}

Assume that $\rho(g,r,d) \geq 0$.
The variety $W^r_d (C)$ is singular along $W^{r+1}_d (C)$.  Blowing up along this subvariety yields the variety $\cG^r_d (C)$ parameterizing (not necessarily complete) linear series of degree $d$ and rank $r$ on $C$.  A natural generalization of the Brill-Noether Theorem is the following:

\begin{GPthm}[\cite{Gieseker82}]
\label{Thm:GPThm}
Let $C$ be a general curve of genus $g$.  If $\rho (g,r,d) \geq 0$, then $\cG^r_d (C)$ is smooth of dimension $\rho(g,r,d)$.
\end{GPthm}

It is a standard result, following \cite[\S IV.4]{ACGH}, that the Zariski cotangent space to $\cG^r_d (C)$ at a point corresponding to a complete linear series $\cL(D)$ is naturally isomorphic to the cokernel of the adjoint multiplication map
\[
\mu_D : \cL(D) \otimes \cL(K_C - D) \to \cL( K_C ).
\]
Thus the cotangent space has dimension $\rho(g,r,d) + \dim \ker \mu_D$, and in particular, $\cG^r_d(C)$ is smooth of dimension $\rho(g,r,d)$ at such a point if and only if the multiplication map $\mu_D$ is injective.
More generally, if $P \in \cG^r_d (C)$ corresponds to a possibly incomplete linear series $W \subset \cL(D)$, then $\cG^r_d(C)$ is smooth of dimension $\rho(g,r,d)$ at $P$ if and only if the multiplication map $W \otimes \cL(K_C - D) \to \cL( K_C )$ is injective.
One deduces that the Gieseker-Petri Theorem is equivalent to the assertion that if $C$ is a general curve of genus $g$, then $\mu_D$ is injective for all divisors $D$ on $C$.

\medskip

A recent application of tropical Brill-Noether theory is the following result \cite{tropicalGP}, which yields a new proof of the Gieseker-Petri Theorem:

\begin{theorem}[\cite{tropicalGP}]
\label{Thm:TropicalGPThm}
Let $C$ be a smooth projective curve of genus $g$ over a discretely valued field with a regular, strongly semistable model whose special fiber is a generic chain of loops $\Gamma$.  Then the multiplication map
$$\mu_D : \cL(D) \otimes \cL(K_C - D) \to \cL(K_C ) $$
is injective for all divisors $D$ on $C$.
\end{theorem}

The argument has much in common with the tropical proof of the Brill-Noether Theorem, using the same metric graph with the same genericity conditions on edge lengths.  The new ingredient is the idea of tropical independence, as defined in \S \ref{Subsection:TropicalIndependence}.  Given a divisor $D \in W^r_d (C)$, the goal is to find functions
$$ f_0 , \ldots , f_r \in \trop (\mathcal{L} (D)) $$
$$ g_0 , \ldots , g_{g-d+r-1} \in \trop (\mathcal{L} (K_C - D)) $$
such that $\{ f_i + g_j  \}_{i,j}$ is tropically independent.

\medskip

There is a dense open subset of $W^r_d (\Gamma)$ consisting of divisors $D$ with the following property:  given an integer $0 \leq i \leq r$, there exists a unique divisor $D_i \sim D$ such that
$$ D_i - iw_g - (r-i)v_1 \geq 0 . $$
These are the divisors referred to as \emph{vertex-avoiding} in \cite{lifting}.

\medskip

We first describe the proof of Theorem~\ref{Thm:TropicalGPThm} in the case that $D$ is vertex-avoiding.  If $D$ is the specialization of a divisor $\mathcal{D} \in W^r_d (C)$, and $p_1 , p_g \in C$ are points specializing to $v_1, w_g$, respectively, then there exists a divisor $\mathcal{D}_i \sim \mathcal{D}$ such that
$$ \mathcal{D}_i - ip_g - (r-i)p_1 \geq 0, $$
and, by the uniqueness of $D_i$, $\mathcal{D}_i$ must specialize to $D_i$.  It follows that there is a function $f_i \in \trop (\mathcal{L} (\mathcal{D}))$ such that $\ddiv (f_i) = D_i - D$.

\medskip

For this open subset of divisors, the argument then proceeds as follows.  By the classification in \cite{tropicalBN}, the divisor $D_i$ fails to have a chip on the $k$th loop if and only if the integer $k$ appears in the $i$th column of the corresponding tableau.  The adjoint divisor $E = K_{\Gamma} - D$ corresponds to the transpose tableau \cite[Theorem 39]{AMSW}, so the divisor $D_i + E_j$ fails to have a chip on the $k$th loop if and only if $k$ appears in the $(i,j)$ position of the tableau.  Since for each $k$ at most one of these divisors fails to have a chip on the $k$th loop, we see that if
$$ \theta = \min \{ f_i + g_j + b_{i,j} \} $$
occurs at least twice at every point of $\Gamma$, then the divisor
$$ \Delta = \ddiv (\theta) + K_{\Gamma} $$
must have a chip on the $k$th loop for all $k$.

\medskip

To see that this is impossible, let $p_k$ be a point of $\Delta$ in $\gamma_k$, and let
$$ D' = p_1 + \cdots + p_g . $$
Then by construction $K_\Gamma-D'$ is equivalent to an effective divisor, so by the tropical Riemann-Roch Theorem we see that $r(D') \geq 1$.
On the other hand, Dhar's burning algorithm shows that $D'$ is universally reduced, so by Proposition~\ref{prop:SimpleBreakDivisors} we have $r(D') = 0$, a contradiction.

\medskip

It is interesting to note that this obstruction is, at heart, combinatorial.  Unlike the earlier proofs via limit linear series, which arrive at a contradiction by constructing a canonical divisor of impossible \emph{degree} (larger than $2g-2$), this argument arrives at a contradiction by constructing a canonical divisor of impossible \emph{shape}.

\medskip

The major obstacle to extending this argument to the case where $D$ is not vertex-avoiding is that the containment $\trop (\mathcal{L} (D)) \subseteq R( \Trop (D))$ is often strict.  Given an arbitrary divisor $D \in W^r_d (C)$ and function $f \in R(\Trop(D))$, it is difficult to determine whether $f$ is the specialization of a function in $\mathcal{L} (D)$.  To avoid this issue, the authors make use of a patching construction, gluing together tropicalizations of different rational functions in a fixed algebraic linear series on different parts of the graph, to arrive at a piecewise linear function in $R(K_{\Gamma})$ that may not be in $\trop ( \mathcal{L} (K_C))$.  Once this piecewise linear function is constructed, the argument proceeds very similarly to the vertex-avoiding case.

\subsection{The Maximal Rank Conjecture}
\label{Subsection:MRC}

One of the most well-known open problems in Brill-Noether theory is the Maximal Rank Conjecture, which predicts the Hilbert function for sufficiently general embeddings of sufficiently general curves.  This conjecture is attributed to Noether in \cite[p.\,4]{ArbarelloCiliberto83}, (see \cite[\S8]{Noether82} and \cite[pp.\,172--173]{CES25} for details) was studied classically by Severi \cite[\S10]{Severi15}, and popularized by Harris \cite[p.\,79]{Harris82}.

\begin{MRC}
\label{Conj:MRC}
Fix nonnegative integers $g,r,d$, let $C$ be a general curve of genus $g$, and let $V \subset \cL(D)$ be a general linear series of rank $r$ and degree $d$ on $C$.  Then the multiplication maps
\[
\mu_m: {\rm Sym}^m V \rightarrow \cL(mD)
\]
have maximal rank for all $m$.  That is, each $\mu_m$ is either injective or surjective.
\end{MRC}

While the Maximal Rank Conjecture remains open in general, several important cases are known \cite{BE87, Voisin92, TiB03, Farkas09}.  For example, it is shown in \cite{BE87} that the Maximal Rank Conjecture holds in the non-special range $d \geq g+r$.  When $d < g+r$, the general linear series of degree $d$ and rank $r$ on a general curve is complete, and for this reason, most of the work in the subject focuses on the case where $V = \cL (D)$.  We note that the arguments of \cite{BE87} and \cite{Farkas09} involve degenerations to unions of two curves that meet in more than one point.  Since such curves are not of compact type, the arguments do not make use of limit linear series.

\medskip

In \cite{MRC}, tropical Brill-Noether theory is used to prove the $m=2$ case of the Maximal Rank Conjecture.

\begin{theorem}[\cite{MRC}]
\label{Thm:TropicalMRC}
Let $C$ be a smooth projective curve of genus $g$ over a discretely valued field with a regular, strongly semistable model whose special fiber is a generic chain of loops $\Gamma$.  For a given $r$ and $d$, let $D$ be a general divisor of rank $r$ and degree $d$ on $C$.  Then the multiplication map
\[
\mu_2: {\rm Sym}^2 \cL(D) \rightarrow \cL(2D)
\]
has maximal rank.
\end{theorem}

The genericity conditions placed on the edge lengths of $\Gamma$ in Theorem~\ref{Thm:TropicalMRC} are stricter than those appearing in the tropical proofs of the Brill-Noether and Gieseker-Petri Theorems.  First, the bridges between the loops are assumed to be much longer than the loops themselves, and second, one must assume that certain integer linear combinations of the edge lengths do not vanish.

\medskip

A simplifying aspect of the Maximal Rank Conjecture is that it concerns a general, rather than arbitrary, divisor.  It therefore suffices to prove that the maximal rank condition holds for a single divisor of the given degree and rank on $C$.  The main result of \cite{lifting} is that every divisor on the generic chain of loops is the specialization of a divisor of the same rank on $C$.  We are therefore free to choose whatever divisor we wish to work with, and in particular we may choose one of the vertex-avoiding divisors described in the previous section.  Recall that, if $D \in W^r_d (\Gamma)$ is vertex-avoiding, then we have an explicit set of piecewise linear functions $f_i \in R(D)$ that are tropicalizations of a basis for the linear series on the curve $C$.  The goal, in the case where the multiplication map is supposed to be injective, is to show that the set $\{ f_i + f_j \}_{i \leq j}$ is tropically independent.  In the surjective case, we must choose a subset of the appropriate size, and then show that this subset is tropically independent.

\medskip

The basic idea of the argument is as follows.  Assume that
$$ \theta = \min \{ f_i + f_j + b_{i,j} \} $$
occurs at least twice at every point of $\Gamma$, and consider the divisor
$$ \Delta = \ddiv (\theta) + 2D . $$
To arrive at a contradiction, one studies the degree distribution of the divisor $\Delta$ across the loops of $\Gamma$.  More precisely, one defines
$$ \delta_k := \deg ( \Delta \vert_{\gamma_k} ) . $$
The first step is to show that $\delta_k \geq 2$ for all $k$.  One then identifies intervals $[a,b]$ for which this inequality must be strict for at least one $k \in [a,b]$.  As one proceeds from left to right across the graph, one encounters such intervals sufficiently many times to obtain $\deg \Delta > 2 \deg D$, a contradiction.

\section{Lifting Problems for Divisors on Metric Graphs}
\label{Section:Lifting}

In this section we discuss the lifting problem in tropical Brill-Noether theory: given a divisor of rank $r$ on a metric graph $\Gamma$, when is it the tropicalization of a rank $r$ divisor on a smooth curve $C$?
There are essentially two formulations of this problem, one in which the curve $C$ is fixed, and one in which it is not.

\medskip

Throughout this section, we let $K$ be a complete and algebraically closed non-trivially valued non-Archimedean field.

\begin{question}
Given a metric graph $\Gamma$ and a divisor $D$ on $\Gamma$, under what conditions do there exist a curve $C/K$ (together with a semistable model $R$) and a divisor of the same rank as $D$ tropicalizing to $\Gamma$ and $D$, respectively?
\end{question}

\begin{question}
Given a curve $C/K$ (together with a semistable model $R$) tropicalizing to a metric graph $\Gamma$, and given a divisor $D$ on $\Gamma$, under what conditions does there exist a divisor on $C$ of the same rank as $D$ tropicalizing to $D$?
\end{question}

These are very difficult questions.  Even for the earlier theory of limit linear series on curves of compact type, the analogous questions remain open.  A partial answer in that setting is given by the Regeneration Theorem of Eisenbud and Harris \cite{EisenbudHarris86}, which says that
if the space of limit linear series has local dimension equal to the Brill-Noether number $\rho$, then the given limit linear series lifts in any one-parameter smoothing.  At the time of writing, there is no corresponding theorem in the tropical setting.\footnote{Amini has apparently made substantial progress in this direction.}

\subsection{Specialization of hyperelliptic curves}

One of the first results concerning lifting of divisors is the classification of vertex-weighted metric graphs that are the specialization of a hyperelliptic curve.
Recall from \S\ref{Subsection:vertexweight} that given a curve $C/K$ and a semistable model $\cC / R$ for $C$, there is a natural way to associate to $\cC$ a vertex-weighted metric graph $(\Gamma, \omega)$.
We call such a pair {\em minimal} if there is no vertex $v$ with ${\rm val}(v)=1$ and $\omega(v)=0$.

\begin{theorem}[\cite{Caporaso12,ABBR14b}]
Let $(\Gamma,\omega)$ be a minimal vertex-weighted metric graph.
There is a smooth projective hyperelliptic curve over a discretely valued field with a regular, strongly semistable model whose special fiber has dual graph $\Gamma$ if and only if the following conditions hold:
\begin{itemize}
\item[(HYP1)] there exists an involution $s$ on $\Gamma$ such that the quotient $\Gamma / s$ is a tree and $s(v)=v$ for all $v \in \Gamma$ with $\omega(v)>0$, and
\item[(HYP2)] for every point $v \in \Gamma$, the number of bridge edges adjacent to $v$ is at most $2\omega(v)+2$.
\end{itemize}
\end{theorem}

Kawaguchi and Yamaki show moreover that, when $\Gamma$ satisfies these conditions, there is a smoothing $C$ for which every divisor on $\Gamma$ lifts to a divisor of the same rank on $C$ \cite{KawaguchiYamaki14}.

\medskip

We outline the necessity of the conditions above in the special case where $\omega=0$, which is equivalent to requiring that $g(C)=g(\Gamma)$.  Note that if $C$ is a hyperelliptic curve, then by the Specialization Theorem any divisor of degree 2 and rank 1 on $C$ specializes to a divisor $D$ of rank at least 1 on $\Gamma$, and by Tropical Clifford's Theorem $D$ must have rank exactly 1.  Now, if $P \in \Gamma$ is not contained in a bridge, then $\vert P \vert = \{ P \}$.  On the other hand, if $P \in \Gamma$ is contained in a bridge, a simple analysis reveals that $D \sim 2P$.  In this way we obtain an involution $s$ on $\Gamma$ mapping each point $P$ to the $P$-reduced divisor equivalent to $D-P$.

\medskip

To see why (HYP2) holds, note that for each type-2 point $v \in \Gamma$, the linear series of degree 2 and rank 1 on $C$ specializes to a linear series of degree 2 and rank 1 on the corresponding curve $C_v$.  Each of the bridges adjacent to $v$ correspond to ramification points of this linear series, but such a linear series has only
$2g(C_v)+2 = 2$ ramification points.

\medskip

To see that the conditions (HYP1) and (HYP2) are sufficient requires significantly more work.

\begin{example}
Consider the metric graph $\Gamma$ pictured in Figure \ref{Fig:NonLiftable}, consisting of a tree with a loop attached to each leaf, with all edge lengths being arbitrary.  Then $\Gamma$ is hyperelliptic, because any divisor of degree 2 supported on the tree has rank 1.  On the other hand, this graph is not the dual graph of the limit of any family of genus 3 hyperelliptic curves, because the vertex of valence 3 in the tree is adjacent to more than 2 bridges.

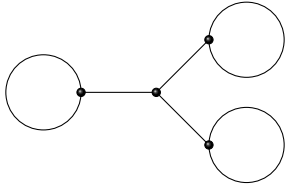
\begin{figure}
\begin{tikzpicture}

\draw (-1.5,0) circle (0.5);
\draw (1.2,.7) circle (0.5);
\draw (1.2,-.7) circle (0.5);
\draw (-1,0)--(0,0);
\draw (0,0)--(.7,.7);
\draw (0,0)--(.7,-.7);
\draw [ball color=black] (-1,0) circle (0.55mm);
\draw [ball color=black] (0,0) circle (0.55mm);
\draw [ball color=black] (.7,.7) circle (0.55mm);
\draw [ball color=black] (.7,-.7) circle (0.55mm);

\end{tikzpicture}
\caption{A hyperelliptic metric graph of genus 3 that is not the skeleton of any hyperelliptic curve of genus 3.}
\label{Fig:NonLiftable}
\end{figure}
\end{example}

For metric graphs of higher gonality, the lifting problem is significantly harder.  In \cite{LuoManjunath14}, Luo and Manjunath describe an algorithm for smoothability of rank one generalized limit linear series on metrized complexes.

\subsection{Lifting divisors on the chain of loops}

For some specific families of graphs, such as the chain of loops discussed \S \ref{Section:Applications}, one can show that the lifting problem is unobstructed.

\begin{theorem}[\cite{lifting}]
\label{Thm:Lifting}
Let $C/K$ be a smooth projective curve of genus~$g$.  If the dual graph of the central fiber of some regular model of $C$ is isometric to a generic chain of loops $\Gamma$ of genus $g$, then every divisor class on $\Gamma$ that is rational over the value group of $K$ lifts to a divisor class of the same rank on $C$.
\end{theorem}

The general strategy for proving Theorem~\ref{Thm:Lifting} is to study the Brill-Noether loci as subschemes
\[
W^r_d(C) \subset \Jac(C).
\]
Since $C$ is maximally degenerate, the universal cover of $\Jac(C)^{\an}$ gives a uniformization
\[
T^{\an} \rightarrow \Jac(C)^{\an}
\]
by an algebraic torus $T$ of dimension $g$.  The tropicalization of this torus is the universal cover of the skeleton of $\Jac(C)$, which, as discussed in \S \ref{Section:Berkovich} is canonically identified with the tropical Jacobian of $\Gamma$ \cite{BakerRabinoff13}.

\medskip

A key tool in the proof of Theorem~\ref{Thm:Lifting} is Rabinoff's lifting theorem~\cite{Rabinoff12}, which can be applied to the analytic preimages in~$T$ of algebraic subschemes of $\Jac(C)$. This lifting theorem says that isolated points in complete intersections of tropicalizations of analytic hypersurfaces lift to points in the analytic intersection with appropriate multiplicities.  This theorem can be applied to translates of the preimage of the theta divisor $\Theta_{\Gamma} = W^0_{g-1} (\Gamma)$, as follows.

\medskip

When $\Gamma$ is the generic chain of loops, one can use the explicit description of $W^r_d (\Gamma)$ from \cite{tropicalBN} to produce explicit translates of $\Theta_C$ whose tropicalizations intersect transversally and locally cut out $W^r_d(\Gamma)$.  By intersecting with $\rho$ additional translates of $\Theta_C$, one obtains an isolated point in a tropical complete intersection, to which we may apply Rabinoff's lifting theorem.  This complete intersection is typically larger than $W^r_d(\Gamma)$, but the argument shows that the tropicalization map from a 0-dimensional slice of $W^r_d(C)$ to the corresponding slice of $W^r_d(\Gamma)$ is injective.  Using again the explicit description of $W^r_d (\Gamma)$, one then shows that the two finite sets have the same cardinality, and hence the map is bijective.

\begin{remark}
As mentioned in the section on the Maximal Rank Conjecture, Theorem \ref{Thm:Lifting} is one of the key ingredients in the proof of Theorem \ref{Thm:TropicalMRC} (the Maximal Rank Conjecture for quadrics).  In particular, in order to show that the maximal rank condition holds for a generic line bundle of a given degree and rank, it suffices to show that it holds for a single line bundle.  Since every divisor of a given rank on the chain of loops lifts to a line bundle on $C$ of the same rank, one is free to work with any divisor of this rank on the chain of loops.
\end{remark}

\subsection{Examples of divisors that do not lift}

Among the results on lifting divisors, there is a plethora of examples of divisors that do not lift.  For example, in \cite{Coppens14b}, Coppens defines a base-point free divisor on a metric graph $\Gamma$ to be a divisor $D$ such that $r (D-p) < r(D)$ for all $p \in \Gamma$.  He then shows that the Clifford and Riemann-Roch bounds are the only obstructions to the existence of base-point free divisors on metric graphs of arbitrary genus.  This is in contrast to the case of algebraic curves, where for example a curve of genus greater than 6 cannot have a base-point free divisor of degree 5 and rank 2.

\medskip

Another example of divisors that do not lift comes from the theory of matroids.

\begin{theorem}[\cite{Cartwright15}]
\label{Thm:Matroids}
Let $M$ be any rank 3 matroid.  Then there exists a graph $G_M$ and a rank 2 divisor $D_M$ on $G_M$ such that, for any infinite field $k$, there are a curve $C$ over $k((t))$ (together with a semistable model $\cC$ of $C$ over $k[[t]]$) and a rank 2 divisor on $C$ tropicalizing to $G_M$ and $D_M$, respectively, if and only if $M$ is realizable over $k$.
\end{theorem}

Combining this with the scheme-theoretic analogue of Mn{\"e}v universality, due to Lafforgue \cite{Lafforgue03}, one obtains the following.

\begin{corollary}[\cite{Cartwright15}]
Let $X$ be a scheme of finite type over $\Spec \ZZ$.  Then there exists a graph $G$ and a rank 2 divisor $D$ on $G$ such that, for any infinite field $k$, $G$ and $D$ are the tropicalizations of a curve $C / k((t))$ and a rank 2 divisor on $C$ if and only if $X$ has a $k$-point.
\end{corollary}

In other words, the obstructions to lifting over a valued field of the form $k((t))$ are essentially as general as possible.

\medskip

Cartwright's construction is as follows.  Recall that a rank 3 simple matroid on a finite set $E$ consists of a collection of subsets of $E$, called \emph{flats}, such that every pair of elements is contained in exactly one flat.  (Here we are abusing language, using the word flat to refer to refer to the maximal, or rank 2, flats.)  The bipartite graph $G_M$ is the \emph{Levi graph} of the matroid $M$, where the vertices correspond to elements and flats, and there is an edge between two vertices if the corresponding element is contained in the corresponding flat.  The divisor $D_M$ is simply the sum of the vertices corresponding to elements of $E$.  A combinatorial argument then shows that the rank of $D_M$ is precisely 2.

\medskip

If $M$ is realizable over $k$, then by definition, there exists a configuration of lines in $\PP^2_k$ where the lines correspond to the elements of $E$, and the flats correspond to points where two or more of the lines intersect.  If we blow up the plane at the intersection points, the dual graph of the resulting configuration is the Levi graph $G_M$, and the pullback of the hyperplane class specializes to the divisor $D_M$.  After some technical deformation arguments, one then sees that the pair $(G_M , D_M)$ admits a lifting when $M$ is realizable over $k$.

\medskip

For the converse, one must essentially show that the above construction is the only possibility.  That is, if $\mathcal{C}$ is a regular semistable curve over $k[[t]]$, the dual graph of the central fiber is $G_M$ and the divisor $D_M$ is the specialization of a rank 2 divisor on $\mathcal{C}$, then in fact the image of the central fiber under the corresponding linear series must provide a realization of the matroid $M$ in $\PP^2_k$.

\section{Bounding the Number of Rational Points on Curves}
\label{Section:Arithmetic}

By Faltings' Theorem (n{\'e}e the Mordell Conjecture), if $C$ is a curve of genus $g \geq 2$ over a number field $K$ then the set $C(K)$ of rational points on $C$ is finite.
Shortly after Faltings proved this theorem, Vojta published a new proof which furnishes an effective upper bound on the number of points in $C(K)$.
However, the Vojta bound is completely theoretical --- to our knowledge no one has ever written down the bound explicitly (and the bound is surely quite far from
optimal).  None of the existing proofs of the Mordell Conjecture gives an algorithm --- even in theory! --- to compute the set $C(K)$.  And in practice the situation is even worse ---
it seems safe to say that no one has ever used the Faltings or Vojta proofs of the Mordell Conjecture to compute $C(K)$ in a single non-trivial example.

\subsection{The Katz--Zureick-Brown refinement of Coleman's bound}
\label{Subsection:KZB}

One of the first significant results in the direction of the Mordell Conjecture was Chabauty's theorem that $C(K)$ is finite provided that the rank of the finitely generated abelian
group $\Jac(C)(K)$ is less than $g$.  Much later, Coleman used his theory of $p$-adic integration to give an effective upper bound on $C(K)$ in this situation.  Coleman's bound
has the advantage of being sharp in certain cases, and the method of proof can be used to compute $C(K)$ in a wide range of concrete examples.
For simplicity, we state the results in this section for $K=\QQ$ only, but everything extends with minor modifications to curves over a number field $K$.  Coleman's theorem is as follows.

\begin{theorem}[\cite{Coleman85}]
\label{Thm:ChabautyColeman}
Let $C$ be a curve of genus $g$ over $\QQ$, and
suppose that the Mordell-Weil rank $r$ of $\Jac (C)(\QQ)$ is strictly less than the genus $g$.  Then
for every prime $p > 2g$ of good reduction for $C$, we have
\begin{equation} \label{eq:ChabautyColemanBound}
\# C(\QQ) \leq \# C( \mathbb{F}_p ) + 2g - 2 .
\end{equation}
\end{theorem}

Coleman's theorem was subsequently strengthened in different ways.  In \cite{LorenziniTucker02}, Lorenzini and Tucker (see also McCallum--Poonen \cite{McCallumPoonen12}) generalized Theorem~\ref{Thm:ChabautyColeman} to primes of bad reduction, replacing $C( \mathbb{F}_p )$ in (\ref{eq:ChabautyColemanBound}) by the smooth $\mathbb{F}_p$-points of the special fiber of the minimal proper regular model for $C$ over $\ZZ_p$.  Stoll replaced the quantity $2g-2$ in (\ref{eq:ChabautyColemanBound}) by $2r$ when $C$ has good reduction at $p$, and asked if this improvement could be established in the bad reduction case as well.
Stoll's question was answered affirmatively by Katz and Zureick-Brown in \cite{KZB13} by supplementing Stoll's method with results from the theory of linear series on tropical curves:

\begin{theorem}[\cite{KZB13}]
\label{Thm:ChabautyColemanBadReduction}
Let $C$ be a curve of genus $g$ over $\QQ$ and suppose that the rank $r$ of $\Jac (C)(\QQ)$ is less than $g$.   Then
for every prime $p > 2r+2$, we have
$$ \# C(\QQ) \leq \# \mathcal{C}^{\rm sm}( \mathbb{F}_p ) + 2r, $$
where $\mathcal{C}$ denotes the minimal proper regular model of $C$ over $\ZZ_p$.
\end{theorem}


In order to explain the relevance of linear series on tropical curves to such a result, we need to briefly explain the basic ideas underlying the previous work of Coleman {\it et. al.}
Let us first outline a proof of Theorem~\ref{Thm:ChabautyColeman}.
Fix a rational point $P \in C(\QQ)$ (if no such point exists, then the theorem is vacuously true) and let $\iota : C \hookrightarrow J$ be the corresponding Abel-Jacobi embedding.
Coleman's theory of $p$-adic integration of $1$-forms associates to each $\omega \in H^0(C,\Omega^1_C)$ and $Q \in C(\QQ_p)$ a (definite) $p$-adic integral $\int_P^Q \omega \in \QQ_p$, obtained by pulling back a corresponding $p$-adic integral on $J$ via the map $\iota$.
Locally on $C$, such $p$-adic integrals can be computed by formally integrating a power series expansion $f_\omega(T)$ for $\omega$ with respect to a local parameter $T$ on some residue disc $U$.
One can show fairly easily that the $p$-adic closure $\overline{J(\QQ)}$ of $J(\QQ)$ in $J(\QQ_p)$ has dimension at most $r$ as a $p$-adic manifold.
The formalism of Coleman's theory implies that forcing the $p$-adic integral of a $1$-form on $J$ to vanish identically on $\overline{J(\QQ)}$ imposes at most $r$ linear conditions on $H^0(J,\Omega^1_J)$.
The functoriality of Coleman integration implies that the $\QQ_p$-vector space $V_{\rm chab}$ of all $\omega \in H^0(C,\Omega^1_C)$ such that $\int_P^Q \omega = 0$ for all $Q \in C(\QQ)$ has dimension at least $g-r > 0$.

\medskip

The condition $p>2g$ implies, by a $p$-adic analogue of Rolle's theorem which can be proved in an elementary way with Newton polygons, that if $f_\omega(T)$ has $n$ zeroes on $U$ then $\int f_\omega(T) \; dT$ has at most $n+1$ zeroes on $U$.
Using this observation, Coleman deduces, by summing over all residue classes, that if $\omega$ is a nonzero $1$-form in $H^0(C,\Omega^1_C)$ vanishing on all of $C(\QQ)$ then
\[
\# C(\QQ) \leq \sum_{\overline{Q} \in \bar{C}({\mathbf F}_p)} \left( 1 +  {\rm ord}_{\overline{Q}} \overline{\omega}  \right),
\]
where $\overline{\omega}$ denotes the reduction of $\omega$ to $\bar{C}$.
Since the $1$-form $\overline{\omega}$ on $\bar{C}$ has a total of $2g-2$ zeros counting multiplicity, we have
\[
\sum_{\overline{Q} \in \bar{C}({\mathbf F}_p)} {\rm ord}_{\overline{Q}} \overline{\omega} \leq 2g - 2,
\]
which yields Coleman's bound.

\medskip

Stoll observed in \cite{Stoll06} that one could do better than this by adapting the differential $\omega$ to the point $\overline{Q}$ rather than using the same differential $\omega$ for all residue classes.
Define the {\it Chabauty divisor}
\[
D_{\rm chab} = \sum_{\overline{Q} \in \bar{C}({\mathbf F}_p)} n_{\overline{Q}} (\overline{Q}),
\]
where $n_{\overline{Q}}$ is the minimum over all nonzero $\omega$ in $V_{\rm chab}$ of  ${\rm ord}_{\overline{Q}} \overline{\omega}$, and let $d$ be the degree of $D_{\rm chab}$.
Since $D_{\rm chab}$ and $K_{\bar{C}} - D_{\rm chab}$ are both equivalent to effective divisors, Clifford's inequality (applied to the smooth proper curve $\bar{C}$)
implies that
\[
h^0(K_{\bar{C}} - D_{\rm chab}) - 1 \leq \frac{1}{2}(2g-2-d).
\]
On the other hand, the semicontinuity of $h^0(D) = r(D)+1$ under specialization shows that $h^0(D_{\rm chab}) \geq {\rm dim} V_{\rm chab} \geq g-r$.
Combining these inequalities gives
\[
g-r-1 \leq \frac{1}{2} (2g-2-d)
\]
and thus $d \leq 2r$, giving Stoll's refinement of Coleman's bound.

\medskip

Lorenzini and Tucker \cite{LorenziniTucker02} had shown earlier that one can generalize Coleman's bound to the case of bad reduction as follows.
Since points of $C(\QQ)$ specialize to the set $\bar{\cC}^{\rm sm}({\mathbf F}_p)$ of smooth ${\mathbf F}_p$-points on the special fiber of $\cC$ under the reduction map,
one obtains by an argument similar to the one above the bound
\begin{equation}
\label{eq:LTinequality}
\# C(\QQ) \leq \sum_{\overline{Q} \in \bar{\cC}^{\rm sm}({\mathbf F}_p)} \left( 1 + n_{\overline{Q}} \right),
\end{equation}
where $\overline{\omega}$ denotes the reduction of $\omega$ to the unique irreducible component of the special
fiber of ${\cC}$ containing $\overline{Q}$.
Choosing a nonzero $\omega \in V_{\rm chab}$ as in Coleman's bound, the fact that the relative dualizing sheaf for ${\cC}$ has degree $2g-2$
gives the Lorenzini-Tucker bound.  A similar argument was found independently by McCallum and Poonen \cite{McCallumPoonen12}.

\medskip

We now explain where the subtlety occurs when one tries to combine the bounds of Stoll and Lorenzini--Tucker.
As above, we define the Chabauty divisor
\[
D_{\rm chab} = \sum_{\overline{Q} \in \bar{\cC}^{\rm sm}({\mathbf F}_p)} n_{\overline{Q}} (\overline{Q})
\]
and we let $d$ be its degree.  As in the case where $C$ has good reduction, the goal is to show that $d \leq 2r$.
When $C$ has good reduction, Stoll proves this by combining the semicontinuity of $h^0$ and Clifford's inequality.
For singular curves, one can still define $h^0$ of a line bundle and it satisfies the desired semicontinuity theorem.
However, even when $C$ has semistable reduction, it is well-known that Clifford's inequality does not hold in the form needed here.
Katz and Zureick-Brown replace the use of Clifford's inequality in Stoll's argument by a hybrid between the classical Clifford inequality and Clifford's inequality for linear series on tropical curves.
In this way, they are able to obtain the desired bound $d \leq 2r$.

\medskip

We briefly highlight the main steps in the argument, following the reformulation in terms of metrized complexes given in \cite{AminiBaker12}.

\medskip

1. As noted by Katz and Zureick-Brown, if one makes a base change from $\QQ_p$ to an extension field over which there is a regular semistable model ${\cC'}$ for $C$ dominating the base change of ${\cC}$,
the corresponding Chabauty divisors satisfy $D'_{\rm chab} \geq D_{\rm chab}$.  We may therefore assume that $\cC$ is a regular semistable model for $C$.

\medskip

2. Let $s = \dim V_{\rm chab} - 1$.  We can identify $V_{\rm chab}$ with an $(s+1)$-dimensional space $W$ of rational functions on $C$ in the usual way by identifying $H^0(C,\Omega^1_C)$ with $\mathcal{L}(K_C)$.
The divisor $D_{\rm chab}$ on $\bar{\cC}^{\rm sm}$ defines in a natural way a divisor ${\mathcal D}$ of degree $d$ on the metrized complex ${\fC}$ associated to $\cC$.
We can promote the divisor $K_{\fC} - {\mathcal D}$ to a limit linear series $(K_{\fC}- {\mathcal D}, \{ H_v \})$ by defining $H_v$ to be the reduction of $W$ to $C_v$ for each $v \in V(G)$.
By the definition of $D_{\rm chab}$, each element of $H_v$ vanishes to order at least $n_{\overline{Q}}$ at each point $\overline{Q}$ in ${\rm supp}(D_{\rm chab}) \cap C_v$.
The Specialization Theorem for limit linear series on metrized complexes then shows that
\[
r_{\fC}(K_{\fC} - {\mathcal D}) \geq s \geq g-r-1.
\]

\medskip

3. On the other hand, Clifford's inequality for metrized complexes implies that
\[
r_{\fC}(K_{\fC} - \mathcal{D}) \leq \frac{1}{2}(2g-2-d).
\]
Combining these inequalities gives $d \leq 2r$ as desired.

\subsection{The uniformity theorems of Katz--Rabinoff--Zureick-Brown}
\label{Section:KRZB}

Together with Rabinoff, Katz and Zureick-Brown have recently used linear series on tropical curves to refine another result due to Stoll.
In \cite{CHM97}, Caporaso, Harris, and Mazur proved that if one assumes the Bombieri--Lang conjecture then there is a uniform bound $M(g,K)$ depending only on $g$ and the number field $K$ such that $|C(K)| \leq M(g,K)$ for every curve $C$ of genus $g\geq 2$ over $K$.
The Bombieri--Lang conjecture, which asserts that the set of rational points on a variety of general type over a number field is not Zariski dense, remains wide open, and until recently little progress had been made in the direction of unconditional proofs of the Caporaso--Harris--Mazur result.
In \cite{Stoll13}, Stoll proved that a uniform bound $M(g,K)$ exists for {\em hyperelliptic curves} provided that one assumes in addition that the Mordell--Weil rank of
$\Jac(C)(K)$ is at most $g-3$.  Katz, Rabinoff, and Zureick-Brown succeeded in removing the hypothesis in Stoll's theorem that $C$ is hyperelliptic, obtaining the following result.

\begin{theorem}[\cite{KRZB15}]
\label{thm:KRZB}
There is an explicit bound $N(g,d)$ such that if $C$ is a curve of genus $g\geq 3$ defined over a number field $K$ of degree $d$ over $\QQ$ and having Mordell-Weil rank $r \leq g-3$, then
$$ \# C(K) \leq N(g,d).$$
When $K=\QQ$, one can take $N(g,1)=76g^2 -82g + 22$.
\end{theorem}

Note that the bound in Theorem~\ref{Thm:ChabautyColemanBadReduction} is not uniform, because the
quantity $\vert \mathcal{C}^{\rm sm}( \mathbb{F}_p ) \vert$ can be arbitrarily large for a given prime $p$ of bad reduction, and the smallest prime $p$ of good
reduction can be arbitrarily large as a function of $g$.
Stoll's main new idea was to apply the Chabauty-Coleman method on residue {\em annuli} instead of just on discs.
Stoll's proof exploits the concrete description of differentials on a hyperelliptic curve as $f(x)dx/y$; the restriction of such a differential to an annulus has a bounded numerator, and Stoll is able to analyze the zeroes of the resulting $p$-adic integral via explicit computations with Newton polygons.

\medskip

For general curves, such an explicit description of differentials and the Newton polygons of their $p$-adic integrals is not possible.  This is where the theory of linear systems on metric graphs becomes useful.
To circumvent the difficulty posed by not having an explicit description of differentials on $C$, Katz, Rabinoff, and Zureick-Brown generalize the Slope Formula (Theorem~\ref{thm:SlopeFormula}) to sections of a metrized line bundle.
For a differential $\omega$, the associated tropical function $F = \log|\omega|$ on the skeleton $\Gamma$ of $C$ belongs to the space $R(K_\Gamma^\#)$ of tropical rational functions $G$ with $K_\Gamma^\# + {\rm div}(G) \geq 0$.
(The absolute value here comes from a natural formal metric on the canonical bundle.)
Belonging to $R(K_\Gamma^\#)$ gives strong constraints on the slopes of $F$, and hence on the number of zeroes of the $p$-adic integral of $\omega$.
The Slope Formula thus replaces the Newton polygons in Stoll's arguments, and estimates on the slopes of the Newton polygon are replaced by
properties of the tropical linear series $\vert K_\Gamma^\# \vert$.

\medskip

A major issue one faces in trying to establish Theorem~\ref{thm:KRZB} (which also shows up in the earlier work of Stoll) is that when $C$ has bad reduction at $p$, there are two different kinds of $p$-adic integrals which need to be considered.
On the one hand, there are the $p$-adic Abelian integrals studied by Colmez, Zarhin, and Vologodsky, which have no periods and are obtained by pulling back the logarithm map on the $p$-adic Lie group ${\rm Jac}(C)(\QQ_p)$ to $C$.
These are the integrals for which one knows that ${\rm dim} (V_{\rm chab}) \geq g-r$.
On the other hand, there are the $p$-adic integrals of Berkovich and Coleman--de Shalit which do have periods but also have better functoriality properties.  These are the integrals which are given locally on residue annuli of a semistable model ${\mathcal C}$ by formally integrating a local Laurent series expansion of $\omega \in H^0(C,\Omega^1)$.
In order to prove Theorem~\ref{thm:KRZB}, one needs to study the difference between the two kinds of $p$-adic integrals.
One of the new discoveries of Katz, Rabinoff, and Zureick-Brown is that the difference can be understood quite concretely using tropical geometry by combining Theorem~\ref{Thm:JacobianSkeleton} with Raynaud's uniformization theory.

\medskip

The methods used by Katz--Rabinoff--Zureick-Brown in \cite{KRZB15} also provide new results in the direction of a ``uniform Manin-Mumford conjecture''.  The Manin--Mumford conjecture, proved by Raynaud, asserts that if $C$ is a curve of genus at least 2 embedded in its Jacobian via an Abel-Jacobi map
$\iota : C \to {\rm Jac}(C)$, then $\iota(C) \cap {\rm Jac}(C)(\overline{K})_{\rm tors}$ is finite.  One can ask whether there is a uniform bound on the size of this intersection as one varies over all curves of a fixed genus $g$.   The following uniform result for the number
of {\em $K$-rational points} on $C$ which are torsion on $J$ is proved in \cite{KRZB15}:

\begin{theorem}
\label{thm:KRZBrationaltorsion}
There is an explicit bound $N(g,d)_{\rm tors}$ (which one can equal to the bound $N(g,d)$ above) such that if $C$ is a curve of genus $g\geq 3$ defined over a number field $K$ of degree $d$ over $\QQ$ and $\iota : C \to {\rm Jac}(C)$ is an Abel-Jacobi embedding defined over $K$, then
$$ \# \iota(C) \cap {\rm Jac}(C)(K)_{\rm tors} \leq N(g,d)_{\rm tors}. $$
\end{theorem}

Note that in Theorem~\ref{thm:KRZBrationaltorsion} there is no restriction on the Mordell--Weil rank of ${\rm Jac}(C)$.
It has been conjectured that $\# A(K)_{\rm tors}$ is bounded uniformly in terms of $[K:\QQ]$ and $g$ for all abelian varieties of dimension $g$ over $K$, which would of course imply Theorem~\ref{thm:KRZBrationaltorsion} as a special case, but this is known only for $g=1$ \cite{Merel96} and
the general case seems far out of reach at present.

\medskip

Katz, Rabinoff, and Zureick-Brown also prove a uniformity result concerning the number of {\em geometric} ($\overline{K}$-rational) torsion points lying on $C$, under a technical assumption about the structure of the stable model at some prime $\mathfrak{p}$.
We refer to \cite{KRZB15} for the precise statement.

\section{Limiting Behavior of Weierstrass Points in Degenerating Families}
\label{Section:Weierstrass}

The theory discussed in this paper has interesting applications to the behavior of Weierstrass points under specialization.
To motivate this kind of question, we begin with a seemingly unrelated classical result due to Andrew Ogg \cite{Ogg78}.

\subsection{Weierstrass points on modular curves}

Let $N$ be a positive integer. The finite-dimensional space $S = S_2(\Gamma_0(N))$ of weight $2$ cusp forms for the congruence subgroup $\Gamma_0(N)$ of ${\rm SL}_2(\ZZ)$ is an important object in number theory.
An element $f \in S$ has a {\em $q$-expansion} of the form $f = \sum_{n=1}^\infty a_n q^n$ with $a_n \in \CC$, which uniquely determines $f$.
For $f \neq 0$ in $S$, define
\begin{equation}
\label{eq:ord}
{\rm ord}(f) = \inf \{ n \; \vert \; a_n \neq 0 \}- 1.
\end{equation}
If $g = g_0(N) = {\rm dim}(S)$, then by Gaussian elimination there exists an element $f \in S$ with ${\rm ord}(f) \geq g-1$.  Is there any unexpected cancellation?  Under certain restrictions on the level $N$, the answer is no:

\begin{theorem}[\cite{Ogg78}]
\label{theorem:Ogg}
If $N=pM$ with $p$ prime, $p \nmid M$, and $g_0(M)=0$ then there is no nonzero element $f$ of $S_2(\Gamma_0(N))$ with ${\rm ord}(f) \geq g$.  (In particular, this holds if $N=p$ is prime.)
\end{theorem}

One can give an enlightening proof of Ogg's theorem using specialization of divisors from curves to metric graphs; the following argument is taken from \cite{Baker08}.

\medskip

First of all, Ogg's theorem can be recast in the following purely geometric way, which is in fact how Ogg formulated and proved the result in \cite{Ogg78}:

\begin{theorem}
\label{theorem:Oggbis}
If $N=pM$ with $p$ prime, $p \nmid M$, and $g_0(M)=0$ then the cusp $\infty$ is not a Weierstrass point on the modular curve $X_0(N)$.
\end{theorem}

Recall that a point $P$ on a genus $g$ curve $X$ is called a {\em Weierstrass point} if there exists a holomorphic differential $\omega \in H^0(X,\Omega^1_X)$ vanishing to order at least $g$ at $P$.
To see the equivalence between Theorems~\ref{theorem:Ogg} and \ref{theorem:Oggbis}, recall that $q$ is an analytic local parameter on $X_0(N)$ at the cusp $\infty$ and the map $f \mapsto f(q)\frac{dq}{q}$ gives an isomorphism between
$S_2(\Gamma_0(N))$ and the space of holomorphic differentials on $X_0(N)$.  Under this isomorphism, the function ${\rm ord}$ defined in (\ref{eq:ord}) becomes the order of vanishing of the corresponding differential at $\infty$.  So there is a nonzero element $f$ of $S_2(\Gamma_0(N))$ with ${\rm ord}(f) \geq g$ if and only if
there is a nonzero holomorphic differential $\omega$ vanishing to order at least $g$ at $\infty$.

\medskip

The reduction of $X = X_0(N)$ modulo $p$ when $p$ exactly divides $N=pM$ is well-understood; the special fiber of the so-called {\em Deligne-Rapoport model} for $X_0(N)$ over $\ZZ_p$ consists of two copies of $X_0(M)$ intersecting transversely at the supersingular points in characteristic $p$.
This model is always semistable but is not in general regular.  (It is very easy to describe the minimal regular model, but we will not need this here.)
In any case, the skeleton $\Gamma$ of $X_0(N)$ over $\QQ_p$ is a ``metric banana graph'' consisting of two vertices connected by a number of edges, as pictured in Figure \ref{Fig:Banana}, and the cusp $\infty$ specializes to one of the two vertices, call it $P$.
Under the hypotheses of Theorem~\ref{theorem:Ogg}, each $X_0(M)$ is a rational curve and so the genus of $\Gamma$ is equal to the genus of $X_0(N)$.  That is, there are $g+1$ edges.
By the Specialization Theorem, if there is a nonzero global section of $K_X$ vanishing to order at least $g$ at $\infty$, then $r(K_\Gamma - gP) \geq 0$.
However, since $K_\Gamma = (g-1)P + (g-1)Q$, where $Q$ is the other vertex, we have $K_\Gamma - gP = (g-1)Q - P$, which is $P$-reduced by Dhar's algorithm and non-effective.  Therefore $r(K_\Gamma - gP) = -1$, and Ogg's theorem is proved.

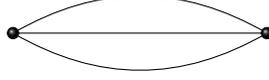
\begin{figure}[h!]
\begin{tikzpicture}[scale=1.5]

\coordinate (A) at (-2,0);
\coordinate (B) at (-4.25,0);


\path (A) edge [bend left] (B);
\path (A) edge [bend right] (B);
\path (A) edge (B);

\draw[ball color = black] (A) circle (0.5mm);
\draw[ball color = black] (B) circle (0.5mm);
\end{tikzpicture}
\caption{The ``Banana'' Graph of Genus 2}
\label{Fig:Banana}
\end{figure}

\medskip
\subsection{Specialization of Weierstrass points}

The essence of the above argument is that if $C$ is a {\em totally degenerate} curve, meaning that the genus of its minimal skeleton $\Gamma$ equals the genus of $C$, then the Weierstrass points on $C$ must specialize to Weierstrass points on $\Gamma$, where
a Weierstrass point on $\Gamma$ is a point $P$ such that $r(K_\Gamma - gP) \geq 0$.
It follows from the Specialization Theorem and the corresponding fact from algebraic geometry that if $\Gamma$ is a metric graph of genus $g \geq 2$ then the set of Weierstrass points on $\Gamma$ is non-empty.  A purely combinatorial proof of this fact was given by Amini \cite{Amini13}.

\medskip

The specialization of Weierstrass points is also a natural thing to study from the purely algebro-geometric point of view, where one is asking about the limiting behavior of the Weierstrass points in a semistable one-parameter family of curves.
This subject, which was previously studied by Eisenbud-Harris \cite{EisenbudHarris87}, Esteves-Medeiros \cite{EstevesMedeiros00}, and several other authors, has seen important recent advances by Amini \cite{Amini14}.
We now summarize the main results proved in Amini's paper.

\medskip

Let $L$ be a line bundle of degree $d$ and rank $r \geq 0$ on a curve $C$ of genus $g$ over an algebraically closed field $k$ of characteristic zero.
Given a point $P \in C(k)$, we define the {\em vanishing set} $S_P(L)$ of $L$ at $P$ to be the set of orders of vanishing of global sections of $L$ at $P$.
We have $|S_P(L)| = r+1$ for all $P \in C(k)$, and for all but finitely many $P \in C(k)$ the vanishing set is $[r] := \{ 0,1,\ldots,r \}$.
A point $P \in C(k)$ whose vanishing set is not $[r]$ is called a {\em Weierstrass point} for $L$.  Equivalently, $P$ is a Weierstrass point for $L$ if there exists a global section of $L$ vanishing to order at least $r+1$ at $P$.
A Weierstrass point of $C$ is by definition a Weierstrass point for the canonical bundle $K_C$.

\medskip

The {\em $L$-weight} of a point $P \in C(k)$ is
\[
{\rm wt}_P(L) = \left( \sum_{m \in S_P(L)} m \right) - {{r+1}\choose{2}} = \sum_{m \in S_P(L)} m - \sum_{i \in [r]} i.
\]
Thus ${\rm wt}_P(L) \geq 0$ for all $P \in C(k)$ and ${\rm wt}_P(L) > 0$ if and only if $P$ is a Weierstrass point for $L$.
The {\em Weierstrass divisor} for $L$ is $\cW = \cW(L) = \sum_{P \in C(k)} {\rm wt}_P(L) (P)$.
If we fix a basis $\cF$ for $H^0(C,L)$, the corresponding {\em Wronskian} ${\rm Wr}_{\cF}$ is a nonzero global section of
$L^{\otimes (r+1)} \otimes K_C^{\otimes \frac{r(r+1)}{2}}$ whose divisor is precisely $\cW(L)$.
In particular, the degree of $\cW(L)$ (i.e., the total number of Weierstrass points counted according to their weights) is $W(L) := d(r+1) + (g-1)r(r+1)$.

\medskip

We seek an explicit formula for $\Trop (\cW)$.
For this, it is convenient to fix a divisor with $L = L(D)$, and to define as usual $\cL(D) = \{ f \in k(C)^* \; | \; {\rm div}(f) + D \geq 0 \}$.
Let $D_\Gamma = \sum_{x \in \Gamma} d_x (x)$ be the specialization $\Trop (D)$ of $D$ to $\Gamma$.
Let $K^\#_\Gamma$ be the canonical divisor of $\Gamma$ considered as a vertex-weighted metric graph, as in \S\ref{Subsection:vertexweight}.
Concretely, we have $K^\#_\Gamma = \sum_{x \in \Gamma} \left( 2g_x - 2 + {\rm val}(x) \right) x$.

\medskip

For a tangent direction $\nu$ at $x$, define $S^\nu(D)$ to be the set of integers occurring as $s^\nu(f)$ for some $f \in \cL(D)$,
where $s^\nu(f)$ is defined as in \S\ref{Section:Berkovich} to be the slope of ${\rm trop}(f)$ in the tangent direction $\nu$.
Since $s^\nu(f)$ coincides with the order of vanishing of the normalized reduction $\bar{f}_x$ at the point of $C_x$ corresponding to $\nu$,
one sees easily that $|S^\nu(D)| = r+1$.

\medskip

For $x \in \Gamma$, let
\[
S_x(D) =  \left\{ \begin{array}{ll} \sum_{\nu \in T_x(\Gamma)} \left( \sum_{s \in S^\nu(D)} s \right) & \textrm{if $x$ is of type-2} \\
0 & \textrm{otherwise,} \\
\end{array} \right.
\]
where $T_x(\Gamma)$ denotes the set of tangent directions at $x$ in $\Gamma$, and let
\[
S(D) = \sum_{x \in \Gamma} S_x(D) \, x.
\]

Note that ${\rm deg}(S(D)) = 0$, since if  $f \in k(C)^*$ then the slope of $F = -\log|f|$ along an oriented edge $\vec{e}$ of $\Gamma$ is the negative of the slope of $F$ along the same edge with the orientation reversed.

\medskip

The following formula is due to Amini.  When $\Gamma$ is the skeleton of a semistable $R$-model $\cC$ for $C$, the formula shows how the Weierstrass points of the generic fiber $C$ specialize to the various components of the special fiber of $\cC$, providing a simple and satisfying answer to a question of
Eisenbud and Harris.

\begin{theorem}[\cite{Amini13}]
Let $\Trop : \Div(C_{\bar{K}}) \to \Div(\Gamma)$ be the natural map.  Then
\begin{equation}
\label{eq:Amini_Formula}
\Trop (\cW(L)) = (r+1) \Trop (D) + {{r+1}\choose{2}} K^\#_\Gamma - S(D).
\end{equation}
\end{theorem}

Note that since ${\rm deg}(S(D)) = 0$, the degree of the right-hand side of (\ref{eq:Amini_Formula}) is $W(L) = {\rm deg}(\cW(L))$ as expected.
Amini also proves an analogue of (\ref{eq:Amini_Formula}) when the residue field of $k$ has positive characteristic.  As this is more technical to state, we will not discuss this here.

\begin{remark}
A metric graph can have infinitely many Weierstrass points; this happens, for example, with the banana graphs of genus $g \geq 3$ discussed above (see \cite{Baker08}).
In general, the set of Weierstrass points on a metric graph $\Gamma$ is a finite disjoint union of closed connected sets.  It is an open problem to determine whether there are intrinsic multiplicities $m(A)$ attached to each connected component $A$ of the Weierstrass locus on a metric graph $\Gamma$ such
that for any curve $C$ having $\Gamma$ as a skeleton, exactly $m(A)$ Weierstrass points of $C$ tropicalize to $A$.
\end{remark}

\subsection{Distribution of Weierstrass points}

Amini uses formula (\ref{eq:Amini_Formula}) to prove a non-Archimedean analogue of the Mumford-Neeman equidistribution theorem, previously conjectured by Baker. We first recall the statement of the latter result, and then present Amini's analogous theorem.

\medskip

Let $C$ be a compact Riemann surface of genus at least 1.  There is a natural volume form $\omega_{\rm Ar}$ on $C$, called the {\em Arakelov form}, which can be defined as follows.
Let $\omega_1,\ldots,\omega_g$ be a orthonormal basis of $\cL (K_C)$ with respect to the Hermitian inner product
\[
\langle \omega, \nu \rangle = \frac{i}{2} \int_C \omega \wedge \bar{\nu}.
\]
Then the $(1,1)$-form $ \omega_C= \frac{i}{2} \sum_{j=1}^g \omega_j \wedge \bar{\omega}_j$
does not depend on the choice of $\omega_1,\ldots,\omega_g$ and has total mass $g$.  We define
\[
\omega_{\rm Ar} := \frac{1}{g} \omega_C.
\]

Geometrically, the curvature form of $\omega_C$ is the pullback of the curvature form of the flat metric on the Jacobian $J$ of $C$ with respect to any Abel-Jacobi map $C \to J$.  Since the flat metric on $J$ is translation-invariant, the pullback in question is independent of the choice of base point in the definition of the Abel-Jacobi map.

\medskip

The Mumford-Neeman theorem \cite{Neeman84} asserts that for any ample line bundle $L$ on $C$, the Weierstrass points of $L^{\otimes n}$ become equidistributed with respect to $\omega_{\rm Ar}$ as $n$ tends to infinity:

\begin{theorem}
\label{thm:Mumford-Neeman}
Let $C$ be a compact Riemann surface of genus at least 1 and let $L$ be an ample line bundle on $C$.
Let
\[
\delta_n = \frac{1}{W(L^{\otimes n})} \sum_{P \in C} {\rm wt}_P(L^{\otimes n}) \delta_P
\]
be the probability measure supported equally on the Weierstrass points of $L^{\otimes n}$.
Then as $n$ tends to infinity, the measures $\delta_n$ converge weakly\footnote{This means that for every continuous function $f : C \to \RR$, we have $\int_C f \, \delta_n = \int_C f \, \omega_{\rm Ar}$.} to the Arakelov metric $\omega_{\rm Ar}$.
\end{theorem}

In order to state Amini's non-Archimedean analogue of Theorem~\ref{thm:Mumford-Neeman} , we will first define the analogue of the Zhang measure on vertex-weighted metric graphs / Berkovich curves following \cite{BakerFaber11}.

\medskip

Let $\Gamma$ be a metric graph of genus $g$.   We fix a weighted graph model $G$ of $\Gamma$ and for each edge $e$ of $G$ let $\ell(e)$ denote the length of $e$.
For each spanning tree $T$ of $G$, let $e_1,\ldots,e_g$ denote the edges of $G$ not belonging to $T$, and let
\[
\mu_T = \sum_{j=1}^g \lambda(e_j)
\]
where $\lambda(e)$ is Lebesgue measure along $e$, normalized to have total mass $1$ (so that $\mu_T$ has total mass $g$).
We also let $w(T) = \prod_{j=1}^g \ell(e_j)$, and let
\[
w(G) = \sum_T w(T)
\]
be the sum of $w(T)$ over all spanning trees $T$ of $G$.
Then the measure
\[
\mu_\Gamma = \sum_{T} \frac{w(T)}{w(G)} \mu_T
\]
is a weighted average of the measures $\mu_T$ over all spanning trees $T$, and in particular has total mass $g$.

In other words, a random point in the complement of a random spanning tree of $G$ is distributed according to the probability measure $\frac{1}{g} \mu_\Gamma$.

\medskip

Now let $(\Gamma,\omega)$ be a vertex-weighted metric graph, in the sense of \S\ref{Subsection:vertexweight}, of genus $g = g(\Gamma) + \sum_{x \in \Gamma} \omega(x)$.
Then the measure
\[
\mu_{(\Gamma,\omega)} := \mu_\Gamma +  \sum_{x \in \Gamma} \omega_x  \delta_x
\]
has total mass $g$.
If $g \geq 1$, we define the {\em Zhang measure} on $(\Gamma,\omega)$ to be the probability measure
\[
\mu_{\rm Zh} := \frac{1}{g} \mu_{(\Gamma,\omega)}.
\]

\begin{theorem}[\cite{Amini14}]
\label{thm:Amini_Equidistribution}
Let $C$ be an algebraic curve of genus at least 1 over the non-Archimedean field $k$ of equal characteristic 0, and let $L$ be an ample line bundle on $C$.
Let ${\mathcal C}$ be a strongly semistable model of $C$ over the valuation ring of $k$, let $(\Gamma,\omega)$ be the weighted graph associated to ${\mathcal C}$ in the sense of \S\ref{Subsection:vertexweight},
and let $\mu_{\rm Zh}$ be the Zhang measure associated to $(\Gamma,\omega)$.
Finally, let
\[
\delta_n = \frac{1}{W(L^{\otimes n})}  \sum_{P \in C} {\rm wt}_P(L^{\otimes n}) \delta_{{\rm Trop}(P)}
\]
be the probability measure on $\Gamma$ supported equally on the tropicalizations of the Weierstrass points of $L^{\otimes n}$ (taken with multiplicities).
Then as $n$ tends to infinity, the measures $\delta_n$ converge weakly on $\Gamma$ to $\mu_{\rm Zh}$.
\end{theorem}

\medskip

A new and concrete consequence of Theorem~\ref{thm:Amini_Equidistribution} is the following:

\begin{corollary}
Let $C$ be an algebraic curve of genus at least 1 over a non-Archimedean field $k$ of equal characteristic 0, and let $L$ be an ample line bundle on $C$.  Fix a strongly semistable model $\cC$ for $C$ over the valuation ring of $k$, let $Z$ be an irreducible component of the special fiber of $\cC$, and let $g_Z$ be the genus of $Z$.
Let $\cW_Z(L^{\otimes n})$ be the set of Weierstrass points of $L^{\otimes n}$ specializing to a nonsingular point of $Z$.  Then
\[
\lim_{n \to \infty} \frac{|\cW_Z(L^{\otimes n})|}{|\cW(L^{\otimes n})|} = \frac{g_Z}{g}.
\]
\end{corollary}

\medskip

It is convenient and enlightening to rephrase Theorem~\ref{thm:Amini_Equidistribution} in terms of the Berkovich analytic space $C^{\rm an}$.
If $\Gamma$ is any skeleton of $C^{\rm an}$ and $\omega$ is the corresponding weight function defined by $\omega(x) = g_x$,
we define the {\em Zhang measure} on $C^{\rm an}$ to be the probability measure
\[
\mu_{\rm Zh} := \frac{1}{g} \iota_* \mu_{(\Gamma,\omega)}
\]
with respect to the natural inclusion $\iota : \Gamma \to C^{\rm an}$.
This measure on $C^{\rm an}$ is easily seen to be independent of the choice of $\Gamma$, and has total mass equal to the genus $g$ of $C$.
Using the fact that $C^\an \cong \varprojlim \Gamma_{\mathcal C}$
(cf. Section~\ref{section:Berkcurves}), one deduces using standard results from real analysis that Theorem~\ref{thm:Amini_Equidistribution} is equivalent to the following reformulation, which more closely resembles Theorem~\ref{thm:Mumford-Neeman}:

\begin{theorem}[\cite{Amini14}]
\label{thm:Amini_Equidistribution_bis}
Let $C$ be an algebraic curve of genus at least 1 over the non-Archimedean field $k$ of equal characteristic 0, and let $L$ be an ample line bundle on $C$.
Let
\[
\delta_n = \frac{1}{W(L^{\otimes n})}  \sum_{P \in C} {\rm wt}_P(L^{\otimes n}) \delta_P
\]
be the probability measure on $C^{\rm an}$ supported equally on the Weierstrass points of $L^{\otimes n}$.
Then as $n$ tends to infinity, the measures $\delta_n$ converge weakly on $C^{\rm an}$ to the Zhang measure $\mu_{\rm Zh}$.
\end{theorem}

\begin{remark}
The measure $\mu_{\rm Zh}$, which was first introduced in the context of vertex-weighted metric graphs by Shouwu Zhang in \cite{Zhang93}, plays the role in the non-Archimedean setting of the Arakelov volume form.
Using a result of Heinz \cite{Heinz04} and the recent work of Chambert-Loir--Ducros \cite{CLD12} and Gubler--Kunnemann \cite{GublerKunnemann14} on non-Archimedean Arakelov theory, one can show that, as in the Archimedean case, $\mu_C$ is obtained by pulling back the curvature form of a canonical translation-invariant metric on $J$ via an Abel-Jacobi map.
There is also evidently a close connection between the measure $\mu_\Gamma$ and the polyhedral decomposition $\{ C_T \}$ of $\Pic^g(\Gamma)$ associated to $G$ (cf. \S\ref{sec:BreakDivisorSection}) which is deserving of further study.
\end{remark}

The proof of Theorem~\ref{thm:Amini_Equidistribution} (and its equivalent formulation Theorem~\ref{thm:Amini_Equidistribution_bis}) is based on formula (\ref{eq:Amini_Formula}) together with the theory of {\em Okounkov bodies}.
The rough idea is that fixing a type-2 point $x$ of $C^{\rm an}$ and a tangent direction $\nu$ at $x$, as well as a divisor D with $L = L(D)$, the rational numbers $\frac{1}{n} S^\nu(nD)$ defined above become equidistributed in a real interval of length $d = {\rm deg}(L)$ as $n \to \infty$.
Combining this ``local'' equidistribution result with (\ref{eq:Amini_Formula}) and  a careful analysis of the variation of the minimum slope along edges of $\Gamma$ gives the desired result.


\section{Further Reading}

There are many topics closely related to the contents of this paper which we have not had space to discuss.  Here is a brief and non-exhaustive list of some related topics and papers which we recommend to the interested reader:

\medskip

1. {\em Harmonic morphisms.}  In algebraic geometry, a base-point free linear series of rank $r$ on a curve $C$ is more or less the same thing as a morphism $C \to \PP^r$.  In tropical geometry, the situation is much more subtle, and no satisfactory analogue of this correspondence is known.
For $r=1$, there is a close relationship (although not a precise correspondence) between tropical $\g^1_d$'s on a metric graph $\Gamma$ and degree $d$ harmonic morphisms from $\Gamma$ to a metric tree.  The theory of harmonic morphisms of metric graphs and metrized complexes of curves is explored in detail
in the papers \cite{ABBR14a,ABBR14b,Chan13,LuoManjunath14}, among others.

\smallskip

2. {\em Spectral bounds for gonality.} In \cite{CFK13}, Cornelissen et. al. establish a spectral lower bound for the {\em stable gonality} (in the sense of harmonic morphisms) of a graph $G$ in terms of the smallest nonzero eigenvalue of the Laplacian of $G$.  This is a tropical analogue of the Li-Yau inequality for Riemann surfaces.
They give applications of their tropical Li-Yau inequality to uniform boundedness of torsion points on rank two Drinfeld modules, as well as to lower bounds on the modular degree of elliptic curves over function fields.  The spectral bound from \cite{CFK13} was subsequently refined by Amini and Kool in \cite{AK14}
to a spectral lower bound for the {\em divisorial gonality} (i.e., the minimal degree of a rank 1 divisor) of a metric graph $\Gamma$.  In \cite{AK14}, as well as in the related paper \cite{RandomGonality}, this circle of ideas is applied to show that the expected gonality of a random graph is asymptotic to the number of vertices.

\smallskip

3. {\em Tropical complexes.} In \cite{Cartwright13}, Cartwright formulates a higher-dimensional analogue of the basic theory of linear series on graphs, including a Specialization Theorem for the rank function.  He calls the objects on which his higher-dimensional linear series live {\em tropical complexes}.
A generalization of the Slope Formula to the context of non-Archimedean varieties and tropical complexes is proved in \cite{GRW14}.

\smallskip

4. {\em Abstract versus embedded tropical curves.}  In this paper we have dealt exclusively with linear series on abstract tropical curves (thought of as metric graphs) and have eschewed the more traditional perspective of tropical varieties as non-Archimedean amoebas associated to subvarieties of tori.  The two approaches are closely related, however: see for example \cite{BPR11, GRW14, CDMY14}.  The theory of linear series on abstract tropical curves has concrete consequences for embedded tropical curves, e.g. with respect to the theory of bitangents and theta characteristics as in \cite{BLMPR14, CJ15}.

\smallskip

5. {\em Algebraic rank.}  In \cite{Caporaso13}, Caporaso introduces a notion of rank for divisors on graphs known as the \emph{algebraic rank}, which is defined geometrically by varying over all curves with the given dual graph and all line bundles with the given specialization.  The algebraic rank differs in general from the combinatorial rank \cite{CLM14}, but the two invariants agree for hyperelliptic graphs and graphs of genus 3 \cite{KY14b}.  Many of the results we have discussed also hold for the algebraic rank.  For example, there are specialization, Riemann-Roch, and Clifford's theorems for algebraic rank \cite{Caporaso13}, and Mn{\"e}v universality holds for obstructions to the lifting problem for algebraic rank \cite{Len14}.


\bibliography{math}

\newcommand{\etalchar}[1]{$^{#1}$}
\begin{thebibliography}{ABBR14b}

\bibitem[AB12]{AminiBaker12}
O.~Amini and M.~Baker.
\newblock Linear series on metrized complexes of algebraic curves.
\newblock preprint arXiv:1204.3508. To appear in {\em Mathematische Annalen},
  2012.

\bibitem[ABBR14a]{ABBR14a}
O.~Amini, M.~Baker, E.~Brugall\'e, and J.~Rabinoff.
\newblock Lifting harmonic morphisms {I}: metrized complexes and {B}erkovich
  skeleta.
\newblock arXiv:1303.481, To appear in Research in the Mathematical Sciences,
  2014.

\bibitem[ABBR14b]{ABBR14b}
O.~Amini, M.~Baker, E.~Brugall\'e, and J.~Rabinoff.
\newblock Lifting harmonic morphisms {II}: tropical curves and metrized
  complexes.
\newblock arXiv:1404.3390, To appear in Algebra \& Number Theory, 2014.

\bibitem[ABKS14]{ABKS13}
Y.~An, M.~Baker, G.~Kuperberg, and F.~Shokrieh.
\newblock Canonical representatives for divisor classes on tropical curves and
  the matrix-tree theorem.
\newblock {\em Forum Math. Sigma}, 2:e24, 25, 2014.

\bibitem[AC83]{ArbarelloCiliberto83}
E.~Arbarello and C.~Ciliberto.
\newblock Adjoint hypersurfaces to curves in {${\bf P}^{r}$} following {P}etri.
\newblock In {\em Commutative algebra ({T}rento, 1981)}, volume~84 of {\em
  Lecture Notes in Pure and Appl. Math.}, pages 1--21. Dekker, New York, 1983.

\bibitem[AC13]{AminiCaporaso13}
O.~Amini and L.~Caporaso.
\newblock Riemann--{R}och theory for weighted graphs and tropical curves.
\newblock {\em Adv. Math.}, 240:1--23, 2013.

\bibitem[ACGH85]{ACGH}
E.~Arbarello, M.~Cornalba, P.~A. Griffiths, and J.~Harris.
\newblock {\em Geometry of algebraic curves. {V}ol. {I}}, volume 267 of {\em
  Grundlehren der Mathematischen Wissenschaften}.
\newblock Springer-Verlag, New York, 1985.

\bibitem[ACP12]{acp}
D.~Abramovich, L.~Caporaso, and S.~Payne.
\newblock The tropicalization of the moduli space of curves.
\newblock preprint arXiv:1212.0373. To appear in {\em Ann. Sci. {\'E}c. Norm.
  Sup{\'e}r.}, 2012.

\bibitem[AK14]{AK14}
O.~Amini and F.~Kool.
\newblock A spectral lower bound for the divisorial gonality of metric graphs.
\newblock preprint arXiv:1407.5614, 2014.

\bibitem[Ami13]{Amini13}
O.~Amini.
\newblock Reduced divisors and embeddings of tropical curves.
\newblock {\em Trans. Amer. Math. Soc.}, 365(9):4851--4880, 2013.

\bibitem[Ami14]{Amini14}
O.~Amini.
\newblock Equidistribution of {W}eierstrass points on curves over
  non-{A}rchimedean fields.
\newblock preprint arXiv:1412.0926v1, 2014.

\bibitem[AMSW13]{AMSW}
R.~Agrawal, G.~Musiker, V.~Sotirov, and F.~Wei.
\newblock Involutions on standard {Y}oung tableaux and divisors on metric
  graphs.
\newblock {\em Electron. J. Combin.}, 20(3), 2013.

\bibitem[Bab39]{Babbage39}
D.~W. Babbage.
\newblock A note on the quadrics through a canonical curve.
\newblock {\em J. London Math. Soc.}, 14:310--315, 1939.

\bibitem[Bac14a]{Backman14a}
S.~Backman.
\newblock Infinite reduction of divisors on metric graphs.
\newblock {\em European J. Combin.}, 35:67--74, 2014.

\bibitem[Bac14b]{Backman14b}
S.~Backman.
\newblock {R}iemann-{R}och theory for graph orientations.
\newblock preprint arXiv:1401.3309, 2014.

\bibitem[Bak08a]{Baker08b}
M.~Baker.
\newblock An introduction to {B}erkovich analytic spaces and non-{A}rchimedean
  potential theory on curves.
\newblock In {\em {$p$}-adic geometry}, volume~45 of {\em Univ. Lecture Ser.},
  pages 123--174. Amer. Math. Soc., Providence, RI, 2008.

\bibitem[Bak08b]{Baker08}
M.~Baker.
\newblock Specialization of linear systems from curves to graphs.
\newblock {\em Algebra Number Theory}, 2(6):613--653, 2008.

\bibitem[BE85]{BE87}
E.~Ballico and Ph. Ellia.
\newblock The maximal rank conjecture for nonspecial curves in {${\bf P}^3$}.
\newblock {\em Invent. Math.}, 79(3):541--555, 1985.

\bibitem[BE91]{BayerEisenbud91}
D.~Bayer and D.~Eisenbud.
\newblock Graph curves.
\newblock {\em Adv. Math.}, 86(1):1--40, 1991.
\newblock With an appendix by Sung Won Park.

\bibitem[Ber90]{Berkovich90}
V.~Berkovich.
\newblock {\em Spectral theory and analytic geometry over non-{A}rchimedean
  fields}, volume~33 of {\em Mathematical Surveys and Monographs}.
\newblock American Mathematical Society, Providence, RI, 1990.

\bibitem[BF11]{BakerFaber11}
M.~Baker and X.~Faber.
\newblock Metric properties of the tropical {A}bel-{J}acobi map.
\newblock {\em J. Algebraic Combin.}, 33(3):349--381, 2011.

\bibitem[Big99]{Biggs99}
N.~L. Biggs.
\newblock Chip-firing and the critical group of a graph.
\newblock {\em J. Algebraic Combin.}, 9(1):25--45, 1999.

\bibitem[BLM{\etalchar{+}}14]{BLMPR14}
M.~Baker, Y.~Len, R.~Morrison, N.~Pflueger, and Q.~Ren.
\newblock Bitangents of tropical plane quartic curves.
\newblock preprint arXiv:1404.7568, 2014.

\bibitem[BMV11]{BrannettiMeloViviani11}
S.~Brannetti, M.~Melo, and F.~Viviani.
\newblock On the tropical {T}orelli map.
\newblock {\em Adv. Math.}, 226(3):2546--2586, 2011.

\bibitem[BN07]{BakerNorine07}
M.~Baker and S.~Norine.
\newblock {R}iemann-{R}och and {A}bel-{J}acobi theory on a finite graph.
\newblock {\em Adv. Math.}, 215(2):766--788, 2007.

\bibitem[BPR11]{BPR11}
M.~Baker, S.~Payne, and J.~Rabinoff.
\newblock Nonarchimedean geometry, tropicalization, and metrics on curves.
\newblock preprint arXiv:1104.0320v3, to appear in {\em Algebraic Geometry},
  2011.

\bibitem[BPR13]{BPR13}
M.~Baker, S.~Payne, and J.~Rabinoff.
\newblock On the structure of non-{A}rchimedean analytic curves.
\newblock In {\em Tropical and non-{A}rchimedean geometry}, volume 605 of {\em
  Contemp. Math.}, pages 93--121. Amer. Math. Soc., Providence, RI, 2013.

\bibitem[BR10]{BakerRumely10}
M.~Baker and R.~Rumely.
\newblock {\em Potential theory and dynamics on the {B}erkovich projective
  line}, volume 159 of {\em Mathematical Surveys and Monographs}.
\newblock American Mathematical Society, Providence, RI, 2010.

\bibitem[BR13]{BakerRabinoff13}
M.~Baker and J.~Rabinoff.
\newblock The skeleton of the {J}acobian, the {J}acobian of the skeleton, and
  lifting meromorphic functions from tropical to algebraic curves.
\newblock preprint arXiv:1308.3864. To appear in {\em Intl. Math. Res. Not.},
  2013.

\bibitem[BS13]{BS13}
M.~Baker and F.~Shokrieh.
\newblock Chip-firing games, potential theory on graphs, and spanning trees.
\newblock {\em Journal of Combinatorial Theory, Series A}, 120(1):164--182,
  2013.

\bibitem[BW14]{BakerWang14}
M.~Baker and Y.~Wang.
\newblock The {B}ernardi process and torsor structures on spanning trees.
\newblock preprint arXiv:1406.1584, 2014.

\bibitem[Cap08]{Caporaso08}
L.~Caporaso.
\newblock N\'eron models and compactified {P}icard schemes over the moduli
  stack of stable curves.
\newblock {\em Amer. J. Math.}, 130(1):1--47, 2008.

\bibitem[Cap11]{Caporaso11}
L.~Caporaso.
\newblock Algebraic and tropical curves: comparing their moduli spaces.
\newblock To appear in Handbook of Moduli, edited by Farkas and Morrison.
  arXiv:1101.4821v2, 2011.

\bibitem[Cap12]{Caporaso12}
L.~Caporaso.
\newblock Algebraic and combinatorial {B}rill-{N}oether theory.
\newblock In {\em Compact moduli spaces and vector bundles}, volume 564 of {\em
  Contemp. Math.}, pages 69--85. Amer. Math. Soc., Providence, RI, 2012.

\bibitem[Cap13]{Caporaso13}
L.~Caporaso.
\newblock Rank of divisors on graphs: an algebro-geometric analysis.
\newblock In {\em A celebration of algebraic geometry}, volume~18 of {\em Clay
  Math. Proc.}, pages 45--64. Amer. Math. Soc., Providence, RI, 2013.

\bibitem[Car13]{Cartwright13}
D.~Cartwright.
\newblock Tropical complexes.
\newblock preprint arXiv:1308.3813, 2013.

\bibitem[Car15]{Cartwright15}
D.~Cartwright.
\newblock Lifting matroid divisors on tropical curves.
\newblock preprint arXiv:1502.03759, 2015.

\bibitem[CDMY14]{CDMY14}
D.~Cartwright, A.~Dudzik, M.~Manjunath, and Y.~Yao.
\newblock Embeddings and immersions of tropical curves.
\newblock preprint arXiv:1409.7372v1, 2014.

\bibitem[CDPR12]{tropicalBN}
F.~Cools, J.~Draisma, S.~Payne, and E.~Robeva.
\newblock A tropical proof of the {B}rill-{N}oether theorem.
\newblock {\em Adv. Math.}, 230(2):759--776, 2012.

\bibitem[CES25]{CES25}
G.~Castelnuovo, F.~Enriques, and F.~Severi.
\newblock Max {N}oether.
\newblock {\em Math. Ann.}, 93(1):161--181, 1925.

\bibitem[CFK15]{CFK13}
G.~Cornelissen, K.~Fumiharu, and J.~Kool.
\newblock A combinatorial {L}i--{Y}au inequality and rational points on curves.
\newblock {\em Math. Ann.}, 361(1-2):211--258, 2015.

\bibitem[Cha13]{Chan13}
M.~Chan.
\newblock Tropical hyperelliptic curves.
\newblock {\em J. Algebraic Combin.}, 37(2):331--359, 2013.

\bibitem[CHM88]{CHM88}
C.~Ciliberto, J.~Harris, and R.~Miranda.
\newblock On the surjectivity of the {W}ahl map.
\newblock {\em Duke Math. J.}, 57(3):829--858, 1988.

\bibitem[CHM97]{CHM97}
L.~Caporaso, J.~Harris, and B.~Mazur.
\newblock Uniformity of rational points.
\newblock {\em J. Amer. Math. Soc.}, 10(1):1--35, 1997.

\bibitem[CJ15]{CJ15}
M.~Chan and P.~Jiradilok.
\newblock Theta characteristics of tropical ${K}_4$ curves.
\newblock preprint arXiv:1503:05776v1, 2015.

\bibitem[CJP14]{lifting}
D.~Cartwright, D.~Jensen, and S.~Payne.
\newblock Lifting divisors on a generic chain of loops.
\newblock preprint arXiv:1404.4001. To appear in {\em Canad. Math. Bull.},
  2014.

\bibitem[CLD12]{CLD12}
A.~Chambert-Loir and A.~Ducros.
\newblock Formes diff\'{e}rentielles r\'{e}elles et courants sur les espaces de
  {B}erkovich.
\newblock preprint arXiv:1204.6277, 2012.

\bibitem[CLM14]{CLM14}
L.~Caporaso, Y.~Len, and M.~Melo.
\newblock Algebraic and combinatorial rank of divisors on finite graphs.
\newblock preprint arXiv:1401:5730v2, 2014.

\bibitem[Col85]{Coleman85}
R.~F. Coleman.
\newblock Effective {C}habauty.
\newblock {\em Duke Math. J.}, 52(3):765--770, 1985.

\bibitem[Con08]{ConradAWS}
B.~Conrad.
\newblock Several approaches to non-{A}rchimedean geometry.
\newblock In {\em {$p$}-adic geometry}, volume~45 of {\em Univ. Lecture Ser.},
  pages 9--63. Amer. Math. Soc., Providence, RI, 2008.

\bibitem[Cop14a]{Coppens14}
M.~Coppens.
\newblock {C}lifford's theorem for graphs.
\newblock preprint arXiv:1304.6101, 2014.

\bibitem[Cop14b]{Coppens14b}
M.~Coppens.
\newblock Free divisor on metric graphs.
\newblock preprint arXiv:1410.3114, 2014.

\bibitem[CV10]{CaporasoViviani10}
L.~Caporaso and F.~Viviani.
\newblock Torelli theorem for graphs and tropical curves.
\newblock {\em Duke Math. J.}, 153(1):129--171, 2010.

\bibitem[Dha90]{Dhar90}
D.~Dhar.
\newblock Self-organized critical state of sandpile automaton models.
\newblock {\em Phys. Rev. Lett.}, 64(14):1613--1616, 1990.

\bibitem[DJKM14]{RandomGonality}
A.~Deveau, D.~Jensen, J.~Kainic, and D.~Mitropolsky.
\newblock Gonality of random graphs.
\newblock preprint arXiv:1409.5688, 2014.

\bibitem[DRSV95]{Dhar-et-al}
D.~Dhar, P.~Ruelle, S.~Sen, and D.-N. Verma.
\newblock Algebraic aspects of abelian sandpile models.
\newblock {\em J. Phys. A}, 28(4):805--831, 1995.

\bibitem[EH83]{EisenbudHarris83c}
D.~Eisenbud and J.~Harris.
\newblock A simpler proof of the {G}ieseker-{P}etri theorem on special
  divisors.
\newblock {\em Invent. Math.}, 74(2):269--280, 1983.

\bibitem[EH86]{EisenbudHarris86}
D.~Eisenbud and J.~Harris.
\newblock Limit linear series: basic theory.
\newblock {\em Invent. Math.}, 85(2):337--371, 1986.

\bibitem[EH87]{EisenbudHarris87}
D.~Eisenbud and J.~Harris.
\newblock Existence, decomposition, and limits of certain {W}eierstrass points.
\newblock {\em Invent. Math.}, 87(3):495--515, 1987.

\bibitem[EM00]{EstevesMedeiros00}
E.~Esteves and N.~Medeiros.
\newblock Limits of {W}eierstrass points in regular smoothings of curves with
  two components.
\newblock {\em C. R. Acad. Sci. Paris S\'er. I Math.}, 330(10):873--878, 2000.

\bibitem[Far09]{Farkas09}
G.~Farkas.
\newblock Koszul divisors on moduli spaces of curves.
\newblock {\em Amer. J. Math.}, 131(3):819--867, 2009.

\bibitem[FGP12]{limits}
T.~Foster, P.~Gross, and S.~Payne.
\newblock Limits of tropicalizations.
\newblock preprint arXiv:1211.2718. To appear in {\em Israel J. Math.}, 2012.

\bibitem[GH80]{GriffithsHarris80}
P.~Griffiths and J.~Harris.
\newblock On the variety of special linear systems on a general algebraic
  curve.
\newblock {\em Duke Math. J.}, 47(1):233--272, 1980.

\bibitem[Gie82]{Gieseker82}
D.~Gieseker.
\newblock Stable curves and special divisors: {P}etri's conjecture.
\newblock {\em Invent. Math.}, 66(2):251--275, 1982.

\bibitem[GK08]{GathmannKerber08}
A.~Gathmann and M.~Kerber.
\newblock A {R}iemann-{R}och theorem in tropical geometry.
\newblock {\em Math. Z.}, 259(1):217--230, 2008.

\bibitem[GK14]{GublerKunnemann14}
W.~Gubler and K.~K{\"u}nnemann.
\newblock A tropical approach to non-{A}rchimedean {A}rakelov theory.
\newblock preprint arXiv:1406.7637, 2014.

\bibitem[GKM09]{GathmannKerberMarkwig09}
A.~Gathmann, M.~Kerber, and H.~Markwig.
\newblock Tropical fans and the moduli spaces of tropical curves.
\newblock {\em Compos. Math.}, 145(1):173--195, 2009.

\bibitem[GRW14]{GRW14}
W.~Gubler, J.~Rabinoff, and A.~Werner.
\newblock Skeletons and tropicalizations.
\newblock preprint arXiv:1404.7044, 2014.

\bibitem[Gub07]{Gubler07}
W.~Gubler.
\newblock Tropical varieties for non-{A}rchimedean analytic spaces.
\newblock {\em Invent. Math.}, 169(2):321--376, 2007.

\bibitem[Gub10]{Gubler10}
W.~Gubler.
\newblock Non-{A}rchimedean canonical measures on abelian varieties.
\newblock {\em Compos. Math.}, 146(3):683--730, 2010.

\bibitem[Har82]{Harris82}
J.~Harris.
\newblock {\em Curves in projective space}, volume~85 of {\em S\'eminaire de
  Math\'ematiques Sup\'erieures}.
\newblock Presses de l'Universit\'e de Montr\'eal, Montreal, Que., 1982.
\newblock With the collaboration of David Eisenbud.

\bibitem[Hei04]{Heinz04}
N.~Heinz.
\newblock Admissible metrics for line bundles on curves and abelian varieties
  over non-{A}rchimedean local fields.
\newblock {\em Arch. Math. (Basel)}, 82(2):128--139, 2004.

\bibitem[HKN13]{HladkyKralNorine10}
J.~Hladk\'y, D.~Kr\'al', and S.~Norine.
\newblock Rank of divisors on tropical curves.
\newblock {\em J. Combin. Theory Ser. A}, 120(7):1521--1538, 2013.

\bibitem[HLM{\etalchar{+}}08]{HLMPPW08}
A.~Holroyd, L.~Levine, K.~M{\'e}sz{\'a}ros, Y.~Peres, J.~Propp, and D.~Wilson.
\newblock Chip-firing and rotor-routing on directed graphs.
\newblock In {\em In and out of equilibrium. 2}, volume~60 of {\em Progr.
  Probab.}, pages 331--364. Birkh\"auser, Basel, 2008.

\bibitem[HM98]{HarrisMorrison98}
J.~Harris and I.~Morrison.
\newblock {\em Moduli of curves}, volume 187 of {\em Graduate Texts in
  Mathematics}.
\newblock Springer-Verlag, New York, 1998.

\bibitem[HMY12]{HMY12}
C.~Haase, G.~Musiker, and J.~Yu.
\newblock Linear systems on tropical curves.
\newblock {\em Math. Z.}, 270(3-4):1111--1140, 2012.

\bibitem[Jen14]{Heawood}
D.~Jensen.
\newblock The locus of {B}rill-{N}oether general graphs is not dense.
\newblock preprint arXiv:1405.6338, 2014.

\bibitem[JP]{MRC}
D.~Jensen and S.~Payne.
\newblock Tropical independence {II}: The {M}aximal {R}ank {C}onjecture for
  quadrics.
\newblock in preparation.

\bibitem[JP14]{tropicalGP}
D.~Jensen and S.~Payne.
\newblock Tropical independence {I}: {S}hapes of divisors and a proof of the
  {G}ieseker-{P}etri theorem.
\newblock {\em Algebra Number Theory}, 8(9):2043--2066, 2014.

\bibitem[Kem71]{Kempf71}
G.~Kempf.
\newblock {\em Schubert methods with an application to algebraic curves}.
\newblock Publ. Math. Centrum, 1971.

\bibitem[KL74]{KleimanLaksov74}
S.~L. Kleiman and D.~Laksov.
\newblock Another proof of the existence of special divisors.
\newblock {\em Acta Math.}, 132:163--176, 1974.

\bibitem[Koz09]{Kozlov09}
D.~Kozlov.
\newblock The topology of moduli spaces of tropical curves with marked points.
\newblock {\em Asian J. Math.}, 13(3):385--403, 2009.

\bibitem[KRZB15]{KRZB15}
E.~Katz, J.~Rabinoff, and D.~Zureick-Brown.
\newblock Uniform bounds for the number of rational points on curves of small
  {M}ordell-{W}eil rank.
\newblock preprint arXiv:1504.00694, 2015.

\bibitem[KT14]{KissTothmeresz14}
V.~Kiss and L.~T{\'o}thm{\'e}r{\'e}sz.
\newblock Chip-firing games on {E}ulerian digraphs and {NP}-hardness of
  computing the rank of a divisor on a graph.
\newblock preprint arXiv:1407.6598, 2014.

\bibitem[KY14a]{KY14b}
S.~Kawaguchi and K.~Yamaki.
\newblock Algebraic rank of hyperelliptic graphs and graphs of genus 3.
\newblock preprint arXiv:1401:3935, 2014.

\bibitem[KY14b]{KawaguchiYamaki14}
S.~Kawaguchi and K.~Yamaki.
\newblock Rank of divisors on hyperelliptic curves and graphs under
  specialization.
\newblock preprint arXiv:1304.6979v3, 2014.

\bibitem[KZB13]{KZB13}
E.~Katz and D.~Zureick-Brown.
\newblock The {C}habauty-{C}oleman bound at a prime of bad reduction and
  {C}lifford bounds for geometric rank functions.
\newblock {\em Compos. Math.}, 149(11):1818--1838, 2013.

\bibitem[Laf03]{Lafforgue03}
L.~Lafforgue.
\newblock {\em Chirurgie des grassmanniennes}, volume~19 of {\em CRM Monograph
  Series}.
\newblock American Mathematical Society, Providence, RI, 2003.

\bibitem[Laz86]{Lazarsfeld86}
R.~Lazarsfeld.
\newblock Brill-{N}oether-{P}etri without degenerations.
\newblock {\em J. Differential Geom.}, 23(3):299--307, 1986.

\bibitem[Len12]{Len12}
Y.~Len.
\newblock The {B}rill-{N}oether rank of a tropical curve.
\newblock preprint arXiv:1209.6309v1, 2012.

\bibitem[Len14]{Len14}
Y.~Len.
\newblock A note on algebraic rank, matroids, and metrized complexes.
\newblock preprint arXiv:1401:8156, 2014.

\bibitem[Liu02]{LiuBook}
Q.~Liu.
\newblock {\em Algebraic geometry and arithmetic curves}, volume~6 of {\em
  Oxford Graduate Texts in Mathematics}.
\newblock Oxford University Press, Oxford, 2002.
\newblock Translated from the French by Reinie Ern{\'e}, Oxford Science
  Publications.

\bibitem[LM14]{LuoManjunath14}
Y.~Luo and M.~Manjunath.
\newblock Smoothing of limit linear series of rank one on saturated metrized
  complexes of algebraic curves.
\newblock preprint arXiv:1411.2325, 2014.

\bibitem[LP10]{LevinePropp}
L.~Levine and J.~Propp.
\newblock What is {$\dots$} a sandpile?
\newblock {\em Notices Amer. Math. Soc.}, 57(8):976--979, 2010.

\bibitem[LPP12]{LPP12}
C.-M. Lim, S.~Payne, and N.~Potashnik.
\newblock A note on {B}rill-{N}oether theory and rank-determining sets for
  metric graphs.
\newblock {\em Int. Math. Res. Not. IMRN}, (23):5484--5504, 2012.

\bibitem[LT02]{LorenziniTucker02}
D.~Lorenzini and T.~J. Tucker.
\newblock Thue equations and the method of {C}habauty-{C}oleman.
\newblock {\em Invent. Math.}, 148(1):47--77, 2002.

\bibitem[Luo11]{Luo11}
Y.~Luo.
\newblock Rank-determining sets of metric graphs.
\newblock {\em J. Combin. Theory Ser. A}, 118:1775--1793, 2011.

\bibitem[Mer96]{Merel96}
L.~Merel.
\newblock Bornes pour la torsion des courbes elliptiques sur les corps de
  nombres.
\newblock {\em Invent. Math.}, 124(1-3):437--449, 1996.

\bibitem[MP12]{McCallumPoonen12}
W.~McCallum and B.~Poonen.
\newblock The method of {C}habauty and {C}oleman.
\newblock In {\em Explicit methods in number theory}, volume~36 of {\em Panor.
  Synth\`eses}, pages 99--117. Soc. Math. France, Paris, 2012.

\bibitem[MS09]{MaclaganSturmfels}
D.~Maclagan and B.~Sturmfels.
\newblock Introduction to tropical geometry.
\newblock book in preparation, 2009.

\bibitem[Mum72]{Mumford72}
D.~Mumford.
\newblock An analytic construction of degenerating abelian varieties over
  complete rings.
\newblock {\em Compositio Math.}, 24:239--272, 1972.

\bibitem[MZ08]{MikhalkinZharkov08}
G.~Mikhalkin and I.~Zharkov.
\newblock Tropical curves, their {J}acobians and theta functions.
\newblock In {\em Curves and abelian varieties}, volume 465 of {\em Contemp.
  Math.}, pages 203--230. Amer. Math. Soc., Providence, RI, 2008.

\bibitem[Nee84]{Neeman84}
A.~Neeman.
\newblock The distribution of {W}eierstrass points on a compact {R}iemann
  surface.
\newblock {\em Ann. of Math. (2)}, 120(2):317--328, 1984.

\bibitem[Nic11]{Nicaise11}
J.~Nicaise.
\newblock Singular cohomology of the analytic {M}ilnor fiber, and mixed {H}odge
  structure on the nearby cohomology.
\newblock {\em J. Algebraic Geom.}, 20(2):199--237, 2011.

\bibitem[Noe82]{Noether82}
M.~Noether.
\newblock Zur {G}rundlegung der {T}heorie der algebraische raumcurven.
\newblock {\em J. Reine Angew. Math.}, 93:271--318, 1882.

\bibitem[Ogg78]{Ogg78}
A.~P. Ogg.
\newblock On the {W}eierstrass points of {$X_{0}(N)$}.
\newblock {\em Illinois J. Math.}, 22(1):31--35, 1978.

\bibitem[Oss14]{Osserman14}
B.~Osserman.
\newblock Limit linear series for curves not of compact type.
\newblock preprint arXiv:1406.6699, 2014.

\bibitem[Pay09]{analytification}
S.~Payne.
\newblock Analytification is the limit of all tropicalizations.
\newblock {\em Math. Res. Lett.}, 16(3):543--556, 2009.

\bibitem[Pay13]{PaynePersonal}
S.~Payne.
\newblock Personal communication, 2013.

\bibitem[PPW13]{Perkinson-et-al}
D.~Perkinson, J.~Perlman, and J.~Wilmes.
\newblock Primer for the algebraic geometry of sandpiles.
\newblock In {\em Tropical and non-{A}rchimedean geometry}, volume 605 of {\em
  Contemp. Math.}, pages 211--256. Amer. Math. Soc., Providence, RI, 2013.

\bibitem[Rab12]{Rabinoff12}
J.~Rabinoff.
\newblock Tropical analytic geometry, {N}ewton polygons, and tropical
  intersections.
\newblock {\em Adv. Math.}, 229(6):3192--3255, 2012.

\bibitem[Ray70]{Raynaud70}
M.~Raynaud.
\newblock Sp\'ecialisation du foncteur de {P}icard.
\newblock {\em Inst. Hautes \'Etudes Sci. Publ. Math.}, (38):27--76, 1970.

\bibitem[Sev15]{Severi15}
F.~Severi.
\newblock Sulla classificazione delle curve algebriche e sul teorema di
  esistenza di {R}iemann.
\newblock {\em Rend. R. Acc. Naz. Lincei}, 24(5):887--888, 1915.

\bibitem[Sto06]{Stoll06}
M.~Stoll.
\newblock Independence of rational points on twists of a given curve.
\newblock {\em Compos. Math.}, 142(5):1201--1214, 2006.

\bibitem[Sto13]{Stoll13}
M.~Stoll.
\newblock Uniform bounds for the number of rational points on hyperelliptic
  curves of small {M}ordell-{W}eil rank.
\newblock preprint arXiv:1307.1773v4, 2013.

\bibitem[Thu05]{ThuillierThesis}
A.~Thuillier.
\newblock {\em Th{\'e}orie du potentiel sur ler courbes en g{\'e}om{\'e}trie
  analytique non archim{\'e}dienne. {A}pplications {\`a} la the{\'e}orie
  d'{A}rakelov}.
\newblock PhD thesis, University of Rennes, 2005.

\bibitem[Thu07]{Thuillier07}
A.~Thuillier.
\newblock G\'eom\'etrie toro\"\i dale et g\'eom\'etrie analytique non
  archim\'edienne.
\newblock {\em Manuscripta Math.}, 123(4):381--451, 2007.

\bibitem[Tib03]{TiB03}
Montserrat Teixidor~i bigas.
\newblock Injectivity of the symmetric map for line bundles.
\newblock {\em Manuscripta Math.}, 112(4):511--517, 2003.

\bibitem[Viv13]{Viviani13}
F.~Viviani.
\newblock Tropicalizing vs. compactifying the {T}orelli morphism.
\newblock In {\em Tropical and non-{A}rchimedean geometry}, volume 605 of {\em
  Contemp. Math.}, pages 181--210. Amer. Math. Soc., Providence, RI, 2013.

\bibitem[Voi92]{Voisin92}
C.~Voisin.
\newblock Sur l'application de {W}ahl des courbes satisfaisant la condition de
  {B}rill-{N}oether-{P}etri.
\newblock {\em Acta Math.}, 168(3-4):249--272, 1992.

\bibitem[Zha93]{Zhang93}
S.~Zhang.
\newblock Admissible pairing on a curve.
\newblock {\em Invent. Math.}, 112(1):171--193, 1993.

\end{thebibliography}

\end{document}